\documentclass[11pt]{article}
\usepackage[latin1]{inputenc}
\usepackage{times,graphics,epsfig,bbm,dsfont}
\usepackage{amsthm,amsmath,amsfonts,amssymb,mathrsfs}
\usepackage{float}
\usepackage{subcaption}
\usepackage[margin=1in]{geometry}
\usepackage{xcolor}
\usepackage{appendix}

\usepackage[crop=off]{auto-pst-pdf}

\newcommand{\tblack}{\textcolor{black}}

\theoremstyle{definition}
\newtheorem{definition}{Definition}
\theoremstyle{remark}

\newtheoremstyle{mytheorem}{0.5cm}{0.2cm}{\slshape}{ }{\bfseries}{.}{ }{}
\theoremstyle{mytheorem}

\newtheorem{Thm}[definition]{Theorem}

\newtheorem{Lem}[definition]{Lemma}
\newtheorem{Res}[definition]{Result}

\title{A Computational Framework for Evaluating\\
	the Role of Mobility on the Propagation\\
	of Epidemics on Point Processes
}

\author{
	Fran\c{c}ois Baccelli\footnote{INRIA, France and UT Austin, USA. Email: baccelli@math.utexas.edu}\hspace{.2cm} and  Nithin Ramesan\footnote{UT Austin, USA. Email: nithinseyon@utexas.edu}
}

\begin{document}
\maketitle

\begin{abstract}
	This paper is focused on SIS (Susceptible-Infected-Susceptible) epidemic dynamics (also known as the contact process) on populations modelled by homogeneous Poisson point processes of the Euclidean plane, where the infection rate of a susceptible individual is proportional to the number of infected individuals in a disc around it. The main focus of the paper is a model where points are also subject to some random motion. Conservation equations for moment measures are leveraged to analyze the stationary regime of the point processes of infected and susceptible individuals. A heuristic factorization of the third moment measure is then proposed to obtain simple polynomial equations allowing one to derive closed form approximations for the fraction of infected individuals in the steady state. These polynomial equations also lead to a phase diagram which tentatively delineates the regions of the space of parameters (population density, infection radius, infection and recovery rate, and motion rate) where the epidemic survives and those where there is extinction. A key take-away from this phase diagram is that the extinction of the epidemic is not always aided by a decrease in the motion rate. These results are substantiated by simulations on large two dimensional tori. These simulations show that the polynomial equations accurately predict the fraction of infected individuals when the epidemic survives. The simulations also show that the proposed phase diagram accurately predicts the parameter regions where the mean survival time of the epidemic increases (resp. decreases) with motion rate.
\end{abstract}


\section{Introduction}
This paper is focused on the use of the theory of point processes for 
the analysis of the propagation of an epidemic and of the policies to control it,
in particular the reduction of individual motion (i.e., lockdown policies).

The general framework is that of a stochastic SIS (Susceptible, Infected, Susceptible) model
on point processes of the Euclidean plane. In SIS, each individual has a state which is either Susceptible
or Infected. It moves from $S$ to $I$ upon infection by another individual. From $I$, it moves to
$S$ after after some recovery time. Below we will use the terms SIS process and
contact process equivalently. The model considered here is stochastic in several distinct ways:
\begin{enumerate}
	\item The epidemic propagates on a random medium: the individuals are located at the points 
	of a random point process on $\mathbb R^2$,
	which accounts for the random geometry of the problem, and in particular 
	the fact that at any given time, the epidemic propagates more easily in more densely populated areas,
	and dies faster in sparsely populated ones; 
	\item The medium changes randomly over time: the individuals have some
	parameterizable random motion, which accounts for the 
	randomness of displacement and allows one to study the effect of mobility reduction;
	\item The infection process itself is a random phenomenon that takes into
	account the geometry of the configuration of other individuals around a susceptible individual
	through the sum over the infection rates of nearby infected individuals.
\end{enumerate}

\subsection{Basic Model}
\label{sec:basmod}
There is initially a single point process $\Xi_0$, which will here be Poisson homogeneous of intensity $\lambda$
in the Euclidean plane, with the points representing the initial location of individuals in the population.
Individuals (hereafter referred to as points) move independently.

We will consider the following {\em random waypoint motion model} where,
at any time, a point stays put for an exponential time
and jumps from its current location to another location with rate $\gamma$;
the displacements are random, i.i.d., independent, and have a symmetrical distribution $D$ on $\mathbb R^2$.
This leads to a location point process $\Xi_t$ at time $t$.
Note that $\Xi_t$ is Poisson with intensity $\lambda$ for all $t$ thanks to the
displacement theorem (Theorem 1.3.9 in \cite{bb}). For the most part, we will consider
the more specific case of a {\em far random waypoint motion} - where the displacements
experienced by points are very large (for example, if $D$ is a two dimensional Gaussian 
vector with i.i.d. coordinates ${\mathcal N}(0,\sigma^2)$ with $\sigma^2$ tending to infinity).\\

There is an infection/contagion/viral charge function $f: \mathbb R^+\to \mathbb R^+$.
The special case 
$$f(r) = \alpha 1_{r\le a}$$
will be considered in the analysis. In this function, $\alpha>0 $ is the pairwise infection rate
and $a >0 $ is the infection radius. In words, the rate at which a susceptible point is
infected is proportional to the number of infected points that are at distance
less than or equal to $a$ from it. More general functions (with bounded or unbounded support) 
can be considered in the theory and in the computational analysis.

The points of $\Xi_t$ can be in one of two states: 1 (or infected) or 0 (or susceptible).
This leads to two point processes $\Phi_t$ and $\Psi_t$ such that
$$ \Xi_t=\Phi_t +\Psi_t,$$
with $\Phi_t$ (resp. $\Psi_t$) the point process of infected (resp. susceptible) points.
The SIS state dynamics is then as follows:

\begin{itemize}
	\item At any time, the state of a susceptible point jumps to infected with
	a rate equal to the current value of the sum of $f$ over all infected points (sum of viral charges) at
	the susceptible point's current location. That is, the transition rate of $X\in \Psi_t$ is
	$$ a(X,\Phi_t)= \sum_{Y \in \Phi_t} f(||X-Y||).$$
	This accounts for the geographic locality of the infection mechanism, i.e., the higher
	chance of infection when surrounded by more infected points.
	\item At any time, the state of an infected point jumps to susceptible with a constant rate $\beta>0$.
\end{itemize}
The above dynamics come in addition to the previously described motion of points. Expressed differently,
our model is that of an SIS epidemic that evolves on a geometric random graph
(where edges exist between points within distance $a$ of each other), and where
the later evolves in time as its vertices move in space.
In all cases described above, the pair $(\Phi_t,\Psi_t)$
is Markov on the space of counting measures, which is not a discrete space.
\tblack{Two equivalent representations of the basic model are presented in Appendix \ref{append:basic_model_alternate}.\\The system has four {\em {rate parameters}}: $\alpha,\beta,\lambda$, and $\gamma$. Since the focus is on the steady-state behaviour of the epidemic, without loss of generality, one of them can be taken equal to 1. When useful, we will take $\alpha=1$ in what follows.} 
\subsection{Aims and Main Results}
The main aim of this paper is to study the role of motion on
the equilibria, or steady-state behaviour of infection.
The equilibria in question are mostly studied in the {\em infinite Euclidean plane}, 
which allows one to leverage the machinery of stationary point processes.
In these infinite models, we answer the following two families of questions:  
\begin{enumerate}
	\item When do the aforementioned equilibria exist, i.e., when is there a non-zero fraction of infected points in steady state? Can we identify a phase diagram,
	namely critical values of the parameters
	that delineate regions where the epidemic almost surely (a.s.) dies out or has a positive probability to survive?
	\item What is the fraction of infected points, $p$, in the steady state? What is the probability that a single infected point in a population causes a sustained epidemic? What is the mean duration between infections of an point? 
\end{enumerate}

\paragraph{Results and Paper Structure:}
Section \ref{sec:graprep}, which is based on results in \cite{liggett}, summarizes 
properties of the contact process that hold for deterministic graphs.
These results are then generalized to our model, i.e., to random geometric graphs
that evolve in time as points move. We establish that when fixing all parameters
except $\beta$, there exists a deterministic
threshold on $\beta$, $\beta_c$, such that the epidemic dies out for $\beta > \beta_c$
and survives for $\beta<\beta_c$.
To use epidemiological terms, we show the existence of an \textit{epidemic threshold}. 
We derive a necessary condition for the existence of a non-degenerate stationary regime
(here non-degenerate means with a positive density of infected points) and establish
a bound (Lemma \ref{lem2}) on the value of the epidemic threshold (eq., a necessary
condition for the survival of the epidemic) in Section \ref{sec:rcp1}.
In this last section, we also establish an infinite sequence of conservation equations
that are satisfied by the dynamics. Of particular importance to what follows,
Theorem \ref{thm3} summarizes the relations that link moment measures of order one, two, and three. 
In Section \ref{sec:HeAn}, we introduce the second-order approximation of
these conservation equations which allows one to obtain heuristic polynomial equations 
which in turn lead to estimates of the fraction $p$ of infected points in the stationary
regime of the epidemic process (which is known as the 
\textit{endemic disease state} in the epidemiology literature).
The results established in Section \ref{sec:graprep} show that $p$ is also equal to
the probability that a single infected point will cause a surviving epidemic -
an important quantity for the study of real-world epidemics.
These heuristics also lead to a tentative phase diagram (Fig. \ref{fig:basphasdiag})
that partitions the space of parameters ($\alpha, \beta, \gamma, \lambda, a$)
into regions of survival and extinction of the epidemic. 
Thoughout the paper, it will be convenient to introduce $\mu=\lambda \pi a^2$,
where $\mu$ is the mean degree of a point in the underlying random geometric graph,
and to discuss partitions of the $(\alpha, \beta, \gamma, \mu)$ parameter space.
We will identify a {\em safe region},
where there is extinction whatever the motion rate $\gamma$, a region 
which is {\em unsafe and motion-insensitive}, namely such that there is
survival for all positive motion rates, and a region which is {\em unsafe and motion-sensitive}.
This last region is the most surprising: when fixing all parameters except $\gamma$,
there are two thresholds $0<\gamma^-_c<\gamma^+_c$ such that
there is survival for $\gamma<\gamma_c^-$ and for $\gamma>\gamma_c^+$ and
extinction for $\gamma_c^-<\gamma<\gamma_c^+$. In other words,
we find that, for certain subsets of the parameter space,
\textit{a reduction in the motion of the population can be favorable to the propagation of the epidemic.}
An analysis, based on the same tools, of the model where there is
no motion of points can be found in \ref{sec:rcp2}. Supporting simulation results for the accuracy of our heuristics and veracity of our phase diagram are distributed throughout Section \ref{sec:HeAn}.\\
Section \ref{sec:var} discusses model variants of practical importance to which the
techniques of the present paper should be applicable. Finally, Section \ref{sec:lico}
gathers the list of conjectures that are made throughout the paper.
\begin{figure}[h!]
	\begin{center}
		\includegraphics[width=0.55\textwidth]{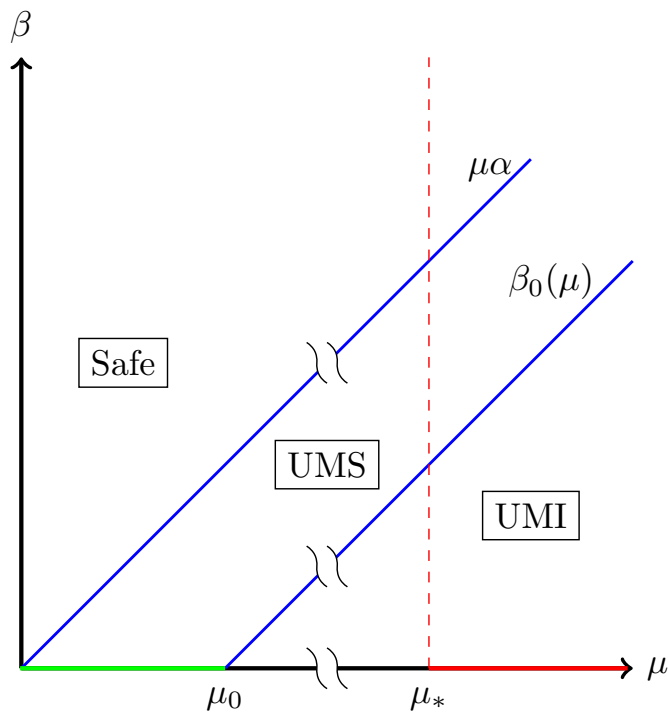}
	\end{center}
	\caption{The $(\mu,\beta)$-phase diagram. 
		Region UMI is the unsafe, motion-insensitive region.
		Region UMS is the unsafe, motion-sensitive region. 
		The region delimited by red semi-line (on the right of $\mu_*\sim 4.5$)
		is the {\em {Boolean supercritical region}}, namely the region where
		there is survival in the no-motion case.
		The green segment, on the left of $\mu_0\sim 0.34$, 
		is the set of values of $\mu$ for which the UMI region is empty.
	}
	\label{fig:basphasdiag}
\end{figure}

\subsection{Related Work}

Previous work related to this paper broadly falls into two categories: the
interacting particle system literature (in particular
the literature on the contact process on graphs) and the 
mathematical epidemiology literature (in particular the literature on SIS epidemics over networks).
In the present paper, the terms `contact process' and `SIS epidemic' will be used
equivalently, as will be the terms `network' and `graph'.
We leverage the mathematical machinery that draws from the former body of literature.
The moment closure heuristics we propose to estimate quantities of interest
are linked to those considered in the latter body of work.

\subsubsection{The Contact Process on Graphs}
In the absence of mobility, the problem was extensively studied in the particle
system literature (see \cite{liggett} and references therein). 
A basic dichotomy in this framework is between finite and infinite graphs.
On finite graphs, the main question is that of the phase transition
between a logarithmic and an exponential growth of the time till absorption (extinction of the epidemic).
This was studied on deterministic graphs like finite grids and regular trees.
There is a large corpus of results on infinite graphs such as grids and
regular trees. This is well covered in the book of T. Liggett \cite{liggett}. 
The contact process was also studied on infinite random graphs with unbounded degrees.
It was first studied on the supercritical Bienaym\'e-Galton-Watson tree \cite{pimentel} where it was shown
that some critical values can be degenerate. 
Contact process models on Euclidean point processes were also thoroughly
studied (see ex. \cite{hao}, \cite{ganesan}, \cite{menard}). 
The case without motion of the present paper ($\gamma=0$) can be seen
as the contact process on the random geometric graph.

\subsubsection{The SIS Epidemic on Networks}
Epidemics on networks have also been extensively studied. The focus here is on
computationally tractable approaches allowing one to predict relevant epidemiological
quantities and to derive qualitative properties of the epidemic.
\cite{mieghem} provides a comprehensive overview of methods and literature on the topic.
In particular, the infinite collection of moment measure equations and subsequent
heuristics to `close' these equations described in Section \ref{ss:heuri}
are similar in spirit to methods used in this literature for the analysis of SIS
epidemics on (non-geometric, finite) graphs. An overview of the application
of such techniques to the analysis of epidemics on graphs is contained in chapter 17
of \cite{newman}, and relevant review articles include \cite{mieghem}
(Section V.A) and \cite{kuehn}. Examples of works that apply moment closure techniques
in an epidemiological setting include \cite{krishnarajah}, \cite{wilkinson}, and \cite{lloyd}.
Epidemics with motion were extensively studied in the random graph setting
(without considering geometry).
The basic model for SIS on graphs is that where agents perform a random walk on the
random graph and where agents meeting at a give point of the graph may infect each other
- see \cite{shneer} and the references therein. \\

In the present paper, we focus on a point process model with the following motion:
a point stays at a given location for an exponential time and then makes a jump with a large magnitude.
This is inspired by the idea of a flight.  
We are not aware of previous research on previous studies on SIS epidemics on point processes with motion.
The main difference with the literature on random geometric graphs is the fact that the structure of the
random geometric graph evolves randomly over time in the case studied in the present paper, whereas 
it is fixed in \cite{ganesan} and the other references mentioned. Further, we derive accurate
heuristics for quantities of interest, unlike most of the work in the interacting particle systems
literature. The main difference with the literature on SIS epidemics on networks is that
we explicitly model the effects of the geometry of the contact graph using the theory
of point processes.\\

In summary, the main novelty of this paper is that it considers both the effects of the geometry
that is inherent to populations affected by epidemics \textit{and} the effects of population motion.
To the best of the authors' knowledge, the key results of the paper,
namely the conservation equations of Theorem \ref{thm3}, the computational heuristics
allowing to evaluate the fraction of infected points in steady state (Section \ref{sec:CSP}),
and the phase diagram alluded to above (Fig. \ref{fig:basphasdiag}), are new.

\section{Basic Properties of SIS Dynamics and their Implications}
\label{sec:graprep}
\paragraph{Deterministic graphs}
For the SIS dynamics (or contact process) on a fixed (deterministic) graph $\Xi$,
the following results are summarized from \cite{liggett} and are obtained from
the graphical representation of the dynamics explained therein.
The underlying probability space is a product space with, for each point, a potential recovery Poisson 
process of rate $\beta$ and as many potential infection point processes of rate $\alpha$ as this
point has neighbors.  
In these results, for all $A\subset \Xi$, $\Phi_t^A$ denotes the subset of infected points
of $\Xi$ at time $t$ if the set of infected points at time 0 is $A$ and
$\mathbb{P}^A$ is the measure that results from the set of infected points at time 0 being $A$.
Using the coupling of the graphical representation of the dynamics, we have
\begin{itemize}
	\item Monotonicity:
	$$\mbox{if}\quad
	A\subset B,\quad \mbox{then}\quad  \Phi_t^{A} \subset
	\Phi_t^{B},\quad \mbox{for all $t$.}$$
	\item Additivity:
	$$ \Phi_t^{A\cup B}= \Phi_t^{A}\cup \Phi_t^{B}, \quad
	\mbox{for all }A,B,t.$$
	\item Self-duality
	$$ \mathbb{P}^A(\Phi_t\cap B\ne \emptyset)= 
	\mathbb{P}^B(\Phi_t\cap A\ne \emptyset), \quad \mbox{for all }A,B,t. $$
\end{itemize}
A direct consequence of monotonicity is the existence of a maximal
(time) invariant measure which is obtained as the limiting measure
when starting with all points infected. The existence of this
limiting measure follows from the fact that the state measure
is stochastically decreasing over time for this initial condition.

A useful consequence of self-duality is the fact that the probability that the
epidemic survives when started from singleton $\{x\}$, with $x\in\Xi$, namely
$$\lim_{t\to \infty} \mathbb{P}^{\{x\}}(\Phi_t\cap \Xi\ne \emptyset),$$
coincides with the mass of the maximal time-invariant probability measure at $x$ (equivalently, the steady-state fraction of infected points), namely
$$\lim_{t\to \infty} \mathbb{P}^\Xi(\Phi_t\cap \{x\}\ne \emptyset).$$

One says that there is strong survival if, for some $x$, for the initial condition $\{x\}$,
$x$ is infected infinitely often with positive probability. In the present paper, 
we concentrate on strong survival (the term survival will always mean strong survival).
The most general structural result on the contact process states that if the probability
that the epidemic strongly survives when started from the initial condition $\{x\}$ 
is positive for some $\beta$, then it is positive for all $\beta' < \beta$ and for all $x$.
This is a consequence of monotonicity. Hence there exists a critical $\beta_c$ (not depending on $x$) 
such that for all $\beta>\beta_c$, the epidemic dies out almost surely
when starting from $x$, and when $\beta<\beta_c$, it
(strongly) survives with a positive probability.\\
Now we generalize these existing results to random geometric graphs, in the three cases of no-motion, random waypoint motion and far random waypoint motion. 



\paragraph{Random geometric graph}
We generalize the above results to the case where the graph is random. The basic properties listed above hold
conditionally on the graph topology. 
If the graph is a random geometric graph of a stationary and ergodic point 
process $\Xi$ under its Palm distribution,
if $\mathcal F$ denotes
the sigma algebra generated by $\Xi$, then 
$$ \mathbb{P}^{\{x\}}(\Phi_t\cap \Xi\ne \emptyset\mid {\mathcal{F}})= 
\mathbb{P}^\Xi(\Phi_t\cap \{x\}\ne \emptyset\mid {\mathcal{F}}), \quad \mbox{for all }t \mbox{ and }x\in \Xi. $$
When letting $t$ tend to infinity, one gets that the (conditional) maximal time stationary measure 
on the random discrete set $\Xi$ puts a mass at $x\in \Xi$ equal to the (conditional) probability
that the epidemic started at $x$ survives.

By unconditioning w.r.t. $\mathcal F$, one gets
$$ \mathbb{P}_{0}(\Phi_\infty\cap \Xi\ne \emptyset\mid \Phi_0=\{0\})= 
\mathbb{P}_0(\Phi_\infty\cap \{0\}\ne \emptyset\mid \Phi_0=\Xi), $$
where $\mathbb P_0$ is the Palm probability of the point process $\Xi$.
The LHS is the probability that the epidemic survives when started
from the typical point (the origin) or equivalently the spatial average
of the probability of survival starting from a single point.
The RHS is the probability that the origin
is infected in the maximal stationary regime, or equivalently the 
fraction of infected points in a large ball under this maximal stationary regime.

\tblack{Also, there exists a constant (this is a constant because of ergodicity) $\beta_c$ such that for $\beta>\beta_c$, the epidemic dies out when started from
any point of the point process (or equivalently the maximal invariant
measure is 0).}, whereas it survives with a positive probability otherwise
(or equivalently the conditional maximal invariant random measure is positive a.s).

Note that, due to ergodicity, if the origin is infected i.o. with positive probability,
it is with probability 1. 

\paragraph{Random geometric graph with point motion}
Assume now that there is motion in the random waypoint sense defined above.
We then have a random graph $\Xi_t$ that evolves with time, the evolution of
which is not affected by the epidemic. One can endow the moving points of this
graph with conditionally independent Poisson point processes of rate $\beta$ and $\alpha$
as above. We claim that monotonicity, additivity and self-duality hold for these dynamics.
So all the conclusions extend to this case as well. In particular, there exists a 
critical $\beta_c$, there exists a maximal invariant measure, etc. We defer the proof to Appendix \ref{appendss:rwp}.

\paragraph{Random geometric graph with far random waypoint motion}

The far random waypoint model studied in the present paper belongs to the class
discussed above when displacements are large. A natural instance is that
of a two dimensional Gaussian vector with i.i.d. ${\mathcal N}(0,\sigma^2)$ coordinates, 
with $\sigma^2$ large. Again, we claim that monotonicity, additivity and self-duality
hold for these dynamics. All the conclusions hence extend to this case as well. We defer the proof to Appendix \ref{appendss:far_rwp}.

\section{Moment Rate Conservation Principle}
\label{sec:rcp1}

\subsection{First Moment Rate Conservation Principle}
Assume that there exists a time-space stationary regime for the dynamics (with or without motion), with $(\Phi,\Psi)$ representatives of the time-space stationary point
processes.
Let $p$ denote the (unknown) stationary fraction of infected points.
That is, in this stationary regime, $\Phi$ has spatial intensity $\lambda p$
and $\Psi$ has intensity $\lambda(1-p)$. 

Let 
$$I_{\Phi} (x) =\sum_{X\in \Phi} f(||X-x||)$$
denote the infection rate of locus $x\in \mathbb R^2$ at time $t$.
The time-space stationary assumption and the Campbell-Mecke theorem imply
$ \mathbb E [ I_{\Phi}(x)]= \lambda p F,$
with
$F= 2\pi \int_0^\infty f(r) rdr.$
In the special case (assumed here by default), $F=\pi a^2 \alpha.$

Pick a subset $D$ of the Euclidean space with volume 1.
The spatial infection rate is defined as
$$ i=\mathbb E [ \sum_{Y\in \Psi \cap D} I_{\Phi}(Y)].$$
From the Campbell-Mecke theorem
$$ i= \lambda(1-p) \mathbb E^0_{\Psi} [I_{\Phi(0)}],$$
with $\mathbb E^0_{\Psi}$ the Palm distribution w.r.t. $\Psi$.
The spatial recovery rate is defined as
$$ r=\mathbb E [ \sum_{X\in \Phi \cap D} \beta]=\lambda p\beta .$$

The Rate Conservation Principle (RCP, \cite{BB03}) gives that $i=r$, namely
\begin{Lem}\label{lem1}
	Under the foregoing assumptions,
	\begin{equation}\label{eq:rcp1}
		p\beta = (1-p) \mathbb E^0_{\Psi} [I_{\Phi}(0)].
	\end{equation}
\end{Lem}
The last relation will be referred to as the {\em first moment RCP}.

\subsection{A Necessary Condition for Survival}
\label{ss:ncfs}
A natural conjecture, backed by simulations, is that there is
repulsion between $\Phi$ and $\Psi$, namely
$$ \mathbb E^0_{\Psi} [I_{\Phi}(0)] \le \mathbb E [I_{\Phi}(0)].$$
In words, in the steady state, a typical susceptible point will 
see less infection than a typical locus in space.
This will be referred to as the {\em repulsion conjecture}.

This and Lemma \ref{lem1} give
$$ p\beta = (1-p) \mathbb E^0_{\Psi} [I_{\Phi(0)}] 
\le (1-p) p \mu,$$
that is, if $p\ne 0$,
$$p \le 1- \frac{\beta}{\alpha \mu},$$
with $\mu= \lambda\pi a^2$ the mean degree in the random geometric graph. 
This proves the following result on the critical value $\beta_c$ defined in Section \ref{sec:graprep}:

\begin{Lem}\label{lem2} Under the $(\Phi,\Psi)$ repulsion conjecture,
\begin{equation}
\label{eq:critab}
\beta_c \le \alpha \mu.
\end{equation}
\end{Lem}

In words, assuming that the repulsion conjecture holds,
if $\beta>\alpha\mu$, then there is extinction a.s. (this is the
safe region alluded to above) and the only invariant measure is the empty measure.
As we shall see, according to our phase diagram, there are regions of the parameter space
where $\beta_c< \alpha \mu$.
In addition either $p=0$ or the fraction of infected points satisfies
\begin{equation}\label{eq:bound}
	0<p \le 1- \frac{\beta}{\alpha \mu}.
\end{equation}
Note that motion plays no role in this lemma. It is in fact hidden in the Palm expectation.
Also note that the result holds for all stationary point processes
(we did not use the Poisson assumption here). This will be used in the next subsection. \\

The result of Lemma \ref{lem2} can be connected to the
basic {\em reproduction number} ($\mathcal{R}_0$) of epidemiology.
This number is defined as the expected number of infections generated by a single infected
point located in a population of all-susceptible points. In several epidemiological models,
an emerging infection is predicted to die out in a population if $\mathcal{R}_0$ is less than 1,
and to survive if it is more then 1.
Since a single infected point has an average degree $\mu$, infects neighbours at rate
$\alpha$ and recovers at rate $\beta$, it is easy to see that
in the fast motion case ($\gamma$ large), $\mathcal{R}_0 = \frac{\alpha\mu}{\beta}$.
In this fast motion case, the bound of Lemma \ref{lem2} should hence be the true value for $\beta_c$,
namely $\beta_c= \alpha\mu$. However, as already mentioned above,
there exist parameter regions where $\beta_c< \alpha \mu$. In other words, in this
random medium, random motion setting, $\mathcal{R}_0$ is not the relevant parameter
to predict whether the epidemic dies out or survives, in that the epidemic
dies out in parameter regions where $\mathcal{R}_0>1$.

\subsection{A Primer on Moment Measures}
\label{sec:prel}
As already explained, we will establish conservation equations for higher order
moment measures of the stationary (in time and in space) point processes of susceptible and
infected points. This subsection summarizes some basic properties of these measures.
\subsubsection{Definition}
The second factorial moment measure of a point process $\phi$, $\rho_{\phi}^{(2)}$, is defined by
$$\mathbb E [  \sum_{X\ne X'\in \phi} g(X,X')]=
\int_{\mathbb R^2} g(x,x') \rho_{\phi}^{(2)} (x,x') {\rm d}x {\rm d}x',$$
for all measurable non-negative functions $g$.
For a stationary point process of intensity $\lambda$,
$$ \rho_{\phi}^{(2)} (x,x')= \lambda^2 \xi^{(2)}_{\phi} (x'-x),$$
with $\xi_\phi$ the pair correlation function of $\phi$.

Similarly, the joint moment measure of $(\phi,\psi)$,
$ \rho_{\phi,\psi}^{(2)}$, is defined by
$$\mathbb E [  \sum_{X\in \phi, Y\in \psi} g(X,Y)]=
\int_{\mathbb R^2} g(x,y) \rho_{\phi,\psi}^{(2)} (x,y) {\rm d}x {\rm d}y,$$
for all measurable non-negative functions $g$.
In the case of two jointly (space)-stationary point processes with
intensities $\lambda$ and $\mu$,
$$ \rho_{\phi,\psi}^{(2)} (x,x')= \lambda \mu \xi^{(2)}_{\phi,\psi} (x'-x),$$
with $\xi^{(2)}_{\phi,\psi}$ the cross-pair correlation function of $(\phi,\psi)$.

In the isotropic case, we have
$$ \xi^{(2)}_{\phi,\psi} (x'-x)= \tilde \xi^{(2)}_{\phi,\psi} (||x'-x||).$$
By abuse of notation, we will often drop the tilde in the RHS of the last relation.

\subsubsection{A General Relation between Pair Correlation Functions}
Let $(\upsilon,\phi,\psi)$ be three jointly stationary point processes 
such that $\upsilon=\phi+\psi$. Let $\lambda,\mu$ be the intensities of
$\phi$ and $\psi$, respectively. Let $p=\lambda/(\lambda+\mu)$. Then,
for all $r$,
\begin{equation}
\label{eq:proto-conservrel}
(1-p)^2\xi^{(2)}_{\psi,\psi} (r)
+ p^2 \xi^{(2)}_{\phi,\phi} (r)
+2 p(1-p) \xi^{(2)}_{\psi,\phi} (r)=
\xi^{(2)}_{\upsilon,\upsilon} (r).
\end{equation}
where the RHS will depend on the fact that $\upsilon$ is, e.g., a clustered or a repulsive point process.
If $\phi+\psi$ forms a Poisson point process, then
\begin{equation}
\label{eq:conservrel}
(1-p)^2\xi^{(2)}_{\psi,\psi} (r)
+p^2 \xi^{(2)}_{\phi,\phi} (r)
+2p(1-p) \xi^{(2)}_{\psi,\phi} (r)=1.
\end{equation}

Throughout the paper, we use the simplified notation
$\xi_{\psi,\psi} (r)$ in place of
$\xi^{(2)}_{\psi,\psi} (r)$
for the cross-pair correlation function of $(\phi,\psi)$.
\subsection{Reformulation of RCP in Terms of Moment Measures}
\label{ss:rtmm}
It is possible to represent the Palm expectation used above in terms of moment measures.
We have
$$ \mathbb E^0_{\Psi} [I_{\Phi(0)}]
= \lambda p \int_{\mathbb R^2} f(x) \xi_{\Phi,\Psi} (x) {\rm d}x.
$$
So (\ref{eq:rcp1}) can be rephrased in terms of the following integral equation:
\begin{equation}\label{eq:rcp1-integ}
	\beta = (1-p) \lambda \int_{\mathbb R^2} \xi_{\Psi,\Phi} (x) f(||x||) {\rm d}x.
\end{equation}

If the cross-pair correlation function
$ \xi_{\Psi,\Phi} (\cdot)$
is constant equal to 1, which is the case
when $\Phi$ and $\Psi$ are independent (possibly in the infinite velocity case),
then this boils down to
\begin{equation}\label{eq:eqorig}
	\beta= \alpha (1-p) \mu,
\end{equation}
that is
\begin{equation}\label{eq:first-order}
	p = 1- \frac{\beta}{\alpha \mu},
\end{equation}
assuming that $\frac{\beta}{\alpha \mu}< 1$.
In other words, if there exists a non-degenerate stationary regime with $\Phi$
and $\Psi$ independent, then this achieves the upper bound in (\ref{eq:bound}).
We conjecture that this regime is achieved in the high velocity case
provided $\alpha \mu> \beta$, which is backed by simulations.

Let us come back to the general case for $\xi_{\Psi,\Phi}$.
If we take the special case discussed above for $f$, we get
$$\beta=(1-p) \lambda \alpha \int_{B(0,a)} \xi_{\Psi,\Phi} (x) {\rm d}x.$$
Repulsion between $\Phi$ and $\Psi$ implies that
$$\int_{B(0,a)} \xi_{\Psi,\Phi} (x) {\rm d}x\le \pi a^2.$$
So for fixed $a,\alpha,\beta,\lambda$, any model with repulsion requires a smaller $p$ 
than the infinite velocity (or fast motion mean-field) model, which is in line with the bound in
(\ref{eq:bound}).

One can use isotropy to write (\ref{eq:rcp1-integ}) as
\begin{equation}
	\label{eq:rcp1-integ-pol}
	\lambda p \beta =\lambda (1-p) \lambda p
	2 \pi \int_{\mathbb R^+} \xi_{\Psi,\Phi} (r) f(r) r dr.
\end{equation}
Note the abuse of notation: we used the same notation as above for the cross-pair correlation functions
in polar coordinates.

\subsection{Higher Moment Rate Conservation Principle}
\label{sec:CSP}
The far random waypoint model features a mix of zero velocity
and far away motion. Each point has an independent exponential clock with rate $\gamma$.
When its clock ticks, the point jumps very far away.
Arrivals are hence according to a Poisson rain (in space-time) with intensity $\lambda \gamma p$  
for type $I$ and an independent Poisson rain of intensity $\lambda \gamma (1-p)$ for type $S$.

We recall the structural results listed in Section \ref{sec:graprep}:
assume all parameters fixed except $\beta$.
There exists $\beta_c$ such that for $\beta>\beta_c$, the epidemic dies out for
sure (or equivalently the maximal invariant measure is the zero measure),
whereas for $\beta> \beta_c$ the epidemic might survive
(or equivalently the maximal invariant measure is a positive measure).

Note that we do not know whether there exists a similar threshold on the parameter $\gamma$.
At first glance, motion seems instrumental in the propagation of the epidemic by
bringing infected points from far away.
However, motion may also dissolve dense clusters where the epidemic survives for an arbitrarily long time.
This is an important mathematical question, to which the tentative phase diagram obtained 
in this section gives a somewhat unexpected answer:
in the proposed setting, motion is not always favoring the survival of the epidemic.

\subsubsection{RCP for Second Moment Measures}
One can see the LHS of the first moment RCP equation (\ref{eq:rcp1})
as the "mass birth rate" of $\rho^{(1)}_{\Psi}$
(or equivalently the mass death rate of $\rho^{(1)}_{\Phi}$)
and the RHS as the "mass death rate" of $\rho^{(1)}_{\Psi}$
(or equivalently the mass birth rate of $\rho^{(1)}_{\Phi}$).
We see that this conservation law on the first moment measure involves
a second moment measure.

One can think in the same terms for second moment measures.
Let $\mu(\Phi)_{\Psi,\Phi}^{0,r}(x)$ denote the conditional density of $\Phi$ at $x$ given that
$\Psi$ has a points at $(0,0)$ and $\Phi$ a point at $(r,0)$.
Let also $ \mu(\Phi)^{0,r}_{\Psi,\Psi}(x)$
denote the conditional density of $\Phi$ at $x$ given that
$\Psi$ has a points at $(0,0)$ and a point at $(r,0)$.
The main result of this section is

\begin{Thm}
	\label{thm3}
	Under the far random waypoint mobility model with jump rate $\gamma$, 
	if there exists a stationary regime, then the stationary pair correlation functions 
	satisfy the following system of integral equations:
	\begin{eqnarray}
		\label{eq:rcpmmmixed-glob}
		p\xi_{\Phi,\Phi} (r) (\beta+\gamma)
		& =  &  p \gamma
		+ (1-p)
		\xi_{\Psi,\Phi} (r) \left(
		f(r)+\int_{\mathbb R^2} \mu(\Phi)_{\Psi,\Phi}^{0,r}(x) f(||x||) {\rm d}x\right)
		\nonumber\\
		p \xi_{\Psi,\Phi} (r) \beta +(1-p) \gamma
		& =  & 
		(1-p) 
		\xi_{\Psi,\Psi} (r)
		\left( \gamma+ \int_{\mathbb R^2} \mu(\Phi)_{\Psi,\Psi}^{0,r}(x) f(||x||) {\rm d}x \right)
		\nonumber\\
		\beta & = & \lambda (1-p) 2 \pi
		\int_{\mathbb R^+} \xi_{\Psi,\Phi} (r) f(r) r dr\nonumber\\
		1
		& = &
		(1-p)^2\xi_{\Psi,\Psi} (r) +
		p^2 \xi_{\Phi,\Phi} (r)
		-2p(1-p) \xi_{\Psi,\Phi} (r).\nonumber\\
	\end{eqnarray}
\end{Thm}

\begin{proof}
	The last equation comes from the fact that $\Xi_t$ is Poisson.
	
	Here is a sketch of the proof for the other equations.
	The "mass birth rate" in $\rho^{(2)}_{\Phi,\Phi} (r)$ is
	\begin{eqnarray*} 
		2 \lambda p\lambda p \gamma 
		+  2 \rho^{(2)}_{\Psi,\Phi} (r)\ 
		\left( f(r)+ 
		\int_{\mathbb R^2} \mu(\Phi)^{0,r}_{\Psi,\Phi}(x) f(||x||) {\rm d}x,
		\right) 
	\end{eqnarray*}
	where we used the fact that arrivals of infected points are Poisson of
	intensity $\lambda\gamma p$ (see the end of Subsection \ref{sec:graprep}).
	and the "mass death rate" in $\rho^{(2)}_{\Phi,\Phi} (r)$ is
	$ 2 \rho^{(2)}_{\Phi,\Phi} (r) (\beta+\gamma) .$
	The 2 comes from the fact that the recovery of motion of a point of $\Phi$ deletes two
	infected points at a distance $r$ of each other.
	The "mass death rate" in $\rho^{(2)}_{\Phi,\Phi} (r)$ is
	\begin{eqnarray}
		\label{eq:rcpmm2mixed}
		p^2\xi_{\Phi,\Phi} (r) (\beta+\gamma)
		=    p^2 \gamma
		+ p(1-p)
		\xi_{\Psi,\Phi} (r) \left(
		f(r)+\int_{\mathbb R^2} \mu(\Phi)_{\Psi,\Phi}^{0,r}(x) f(||\mathbf{x}||) d\mathbf{x}\right).
	\end{eqnarray}
\end{proof}

Similarly, the "mass birth rate" in $\rho^{(2)}_{\Psi,\Psi} (r)$ is
$$ 2 \rho^{(2)}_{\Psi,\Phi} (r) \beta + 2 \lambda^2 (1-p)^2 \gamma ,$$
and the "mass death rate" in $\rho^{(2)}_{\Psi,\Psi} (r)$ is
$$ 
2 \rho^{(2)}_{\Psi,\Psi} (r)
\left( \gamma+ \int_{\mathbb R^2} \mu(\Phi)^{0,r}_{\Psi,\Psi}(\mathbf{x})
f(||\mathbf{x}||) \mathbf{{\rm d}x} \right).$$
Hence
\begin{eqnarray}
	\label{eq:rcpmm3mixed}
	p (1-p)\xi_{\Psi,\Phi} (r) \beta
	+(1-p)^2 \gamma =
	(1-p)^2 \xi_{\Psi,\Psi} (r)
	\left( \gamma+ \int_{\mathbb R^2} \mu(\Phi)_{\Psi,\Psi}^{0,r}(\mathbf{x})
	f(||\mathbf{x}||) d\mathbf{x} \right).
\end{eqnarray}

If $\xi\equiv 1$, this leads to (\ref{eq:eqorig}).

Assume that when $\gamma\to \infty$, there is a finite limit for each $\xi$ function.
Then (\ref{eq:rcpmm2mixed}) implies that,
when $\gamma$ tends to infinity,
$ \xi_{\Phi,\Phi} (r)$ tends to 1 for all $r$.
By the same argument used on (\ref{eq:rcpmm3mixed}), one gets that,
when $\gamma$ tends to infinity,
$\xi_{\Psi,\Psi} (r)$ tends to 1 for all $r$.
Hence, from (\ref{eq:conservrel}),
$\xi_{\Psi,\Psi} (r)$ tends to 1 for all $r$ as well. 
But we know that for such pair correlation functions, the relation
(\ref{eq:eqorig}), $ \beta=(1-p) \alpha \mu$, holds.
In this sense, the last equations give the desired continuum between the
no velocity and the high velocity cases.

\subsubsection{RCP for Higher Order Moment Measures}
The last theorem provides a conservation law on second moment measures which
involves third moment measures. One can go along the path that would now consist
in writing down a conservation law for third moment measures which will involve
moment measures of order four and so on. This infinite hierarchy
of integral equations (which was discussed in a different context
in \cite{bmn}) should characterize the dynamics in an exact way.
We will not pursue this line of thoughts here.
We will rather follow a more computational path 
which consists in introducing various factorization heuristics for third moment measures
which are introduced and analyzed below.

\section{Heuristic Analysis}
\label{sec:HeAn}
We will see below that each factorization leads to a system of integral and polynomial equations
jointly satisfied by the first
and second order moment measures, which in turn provide computational approximations 
for these measures. Before introducing these factorizations, let us describe what they bring in the two next subsections.

\subsection{Prediction of the Stationary Fraction of Infected Points}
Assume one wants to predict the
fraction $p$ of points that are infected in the steady state when there is survival,
for a given set of parameters $(\alpha, \beta, \gamma, \mu)$.

Let us illustrate through one example among several others
discussed below how to use some polynomial heuristic to do this.
Let $w$ be the value of the cross-pair correlation function at distance zero
of the infected and susceptible point processes. The intuitive meaning
of $w$ is as follows: given that there is a susceptible point at the origin,
the intensity of infected points in the near vicinity of the origin (near 
is understood w.r.t. $a$) is not $\lambda p$ as it would be if there was no
probabilistic dependence between infected and susceptible, but $\lambda p w$ with $w<1$.
Then, we show that for each one of the heuristic factorizations in question,
the unknown variable $w$ satisfies some explicit polynomial equation
that practically allow one to characterize $w$ and $p$ in closed form.
For instance, for a certain factorization based on Bayes' rule
(see Subsection \ref{sss:m2bi} below), $w$ is a positive number satisfying 
\begin{eqnarray}
	& &\hspace{-1cm} (2 \gamma+\beta)\big(
	(\alpha \mu w -\beta +2 \gamma)(w^2\alpha^2\mu^2-2\beta(\alpha \mu w-\beta)w)\hspace{-.1cm}
	+\hspace{-.1cm} w(\alpha \mu w -\beta) \beta^2 \hspace{-.1cm}- \hspace{-.1cm}2\gamma\beta^2\big) \nonumber\\
	= &&(\alpha \mu w-\beta)
	(\alpha \mu w -\beta +2 \gamma)\big(
	2\gamma(\alpha \mu w-\beta) + 2\alpha \beta w+ \beta w (\alpha \mu w -\beta)\big),\hspace{-3cm}
\end{eqnarray}
where $\gamma$ is the motion rate, $\beta$ the recovery rate, $\alpha$ the infection rate,
and $\mu=\lambda \pi a^2$, with $\lambda$ the spatial intensity of the Poisson point process
and $a$ is the infection radius. Once $w$ is determined through this
polynomial equation, $p$ is then obtained through the formula
\begin{equation} p=1-\frac{\beta}{\alpha \mu w}.\end{equation}
Discrete event simulation of the dynamics over large tori shows
that the solution of this polynomial equation allows one to accurately
predict the fraction of infected points as announced.
Each factorization gives slightly different numerical results, as is natural
for different heuristics, but always in line with the simulation results.

\subsection{Prediction of Criticality Parameters and Tentative Phase Diagram}
\label{ss:criticality_phase_diag}
As we shall see below, each of these polynomial equations also leads to a phase diagram
determining the regions of the parameter space where there
is survival or extinction. The boundaries between these regions 
are characterized by stating that $p=0$ in the polynomial system.
The rationale is as follows.
Consider a region $P$ of the parameter space where $p>0$ and another region $Q$ where $p=0$.
Assume that the boundary between these two regions forms a smooth curve $C$. The rationale
is that when moving in $Q$ towards $C$, the value of $p$ should tend to 0.
This continuity property is not granted for all phase transitions. In the present situation, this
assumption is substantiated by simulations
(see Table \ref{table:sick_proportion-vela-below-vary-beta} for an example).
This leads to further polynomial relations between the
$( \alpha,\beta, \gamma, \mu)$ parameters which determine the boundaries between the regions
where $p=0$ and $p>0$.
The general shape of the resulting phase diagram is depicted in Figure \ref{fig:basphasdiag}.
The precise definitions of the boundaries between the regions again slightly
vary depending on the chose heuristic. The general structure is nevertheless
common to all the considered heuristics.

Since this phase diagram is based on the polynomial heuristics, it is itself a heuristic.
Its qualitative validity could not be proved within the framework of this paper.
We explain in Subsection \ref{sec:soc} how it can be partially substantiated by
simulation on large tori. In particular, it is shown that, on any large torus, 
when picking a $\mu$ and a $\beta$ in the unsafe motion-insensitive region,
the mean time to extinction sharply increases with motion rate in the vicinity of $\gamma_c^-$
and sharply decreases with this rate in the vicinity of $\gamma_c^+$. These
monotonicity properties can be substantiated
by simulation using confidence intervals on this mean time to extinction. 
In addition to partially substantiating the phase diagram (which deals with epidemics
living on the infinite plane), these simulation results also show that the phase diagram
accurately predicts the trends of the survival time of the epidemic when it lives on a large finite torus.\\

\subsection{Heuristic Factorizations of Third Moment Measures}
\label{ss:heuri}

The announced heuristics are of two types: either based on Bayes' formula and a conditional
independence approximation, or based on a simple form of conditional dependence.

\subsubsection{Heuristics based on Bayes' formula and conditional independence}

Below, we describe various heuristics which are all based
on a conditional independence approximation and can be summarized as follows:
the $\mu$ densities, defined in relation with Theorem \ref{thm3},
are conditional densities of infected points at $\mathbf{x}\in \mathbb R^2$
given the two events, e.g. that there is a susceptible point at $\mathbf{0}=(0,0)$ and
another susceptible point at $\mathbf{r}:=(r,0)$. We use Bayes' formula to represent this
in terms of the conditional probability that there is a susceptible point at $\mathbf 0$
and another susceptible point at $\mathbf r$
given there is an infected point at $\mathbf x$. We then use the conditional independence
approximation to represent the latter in terms of a product of pair correlation functions.

Consider first $\mu(\Phi)^{\mathbf{0},\mathbf{r}}_{\Psi,\Psi}(\mathbf{x})$,
which we recall to be the conditional density of $\Phi$ at $x$ 
under the two point Palm probability of $\Psi$ at $(0,0)$ and $(0,r)$.
Denote by $\mu(\Psi,\Phi)^{r}_{\Psi}(0,x)$ the joint
density of $(\Psi,\Phi)$ at $(\mathbf{0},\mathbf{x})$ under the Palm of $\Psi$ at $\mathbf{r}$
and use a similar notation for $\mu(\Psi,\Phi)^{0}_{\Psi}(\mathbf{r},\mathbf{x})$
and $\mu(\Psi,\Psi)^{\mathbf x}_{\Phi}(\mathbf{0},\mathbf{r})$.
The third moment density of $(\Psi,\Psi,\Phi)$ at $(\mathbf{0,r,x})$ is
$$\mu(\Phi)^{\mathbf{0},\mathbf{r}}_{\Psi,\Psi}(\mathbf{x}) \xi_{\Psi,\Psi}(\mathbf{r})
\lambda^2 (1-p)^2$$
and we have the following conditional representations of the latter:
\begin{eqnarray*}
\mu(\Phi)^{\mathbf {0,r}}_{\Psi,\Psi}(\mathbf {x}) 
\xi_{\Psi,\Psi}(\mathbf {r}) \lambda^2 (1-p)^2
& = & \mu(\Psi,\Psi)^{\mathbf {x}}_{\Phi}(\mathbf {0,r}) \lambda p \\
& = & \mu(\Psi,\Phi)^{\mathbf {r}}_{\Psi}(\mathbf {0,x}) \lambda (1-p) \\
& = & \mu(\Psi,\Phi)^{\mathbf {0}}_{\Psi}(\mathbf {r,x}) \lambda (1-p) ,
\end{eqnarray*}
which is in essence Bayes' rule rewritten in three different ways.

For three jointly stationary point processes $\pi,\phi$, and $\psi$,
let $\rho^{(3)}_{\phi,\psi,\pi}(x,y,z)$ denote the
the third moment density of $(\phi,\psi,\pi)$ at $(x,y,z)$.
The use of Bayes' rule is heuristically justified when interpreting
$$\rho^{(3)}_{\phi,\psi,\pi}(x,y,z) \Delta x \Delta y \Delta z$$
as the probability that $\phi$ has one point in a small neighborhood of $x$ of volume $\Delta x$,
$\psi$ has one point in a small neighborhood of $y$ of volume $\Delta y$,
and $\pi$ has one point in a small neighborhood of $z$ of volume $\Delta z$.

So for all positive integers $k$ and $l$,
\begin{eqnarray*}
	& &\hspace{-1cm} \left(\mu(\Phi)^{\mathbf{0,r}}_{\Psi,\Psi}(\mathbf{x}) 
	\xi_{\Psi,\Psi}(r) \lambda^2 (1-p)^2\right)^{k+2l}\\
	& = &
	\left(\mu(\Psi,\Psi)^{\mathbf{x}}_{\Phi}(\mathbf{0,r}) \lambda p\right)^{k} 
	\left(\mu(\Psi,\Phi)^{\mathbf{r}}_{\Psi}(\mathbf{0,x}) \lambda (1-p)\right)^{l} 
	\left(\mu(\Psi,\Phi)^{\mathbf{0}}_{\Psi}(\mathbf{r,x}) \lambda (1-p)\right)^{l} \\
	& = &
	\left(
	\xi_{\Psi,\Phi}(||\mathbf{x}||) \lambda(1-p) 
	\xi_{\Psi,\Phi}(||\mathbf{x-r}||) \lambda(1-p) \lambda p \right)^{k} \\
	& & \left(
	\xi_{\Psi,\Psi}(r) \lambda(1-p)
	\xi_{\Psi,\Phi}(||\mathbf{x-r}||)\lambda p \lambda(1-p) \right)^{l} \\
	& & \left(
	\xi_{\Psi,\Psi}(r) \lambda(1-p)
	\xi_{\Psi,\Phi}(||\mathbf{x}||) \lambda p \lambda(1-p) \right)^{l},
\end{eqnarray*}
where the last relation follows from the conditional independence heuristic.
The meaning of $k$ and $l$, which will be used for the classification
below, is the following: a bigger $k$
puts in some sense more emphasis on the positive correlation structure.
It follows that
\begin{eqnarray}
	\label{eq:genb1}
	\mu(\Phi)^{\mathbf{0,r}}_{\Psi,\Psi}(\mathbf{x})
	=\lambda p
	\xi_{\Psi,\Phi}(||\mathbf{x}||)^{\frac{k+l}{k+2l}}
	\xi_{\Psi,\Phi}(||\mathbf{x-r}||)^{\frac{k+l}{k+2l}}
	\xi_{\Psi,\Psi}(r)^{-\frac{k}{k+2l}}.
\end{eqnarray}
Similarly
\begin{eqnarray*}
	& &\hspace{-1cm} \left(\mu(\Phi)^{\mathbf{0,r}}_{\Psi,\Phi}(\mathbf{x}) 
	\xi_{\Psi,\Phi}(r) \lambda^2 (1-p)p\right)^{k+2l}\\
	& = &
	\left(\mu(\Psi,\Phi)^{\mathbf{x}}_{\Phi}(\mathbf{0,r}) \lambda p\right)^{l} 
	\left(\mu(\Psi,\Phi)^{\mathbf{r}}_{\Phi}(\mathbf{0,x}) \lambda p\right)^{l} 
	\left(\mu(\Phi,\Phi)^{\mathbf 0}_{\Psi}(\mathbf{r,x}) \lambda (1-p)\right)^{k} \\
	& = &
	\left(
	\xi_{\Psi,\Phi}(||\mathbf{x}||) \lambda(1-p) 
	\xi_{\Phi,\Phi}(||\mathbf{x-r}||) \lambda p \lambda p \right)^{l} \\
	& & \left(
	\xi_{\Psi,\Phi}(r) \lambda(1-p)
	\xi_{\Phi,\Phi}(||\mathbf{x-r}||) \lambda p \lambda p \right)^{l}\\
	& & \left(
	\xi_{\Psi,\Phi}(r) \lambda p
	\xi_{\Psi,\Phi}(||\mathbf{x}||)\lambda p \lambda(1-p) \right)^{k} .
\end{eqnarray*}
Hence
\begin{eqnarray}
	\label{eq:genb2}
	\mu(\Phi)^{\mathbf{0,r}}_{\Psi,\Phi}(\mathbf{x})
	=\lambda p
	\xi_{\Psi,\Phi}(||\mathbf{x}||)^{\frac{k+l}{k+2l}}
	\xi_{\Phi,\Phi}(||\mathbf{x-r}||)^{\frac{2l}{k+2l}}
	\xi_{\Psi,\Phi}(r)^{-\frac{l}{k+2l}}.
\end{eqnarray}

\subsubsection{Heuristics based on mean values}
Consider first $\mu(\Phi)^{\mathbf{0},\mathbf{r}}_{\Psi,\Phi}(\mathbf{x})$,
which we recall to be the conditional density of $\Phi$ at $x$ 
under the two point Palm probability of $\Psi$ at $(0,0)$ and $\Phi$ at $(0,r)$.
Heuristically this, multiplied by $\Delta x$, is the probability of 
of event $A$ that there is a point in a region of small volume $\Delta x$ around $x$
given two events $B$ and $C$, that is $ P(A \mid B\cap C)$.
The only data we have are $P(A\mid B)$ and $P(A\mid C)$. A natural
heuristic is the geometric mean
$$ P(A \mid B\cap C) \sim \sqrt{P(A\mid B) P(A\mid C)}.$$
A more general heuristic is
$$ P(A \mid B\cap C) \sim P(A\mid B)^{\eta} P(A\mid C)^{1-\eta},$$
with $0\le \eta \le 1$.
One can also consider the arithmetic mean
$$ P(A \mid B\cap C) \sim \frac { P(A\mid B) + P(A\mid C)}{2}.$$
A more general heuristic is
$$ P(A \mid B\cap C) \sim P(A\mid B) {\eta}+ P(A\mid C)(1-\eta),$$
with $0\le \eta \le 1$.
The geometric mean leads to
\begin{eqnarray}
	\label{eq:gd1}
	\mu(\Phi)^{\mathbf{0,r}}_{\Psi,\Psi}(\mathbf{x})
	=\lambda p
	\xi_{\Psi,\Phi}(||\mathbf{x}||)^{\eta}
	\xi_{\Psi,\Phi}(||\mathbf{x-r}||)^{1-\eta}
\end{eqnarray}
and
\begin{eqnarray}
	\label{eq:gd2}
	\mu(\Phi)^{\mathbf{0,r}}_{\Psi,\Phi}(\mathbf{x})
	=\lambda p
	\xi_{\Psi,\Phi}(||\mathbf{x}||)^{\eta}
	\xi_{\Phi,\Phi}(||\mathbf{x-r}||)^{1-\eta}.
\end{eqnarray}
Note that a bigger $\eta$ puts more emphasis on the positive correlation.
The arithmetic mean leads to
\begin{eqnarray}
	\label{eq:ad1}
	\mu(\Phi)^{\mathbf{0,r}}_{\Psi,\Psi}(\mathbf{x})
	=\lambda p \left( \eta
	\xi_{\Psi,\Phi}(||\mathbf{x}||)
	+(1-\eta) \xi_{\Psi,\Phi}(||\mathbf{x-r}||)\right)
\end{eqnarray}
and
\begin{eqnarray}
	\label{eq:ad2}
	\mu(\Phi)^{\mathbf{0,r}}_{\Psi,\Phi}(\mathbf{x})
	=\lambda p \left(
	\eta \xi_{\Psi,\Phi}(||\mathbf{x}||)
	+ (1-\eta) \xi_{\Phi,\Phi}(||\mathbf{x-r}||)\right).
\end{eqnarray}
Here, a bigger $\eta$ puts more emphasis on the negative correlation.

\subsubsection{Combining and Classifying Heuristics}
These two broad type of heuristics 
admit several variants and combinations. For instance, the variants of Bayes' heuristic 
are obtained by varying $k$ and $l$ in (\ref{eq:genb2}) and those 
based on mean values are obtained by choosing either a geometric or an arithmetic mean
and by varying the parameter $\eta$ in (\ref{eq:gd1})--(\ref{eq:ad2}).
An example of combination is that where in the Bayes' approach, 
one replaces the conditional independence step by the mean value heuristic. 
Mixtures of heuristics can also be used, for example by averaging two Bayes' heuristics
- we will see one such heuristic (named M2BI), below. We defer the full classification and naming of these heuristics to Appendix \ref{append:motion_heuristics}.

\subsection{Terminology for Functional and Polynomial Equations}
\label{ss:prince}

\paragraph{Integral Equations}
When plugging any of the heuristics into
the equations of (\ref{eq:rcpmmmixed-glob}),
we get a system of integral equations, which will be referred to
as the {\em pairwise-interaction second moment measure functional equations}.

\paragraph{Polynomial Equations}
Each of these functional equations in turn leads to
polynomial equations satisfied by the value of the pair correlation functions close to zero.
These will be referred to as the {\em pairwise-interaction second moment measure polynomial equations}.
There is a functional and a polynomial equation for each heuristic of the classification.
The setting for polynomial equations is as follows:
it considers the special case with $f(r)=\alpha 1_{r<a}$ and 
(with $\mu = \lambda \pi a^2$), it assumes that $\mu w > \beta$ and that
\begin{itemize}
	\item $\xi_{\Psi,\Phi}(.)$ is almost constant on $(0,a)$ 
	and equal to $w$\footnote{Assuming that $w<1$ is equivalent to what we called cluster or second
		repulsion above; more general assumptions should be considered in the no-motion case};
	\item $\xi_{\Phi,\Phi}(.)$ is almost constant on $(0,a)$ and equal to $v$.
	\item $\xi_{\Psi,\Psi}(.)$ is almost constant on $(0,a)$ and equal to $z$.
\end{itemize}
For numerical justifications of this heuristic, see Figure \ref{fig:pcg1} below.

The polynomial equations will be in the three variables $v,w,z$.
Note that (\ref{eq:rcp1-integ}) then reads
\begin{eqnarray}
	\label{eq:univ}
	\beta=\alpha\lambda (1-p) \pi a^2 w= (1-p) \alpha\mu w.
\end{eqnarray}

So if there exists a stationary regime, with flat enough pair correlation
functions in the said range, then necessarily the variables $p,v,z$ and $w$ will satisfy the announced
'polynomial' equation. Note that this is {\em not} sufficient for the functional equation to
have a non-degenerate solution.
We now provide and subsequently use two examples of heuristics: named M2BI
(for Mixtures of 2 types of Bayes-Independent) and B1I (a Bayes-Independent heuristic)
to see how the associated sets of functional
equations yield polynomial equations and hence tentative phase diagrams. The other heuristics 
(B1G1, M$\infty$BI, M$\infty$B1G1, etc.) are listed and
discussed in Appendix \ref{append:motion_heuristics_polynomial}.

\paragraph{Terminology}
Below and in Appendix \ref{append:motion_heuristics_polynomial} we use the following code:
{\bf f} for functional and {\bf p} for polynomial. For instance f-b1i means the functional equation 
associated to the B1I heuristic which is described below, p-m$\infty$b1g1 means the polynomial equation
of the M$\infty$B1G1 heuristic, etc.

\subsubsection{Heuristic M2BI}
\label{sss:m2bi}
This heuristic is obtained by mixing two Bayes-Independent heuristics
(obtained by setting $k=\infty$ and $l=\infty$ respectively in (\ref{eq:genb1}),
(\ref{eq:genb2})), by taking their mean, to get:
\begin{align} 
	\label{eq:heurm2bi1}
	\mu(\Phi)^{\mathbf{0,r}}_{\Psi,\Psi}(\mathbf{x})  =  \frac{\lambda p} 2 \frac{\xi_{\Psi,\Phi}(||\mathbf{x}||)
		\xi_{\Psi,\Phi}(||\mathbf{x-r}||)}
	{\xi_{\Psi,\Psi}(r)}  +\frac {\lambda p} 2 \xi_{\Psi,\Phi}(||\mathbf{x}||)^{\frac 1 2} \xi_{\Psi,\Phi}(||\mathbf{x-r}||)^{\frac 1 2}
\end{align}
and
\begin{align} 
	\label{eq:heurm2bi2}
	\mu(\Phi)^{\mathbf{0,r}}_{\Psi,\Phi}(\mathbf{x})  =  \frac{\lambda p} 2 \xi_{\Psi,\Phi}(||\mathbf{x}||)
	+\frac {\lambda p} 2
	\frac{ \xi_{\Psi,\Phi}(||\mathbf{x}||)^{\frac 1 2} \xi_{\Phi,\Phi}(||\mathbf{x-r}||)}
	{\xi_{\Psi,\Phi}(r)^{\frac 1 2}}.
\end{align}
\paragraph{Functional equation}
Under Heuristic M2BI, the version of (\ref{eq:rcpmmmixed-glob}) is
\begin{eqnarray}
	\label{eq:recursive-v-2}
	(\beta+\gamma) p
	\xi_{\Phi,\Phi} (r) 
	& =  & p \gamma
	+ (1-p) \xi_{\Psi,\Phi} (r) f(r) + \frac 1 2 \lambda \left(1-p\right) p 
	\xi_{\Psi,\Phi}(r) \int_{\mathbb R^2} \xi_{\Psi,\Phi}(||\mathbf{x}||) f(||\mathbf{x}||)
	d\mathbf{x}\nonumber\\
	& & \hspace{-2cm} + \frac 1 2 \lambda \left(1-p\right) p
	\xi_{\Psi,\Phi}(r)^{\frac 1 2}
	\int_{\mathbb R^2} 
	\xi_{\Psi,\Phi}(||\mathbf{x}||)^{\frac 1 2}
	\xi_{\Phi,\Phi}(||\mathbf{x-r}||) f(||\mathbf{x}||) d\mathbf{x}\nonumber\\
	\beta p \xi_{\Psi,\Phi} (r)  & = &
	(1-p) \gamma \left( \xi_{\Psi,\Psi} (r) -1\right) + \frac 1 2 \lambda \left(1-p\right) p
	\int_{\mathbb R^2} \xi_{\Psi,\Phi}(||\mathbf{x}||)\xi_{\Psi,\Phi}(||\mathbf{x-r}||)
	f(||\mathbf{x}||) d\mathbf{x}.\nonumber\\
	& &\hspace{-2cm} + \frac 1 2 \lambda \left(1-p\right) p
	\xi_{\Psi,\Psi} (r)\int_{\mathbb R^2}
	\xi_{\Psi,\Phi}(||\mathbf{x}||)^{\frac 1 2}\xi_{\Psi,\Phi}(||\mathbf{x-r}||)^{\frac 1 2}
	f(||\mathbf{x}||) d\mathbf{x}.\nonumber\\
\end{eqnarray}
This should again be complemented by
\begin{eqnarray}
	\label{eq:recursivec1-mot}
	p & = & 1- \frac{\beta}{ \lambda 2 \pi \int_{\mathbb R^+} \xi_{\Psi,\Phi} (r) f(r) r dr},
\end{eqnarray}
and
\begin{equation} 
	\label{eq:recursivec2-mot}
	\xi_{\Psi,\Psi} (r) = \frac 1 {\left(1-p\right)^2}
	\left( 1- 
	\left(p\right)^2 \xi_{\Phi,\Phi} (r)
	-2p\left(1-p\right) \xi_{\Psi,\Phi} (r) \right).
\end{equation}

\paragraph{Polynomial equation}
The associated polynomial equations read
\begin{eqnarray}
	(2 \gamma+\beta) pv & = & 2 \gamma p + 2 \alpha (1-p)  w + \beta  p w
	\nonumber
	\\
	\beta p w   & = &  (1-p) \gamma (z-1)
	+ \frac 1 2 \beta  p w
	+\frac 1 2 \beta p z\nonumber\\
	\beta & = & (1-p) \alpha \mu w\nonumber
	\\
	1  & = & (1-p)^2 z + 2p(1-p)w + p^2 v. 
	\label{eq:fantasticm2bi}
\end{eqnarray}
Using the last and the second equations, we can eliminate $z$ to get
\begin{equation}
	\label{eq:fant2}
	\beta p(1-p)^2 w =(1 -2p(1-p)w -p^2v)(2\gamma(1-p)+\beta p)- 2\gamma(1-p)^3.
\end{equation}
This in turn gives
\begin{eqnarray*}
	(2 \gamma+\beta) v(\alpha \mu w - \beta) & = & 2\gamma(\alpha \mu w-\beta) + 2\alpha \beta w+ 
	\beta w (\alpha \mu w -\beta)
	\\
	w (\alpha \mu w-\beta) \beta^2  & = &
	- 2\gamma\beta^2\\
	& & \hspace{-2cm} +
	(\alpha \mu w -\beta +2 \gamma)(w^2\alpha^2\mu^2-(\alpha \mu w -\beta)^2v-2\beta(\alpha \mu w-\beta)w) .
\end{eqnarray*}
When now eliminating $v$ in the last system, we get that $w$ satisfies the degree 4 equation:
\begin{eqnarray}
	& &\hspace{-1cm} (2 \gamma+\beta)\big(
	(\alpha \mu w -\beta +2 \gamma)(w^2\alpha^2\mu^2-2\beta(\alpha \mu w-\beta)w)
	+ w(\alpha \mu w -\beta) \beta^2 - 2\gamma\beta^2\big) \nonumber\\
	= && (\alpha \mu w-\beta)
	(\alpha \mu w -\beta +2 \gamma)\big(
	2\gamma(\alpha \mu w-\beta) + 2\alpha \beta w+ \beta w (\alpha \mu w -\beta)\big).\nonumber\\
\end{eqnarray}

\subsubsection{Heuristic B1I}
\label{sss:b1i}
We arrive at this heuristic by setting $l=k=1$ in (\ref{eq:genb1}), (\ref{eq:genb2}) to get:
\begin{eqnarray} 
	\label{eq:heurb1i1}
	\mu(\Phi)^{\mathbf{0,r}}_{\Psi,\Psi}(\mathbf{x}) & = & \lambda p \frac{\xi_{\Psi,\Phi}(||\mathbf{x}||)^{\frac 2 3}
		\xi_{\Psi,\Phi}(||\mathbf{x-r}||)^{\frac 2 3}}
	{\xi_{\Psi,\Psi}(r)^{\frac 1 3}} 
\end{eqnarray}
and
\begin{eqnarray} 
	\label{eq:heurb1i2}
	\mu(\Phi)^{\mathbf{0,r}}_{\Psi,\Phi}(\mathbf{x}) & = & \lambda p  \frac{\xi_{\Psi,\Phi}(||\mathbf{x}||)^{\frac 2 3}
		\xi_{\Phi,\Phi}(||\mathbf{x-r}||)^{\frac 2 3}}
	{\xi_{\Psi,\Phi}(r)^{\frac 1 3}}.
\end{eqnarray}
We then proceed exactly as we did in Section \ref{sss:m2bi} to arrive at polynomial equations for the B1I heuristic that are as follows:
\begin{eqnarray}
	(\gamma+\beta) pv & = &  \gamma p +  \alpha (1-p)  w+
	\beta p  v^{\frac 2 3} w^{\frac 1 3}  \nonumber \\
	\beta p w   & = &  (1-p) \gamma (z-1)
	+ \beta p z^{\frac 2 3} w^{\frac 1 3}\nonumber\\
	\beta & = & (1-p) \alpha \mu w\nonumber
	\\
	1  & = & (1-p)^2 z + 2p(1-p)w + p^2 v. 
	\label{eq:fantasticb1i}
\end{eqnarray}
\subsection{Numerical Estimates of the Fraction of Infected Points}
In this section, we compare the fraction of infected points as estimated by our integral and polynomial systems
and as obtained by discrete event simulation. For the latter, the system evolves on a square with edges
wrapped around to form a torus so as to avoid border effects. Points move according to the mobility model
defined above. The infection function is $f(r) = \alpha 1_{ r\leq a }$.

Tables \ref{table:sick_proportion-vela2}, \ref{table:sick_proportion-vela-below-mediumr}, and
\ref{table:sick_proportion-vela-below-smallr} study the effect of mobility for a variety of scenarios.

\begin{table}[h]
	\centering
	\begin{tabular}{||c|c|c|c|c|c||} 
		\hline
		$\gamma$ & 0 & .2 & 1 & 5 & $\infty$ \\
		\hline\hline
		$p_{\mathrm {sim}}$   & 0.26  & 0.28  & 0.29  & 0.33  & 0.36 \\
		\hline
		$p_{\mathrm {p-b1i}}$  & 0.313  & 0.315  & 0.323  & 0.341  &  0.363 \\
		\hline
		$p_{\mathrm {p-b1g1}}$ & 0.325  & 0.326  & 0.331  & 0.343  &  0.363 \\
		\hline
		$p_{\mathrm {p-m2bi}}$ & 0.328  & 0.328  & 0.329  & 0.341  & 0.363 \\
		\hline
		$p_{\mathrm {p-m}\infty\mathrm{bi}}$ & 0.33  & 0.33  & 0.33  & 0.34  & 0.36 \\
		\hline
		$p_{\mathrm {f-h0}}$  & 0.23  & 0.28  & 0.29  & 0.32  &  0.36 \\
		\hline
		$p_{\mathrm {p-h0}}$  & 0.23  &  0.25 & 0.27  & 0.32  & 0.36  \\
		\hline
	\end{tabular}
	\caption{Effect of mobility. Way above Boolean-percolation.
		Fraction of infected points ($p$) obtained by simulation,
		the functional fixed point equation and the polynomial equation.
		This is for $\beta=8$, $a=2$, $\lambda=1$ and $\alpha=1$, so that $\mu\sim 12.56$.
		The agreement with simulation is good.
		Both b1 and m2bi slightly overestimate $p$, with b1i being here a bit closer than the others.}
	\label{table:sick_proportion-vela2}
	
	\hspace{.4cm}
	
	\centering
	\begin{tabular}{||c|c|c|c|c|c||} 
		\hline
		$\gamma$ & 0+ & 0.1 & 1 & 10  & $\infty$  \\
		\hline\hline
		$p_{\mathrm {sim}}$   &       &       &  0.22   &  0.31  & 0.36 \\
		\hline
		$p_{\mathrm{p-b1i}}$  & 0.176 & 0.188  &  0.253  &  0.342 & 0.363 \\
		\hline
		$p_{\mathrm{p-b1g1}}$ & 0.205 & 0.213  &  0.264  &  0.348 & 0.363 \\
		\hline
		$p_{\mathrm{p-m2bi}}$ & 0.245 & 0.240 &  0.254  &  0.336 & 0.363 \\
		\hline
		$p_{\mathrm {p-m}\infty\mathrm{bi}}$  & 0.26  &  0.26 &  0.27   &  0.34  & 0.36 \\
		\hline
	\end{tabular}
	\caption{Effect of mobility. Below Boolean percolation, medium recovery rate.
		Fraction of infected points ($p$) obtained by simulation,
		the functional fixed point equation and the polynomial equation.
		This is for $\beta=2$, $a=1$, $\lambda=1$ and $\alpha=1$,
		so that $\mu\sim 3.14$. Simulation is inefficient at low speeds.
		The agreement with simulation is again good when available.
		Both b1 and m2bi slightly overestimate $p$ and provide similar results.}
	\label{table:sick_proportion-vela-below-mediumr}
\end{table}

\begin{table}[t]
	\centering
	\begin{tabular}{||c|c|c|c|c|c|c|c||} 
		\hline
		$\gamma$              & 0+   &0.01 & 0.1  & .2    & 1     & 5    & 100   \\
		\hline\hline
		$p_{\mathrm {sim}}$   &      &     &      & 0.54  & 0.61  & 0.66 & 0.68  \\
		\hline
		$p_{\mathrm{p-b1i}}$  & 0.478&     & 0.503& 0.523 & 0.599 & 0.657& 0.680  \\
		\hline
		$p_{\mathrm{p-b1g1}}$ & 0.530&     & 0.544& 0.557 & 0.609 & 0.658& 0.680  \\
		\hline
		$p_{\mathrm{p-m2bi}}$ &0.523 &     &0.538 & 0.551 & 0.605 & 0.656& 0.680 \\
		\hline
		$p_{\mathrm {p-m}\infty\mathrm{bi}}$  &      &0.54 & 0.55 & 0.56  & 0.61  & 0.66 & 0.68  \\
		\hline
	\end{tabular}
	\caption{Effect of mobility. Below Boolean-percolation, low recovery rate. Fraction of infected points
		($p$) obtained by simulation,
		the functional fixed point equation and the polynomial equation.
		This is for $\beta=1$, $a=1$, $\lambda=1$ and $\alpha=1$,
		so that $\mu\sim 3.14$.  The agreement with simulation is again good when available
		and again provide similar results.}
	\label{table:sick_proportion-vela-below-smallr}
\end{table}

Instances of pair correlation functions obtained from the functional
equation numerical scheme associated with Heuristic f-mb2i are given in Figure \ref{fig:pcg1}.
All numerical solutions of integral equations lead to similar pictures,
which justifies the polynomial heuristics.

\begin{figure}[h!]
\begin{center}
\includegraphics[width=.32\textwidth]{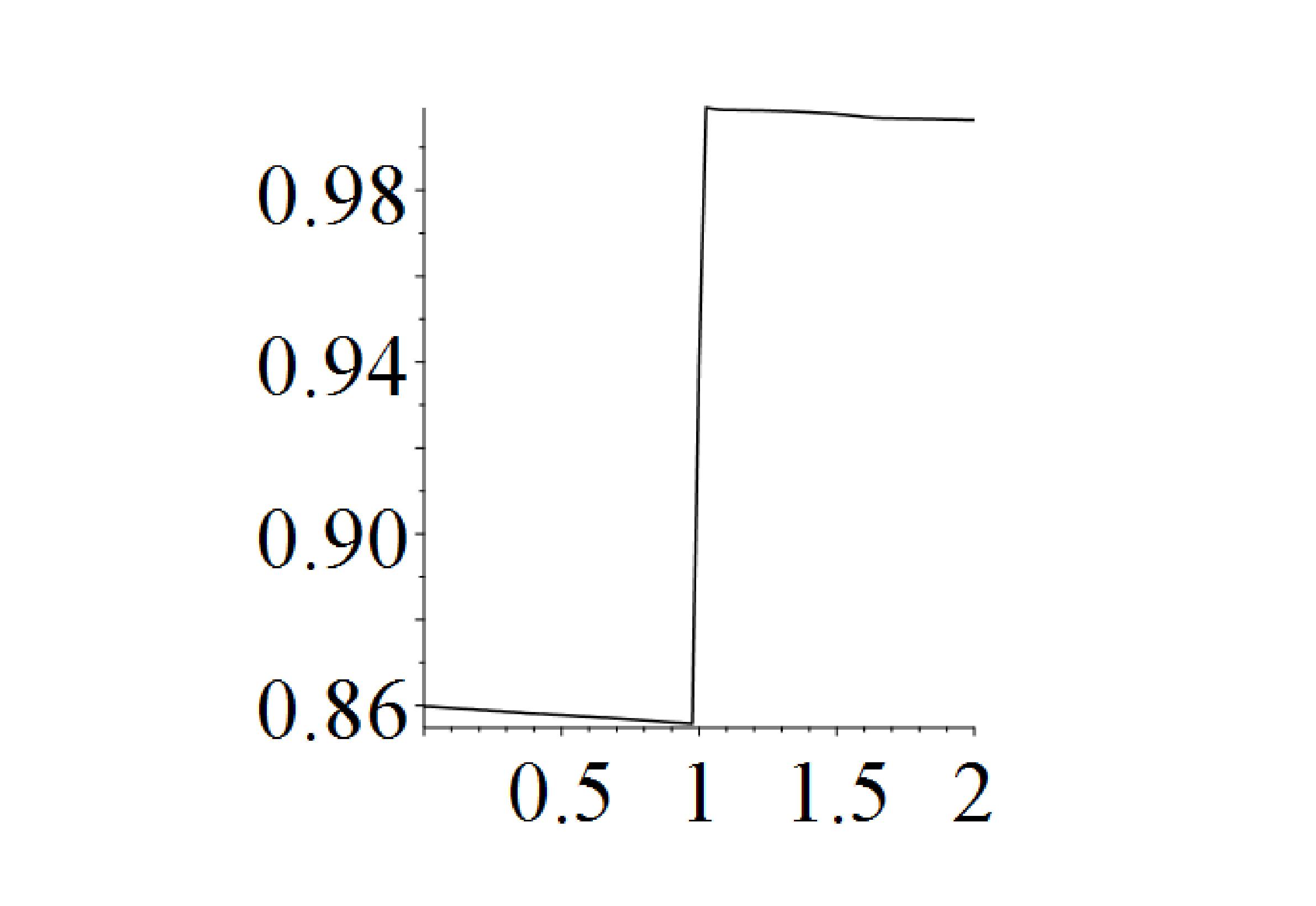}
\includegraphics[width=.32\textwidth]{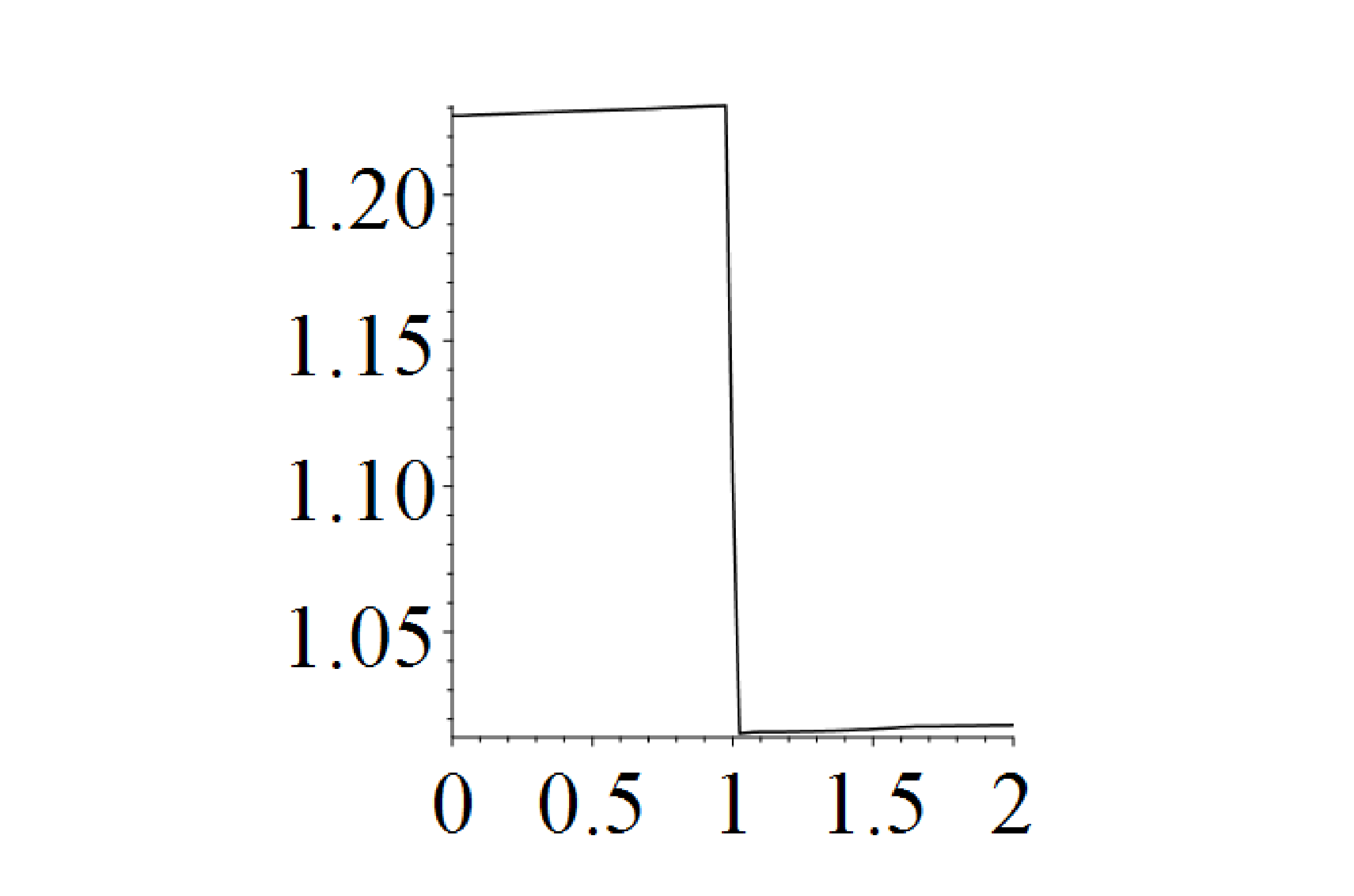}
\includegraphics[width=.32\textwidth]{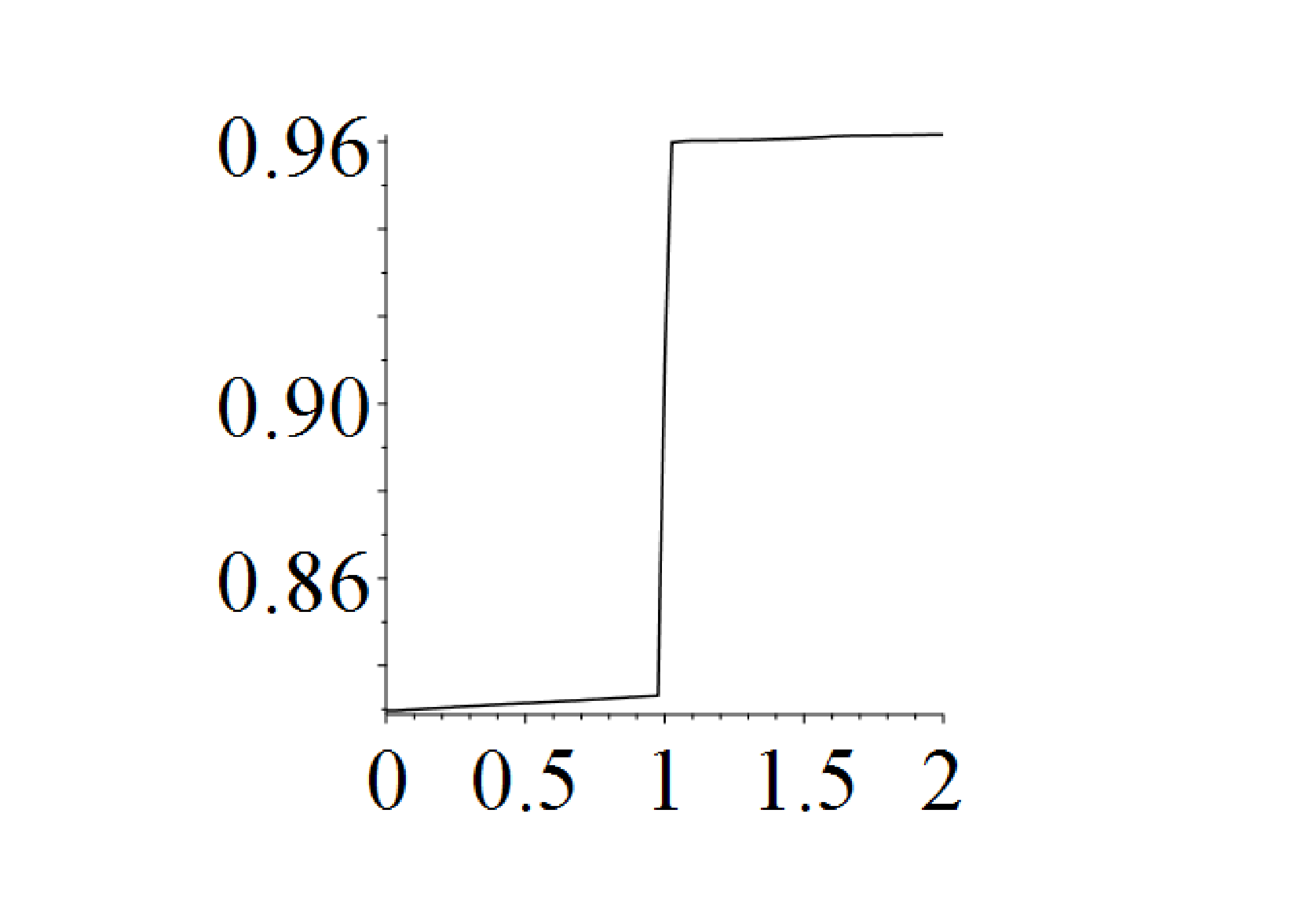}
\end{center}
\caption{Left: the $\xi_{\Psi,\Phi}(r)$ function;
Center: the $\xi_{\Phi,\Phi}(r)$ function;
Right: the $\xi_{\Psi,\Psi}(r)$ function.
Here, $\beta=1$, $a=1$, $\lambda=1$, $\gamma=1$, and $\alpha=1$.
}
\label{fig:pcg1}
\end{figure}
We now briefly define a concept relevant to the subsequent sections. Boolean percolation
is defined as the existence of an infinite connected component of the random graph that 
the epidemic evolves on. A well-known result for random geometric graphs states that such
a component exists if and only if  $\mu = \lambda \pi a^2 > \mu_*$, and the approximate
value of this percolation threshold is known to be $\mu_*\approx 4.512$ (\cite{fransMeester}, chapter 2). 
Percolation is relevant for our discussions because, below percolation ($\mu < \mu_*$)
and in the no-motion case ($\gamma=0$), there is no steady-state where $p>0$ 
- this is because the contact process a.s. dies out eventually on all finite connected graphs.
Above percolation and in the no-motion case, however, there is a possibility of a non-degenerate steady state. 

\subsection{A Tentative Phase Diagram}
\label{ss:atpd}
In this subsection, $\alpha$ is fixed (for instance taken equal to 1, without loss of generality).
The polynomial systems described above lead to phase diagrams.
The simplest instance is the $(\mu,\gamma)$-phase diagram, which is predicted by the
monotonicity properties of the SIS dynamics:
there exists a function $\beta_c=\beta_c(\mu,\gamma)$, which will be
referred to as the {\em extinction-survival critical recovery rate} such that 
there is survival for all $\beta<\beta_c$ and extinction above.
The main novelty here is an explicit expression for the 
function $\beta_c(\mu,\gamma)$, which is derived from the polynomial systems discussed
above, and more precisely from the analysis of properties of the roots of these equations.

We will also discuss the $(\mu,\beta)$-phase diagram. This diagram is meant to investigate
the influence of $\gamma$. The analysis of certain roots of our polynomial
equations lead to the definition of
a partition of the $(\mu,\beta)$ positive orthant in 3 disjoint regions:
\begin{itemize}
	\item A {\em safe region} where the epidemic is always extinct regardless of the positive motion rate.
	This region is the wedge $\beta>\alpha \mu$.
	\item 
	An {\em unsafe region} which is the geometric complement of the latter; in this region, there
	are motion rates such that the epidemic survives.
\end{itemize}
The unsafe region can in turn be partitioned into 
a {\em motion-insensitive} (UMI) and a {\em motion-sensitive} (UMS) region:
\begin{itemize}
	\item In the UMI region, the epidemic {\em always survives}, regardless of the positive motion rate.
	This region is wedge-like too, of the form $\mu>\mu_0,\ \beta<\beta_0(\mu)$. Here
	$\mu_0$ is an absolute constant smaller than the percolation threshold $\mu_*$, and
	$\beta_0(\cdot)$ is a function to be specified, which will be referred to as the
	{\em motion-sensitivity critical recovery rate}.
	\begin{itemize}
		\item In the part of the UMI region where $\mu<\mu_*$, the epidemic is extinct for
		0 motion and survives for all non-zero motion rates.
		\item In the part of the UMI region where $\mu>\mu_*$, the epidemic survives for all
		motion rates, including 0 motion.
	\end{itemize}
	\item In the UMS region, the epidemic survives if motion is {\em low or high enough}.
	More precisely there exist functions $\gamma_c^+(\mu,\beta)$
	(the upper extinction-survival critical motion rate)
	and $\gamma_c^-(\mu,\beta)$
	(the lower extinction-survival critical motion rate),
	to be specified, such that there is survival if $\gamma<\gamma_c^-$
	or $\gamma>\gamma_c^+$, and extinction otherwise.
	This region is a strip-like region of the form $\beta_0(\mu)<\beta <\alpha\mu$,
	with $\beta_0(\mu)=0$ when $\mu<\mu_0$.
	\begin{itemize}
		\item 
		In the part of the UMS region where $\mu<\mu_*$, the epidemic is extinct for 0 motion, survives for small
		enough motion rates, is extinct for intermediate motion rates, and survives for high enough motion rates.
		\item 
		In the part of the UMS region above $\mu_*$, the epidemic survives for no and small enough motion rate,
		is extinct for intermediate motion rates, and survives for high enough motion rates.
	\end{itemize}
\end{itemize}
This last phase diagram is depicted in Figure \ref{fig:basphasdiag}.

As already mentioned, these phase diagrams are obtained when studying certain
roots of the polynomial systems.
The exact values of the constants and functions introduced above
diagram depend on the heuristic, but the global picture is the same for all
in spite of numerical discrepancies. 
The picture that emerges from this analysis is consistent for all heuristics
and partially substantiated by simulation.
Below, we first proceed with the analysis of the roots of the polynomial systems
and then discuss the simulation validation.

\subsection{Critical Values of Interest}

It follows from the structural results of Section \ref{sec:graprep} that, for all $\gamma>0$
and $\mu$, there is a critical value of $\beta$, a constant
$\beta_c=\beta_c(\gamma,\mu)$ {\em less than or equal to $\alpha \mu$}, such that there 
is extinction (there is no positive invariant measure) if $\beta>\beta_c$ and survival (there is
a positive invariant measure) if $\beta<\beta_c$.

In contrast, there is no mathematical argument showing the existence
of a critical value $\gamma_c$ above which survival would hold and below which there
would be extinction. In fact the analysis which follows suggests that
this is not the case. However, when decreasing $\gamma$ from $\infty$ (where
the epidemic should survives if $\mu\alpha > \beta$), if we are below Boolean
percolation, there ought to be a $\gamma_c^+$ which is the largest $\gamma$ such that above
this value the epidemic survives but not immediately below.

We now see through two examples how the polynomial heuristics can be used to
derive the phase diagram and to estimate the critical values that form the boundaries
between the regions of this phase diagram. 
The general idea is to leverage the fact substantiated by simulation that $p$ is continuous at the
boundary of the regions of the phase diagram, that the density of the positive invariant
maximal measure which exists in the survival region tends to 0 when getting close
to the boundary between this region and the extinction region.

\subsubsection{M2BI}
\paragraph{$(\mu,\beta)$-phase diagram}
A direct analysis of the roots of (\ref{eq:fantasticm2bi}) around $p\sim 0$ gives that
$p\sim 0$ is only possible if
\begin{equation}
	\label{eq:fant-mb2-mob}
	8(\mu\alpha-\beta)\gamma^2
	+2\beta(3(\mu\alpha-\beta) -2\alpha) \gamma
	+\beta^2(\mu\alpha-\beta)=0.
\end{equation}
This is easily obtained when showing that the first equation in (\ref{eq:fantasticm2bi}) implies that 
$$v(p)p\to_{p\to 0} \frac {2\beta}{\mu(2\gamma+\beta)}$$
and by using this fact in (\ref{eq:fant2}) when making an expansion in $p$ close to 0.

When looking at (\ref{eq:fant-mb2-mob}) as a quadratic in $\gamma$, we get that
if there exists a positive (resp. negative) solution, then the other solution is positive (resp. negative).
For real solutions to exist, it is necessary and sufficient to have either
\begin{equation}
	\label{eq:fant4mb2}
	\mu \alpha -\beta \le \frac{2\alpha }{3+\sqrt{8}}\sim \alpha 0.343 
\end{equation}
or
\begin{equation}
	\label{eq:fant4b}
	\mu {\alpha} -\beta \ge \frac{2\alpha }{3-\sqrt{8}}\sim {\alpha}11.65.
\end{equation}
This is easily obtained when studying the discriminant 
of (\ref{eq:fant-mb2-mob}). If these inequalities are not satisfied, then this
discriminant is negative, so that there is no real valued $\gamma$ solving
(\ref{eq:fant-mb2-mob}).
If (\ref{eq:fant4mb2}) holds, then there are two positive roots.
If (\ref{eq:fant4b}) holds, then there are two negative roots, which for us is
equivalent to no root.

Below, we focus on the region (\ref{eq:fant4mb2}), which is the only one in which criticality is possible.
Let
\begin{equation}
	\label{eq:mu0mb2i}
	\mu_0=\alpha\eta:= \alpha \frac{2}{3+\sqrt{8}}.
\end{equation}
If $\mu<\mu_0$, then (\ref{eq:fant4mb2}) holds
so that there exist two positive $\gamma$ solving (\ref{eq:fant-mb2-mob}),
say $0<\gamma_c^-<\gamma_c^+$.

If $\mu >\mu_0$, then there exists a $\beta_0(\mu,\alpha)$ defined by
\begin{equation}
	\label{eq:beta0mb2}
	\beta_0 =\mu\alpha - \eta\alpha,
\end{equation}
such that if $\beta<\beta_0$,
then the quadratic equation (\ref{eq:fant-mb2-mob}) has no real root in $\gamma$, which
means that any positive motion results in survival,\footnote{
	note that this is consistent with what we see in
	Tables \ref{table:sick_proportion-vela-below-mediumr} and
	\ref{table:sick_proportion-vela-below-smallr}, where the discrepancy between $\mu$ and $\beta/\alpha $
	is above 0.343 and where the epidemics seems to persist for all $\gamma >0$.}
whereas if $\beta > \beta_0$, then the roots
$0<\gamma_c^-<\gamma_c^+$ exist.
The upper phase transition w.r.t. $\gamma$ takes place at $\gamma_c^+$:
Above $\gamma_c^+$, there is survival, and below $\gamma_c^+$, there is extinction.
Indeed, we conjecture that the system has a positive fraction of infected points for
infinite $\gamma$, and the threshold of interest is hence that obtained when decreasing
$\gamma$ from infinity and looking at the largest $\gamma$ above which the epidemic survives.
We have
\begin{equation}
	\label{eq:fantlmb2}
	\gamma_c^+(\mu,\beta)=\beta \frac{2\alpha -3(\mu\alpha-\beta)+
		\sqrt{(2\alpha-3(\mu\alpha-\beta))^2 -8(\mu\alpha-\beta)^2}} {8(\mu\alpha-\beta)}
\end{equation}
and
\begin{equation}
	\label{eq:fantlmb2min}
	\gamma_c^-(\mu,\beta)=\beta \frac{2\alpha -3(\mu\alpha-\beta)-
		\sqrt{(2\alpha-3(\mu\alpha-\beta))^2 -8(\mu\alpha-\beta)^2}} {8(\mu\alpha-\beta)}.
\end{equation}
Since there is survival for $\gamma>\gamma_c^+$,
it ought to be that for $\gamma\in (\gamma_c^-,\gamma_c^+)$, there is extinction. By the same argument,
for $\gamma<\gamma_c^-$, there is survival.

Note that if $\beta=\beta_0$, 
then the corresponding value of $\gamma_c^+=\gamma_c^-$ is
\begin{equation}
	\label{eq:gamma0mb2}
	\gamma_0(\mu)= \frac{\beta_0(2\alpha-3(\mu\alpha-\beta_0))}{8(\mu\alpha-\beta_0)}
	= \frac{(\mu\alpha-\eta\alpha)(2\alpha-3\eta\alpha)}{8\eta\alpha}
	= \frac{\alpha (\mu-\eta)(2-3\eta)}{8\eta}.
\end{equation}
The last function is just an affine function in $\mu$ with positive slope. 
It is strictly positive for $\mu>\mu_0 \sim 0.343.$

Hence we get the following m2bi $(\mu,\beta)$-phase diagram:
\begin{Res}
	Assume that $\mu\alpha >\beta$. 
	\begin{itemize}
		\item In the motion-subcritical regions, namely for all $\mu <\mu_0$, with $\mu_0$ given by (\ref{eq:mu0mb2i}),
		the system is motion-sensitive and
		there is survival for all values of $\gamma$ larger $\gamma_c^+$ or smaller than $\gamma_c^-$,
		and extinction for $\gamma$ between these two values, with $\gamma_c^+$
		and $\gamma_c^-$ given by (\ref{eq:fantlmb2}) and (\ref{eq:fantlmb2min}) respectively.
		\item In the motion-supercritical region, namely for all $\mu > \mu_0$, 
		\begin{itemize}
			\item if $\beta < \beta_0$,
			with $\beta_0$ defined in (\ref{eq:beta0mb2}),
			the system is motion-insensitive in that there
			is survival for all positive values of $\gamma$;
			\item if $\beta > \beta_0$, then the system is motion sensitive.
		\end{itemize}
	\end{itemize}
\end{Res}

Instances of the functions $\beta\to \gamma_c^+(\mu,\beta)$ and $\beta\to \gamma_c^-$ are depicted 
in Figure \ref{fig:gammacritmb2}.

\begin{figure}[h!]
	\begin{center}
		\includegraphics[width=.45\textwidth]{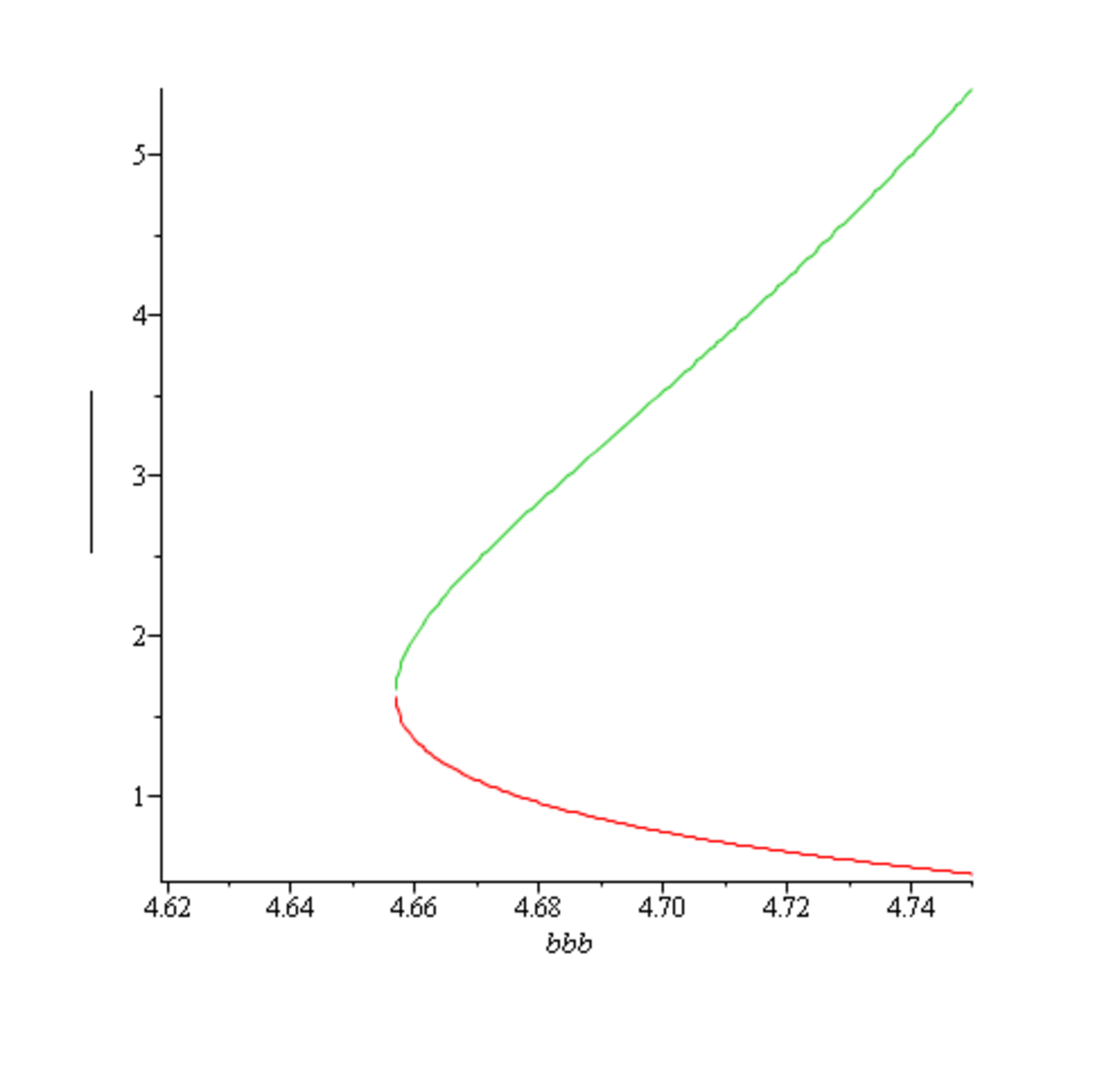}
		\includegraphics[width=.54\textwidth]{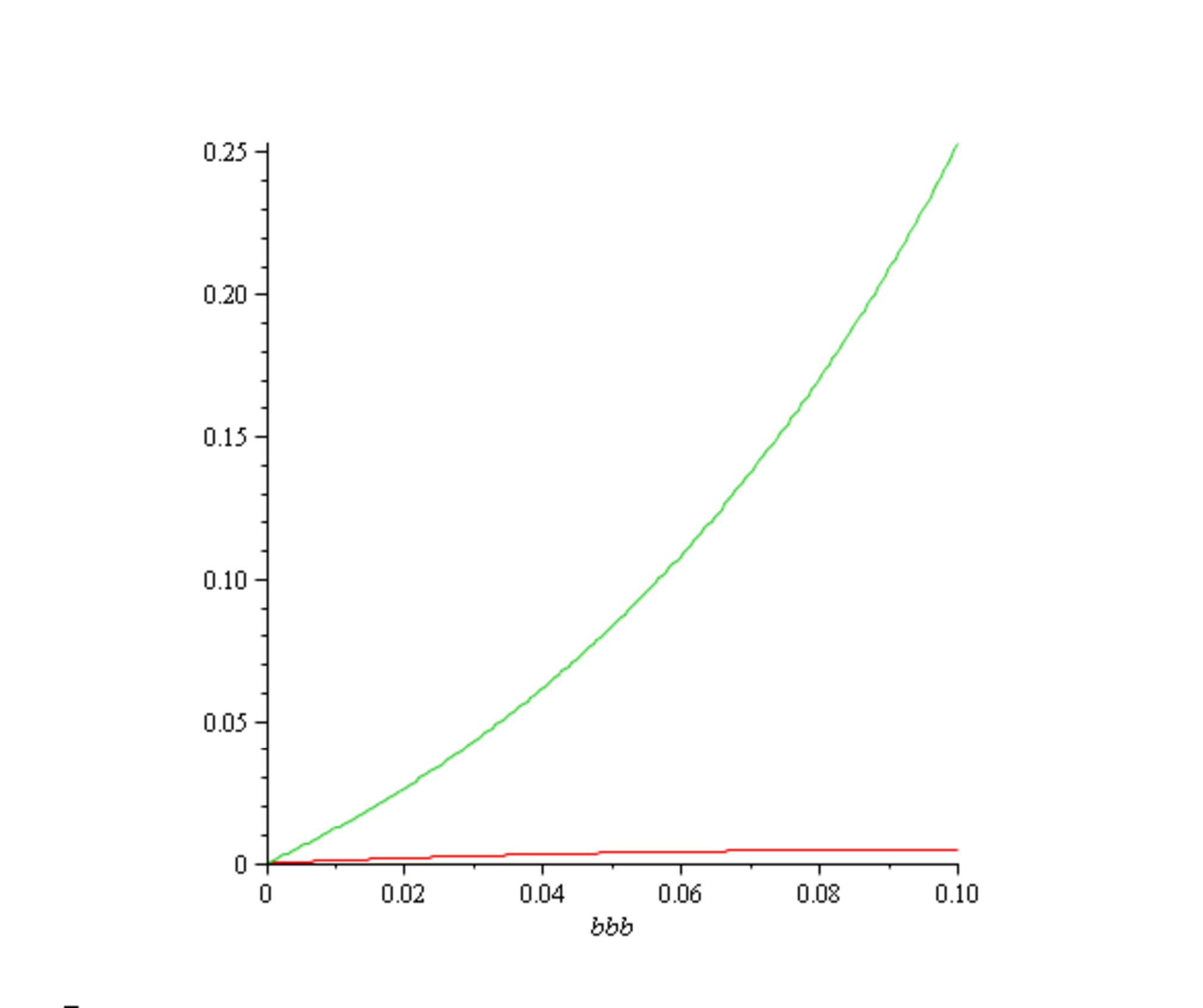}
	\end{center}
	\caption{The p-m2bi $\beta\to \gamma_c^+(\mu,\beta)$ (green) and $\beta\to\gamma_c^-(\mu,\beta)$ (red) functions.
		The $x$-axis represents $\beta$. 
		Left: $\alpha=1$ and $\mu=5>\mu_0$; there is survival
		on the left of the curve and extinction on the right.
		If $\beta<\beta_0$, where $\beta_0$ is the largest $c$ such that
		the line $\beta=c$ does not intersect the $\gamma_c$ curves, the epidemic is motion-insensitive.
		Right: $\alpha=1$ and $\mu=.25<\mu_0$; there is survival above the green curve
		and below the red one. There is extinction between the two curves. There is no 
		positive $\beta$ making the epidemic motion-insensitive. Notice that the $\gamma_c^-$ function
		is increasing for small $\beta$. It reaches a maximum and then decreases to 0
		when $\beta$ is large.}
	\label{fig:gammacritmb2}
\end{figure}

\paragraph{$(\mu,\gamma)$-phase diagram}
One can also use the same method to analyze the critical function $\beta_c$.
There are two ways of evaluating this quantity.

The first one is obtained from (\ref{eq:fant-mb2-mob}). The $\beta_c$ function
satisfies the following polynomial equation of degree 3 in $\beta$:
\begin{equation}
	\label{eq:fant-mb2-mob-betaa}
	\beta^3 +
	\beta^2 (6\gamma-\alpha\mu) +\beta2\gamma(2\alpha -3\mu\alpha+4\gamma)
	-8\mu\alpha\gamma^2=0.
\end{equation}
There is numerical evidence that this degree three equation has a single positive root
that will be denoted $\beta_c(\mu,\gamma)$.

\begin{figure}[h]
	\begin{center}
		\includegraphics[width=.45\textwidth]{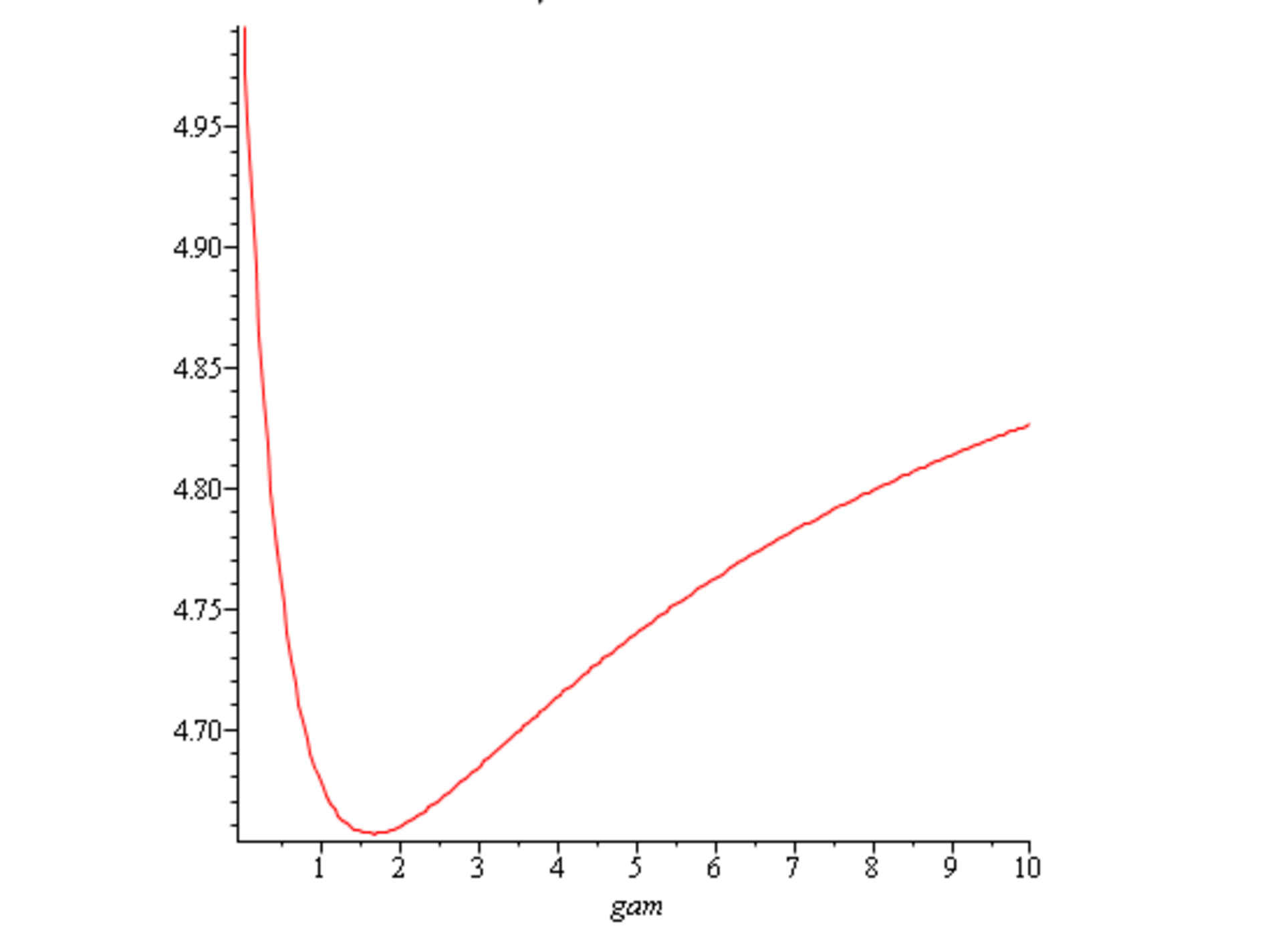}
		\includegraphics[width=.47\textwidth]{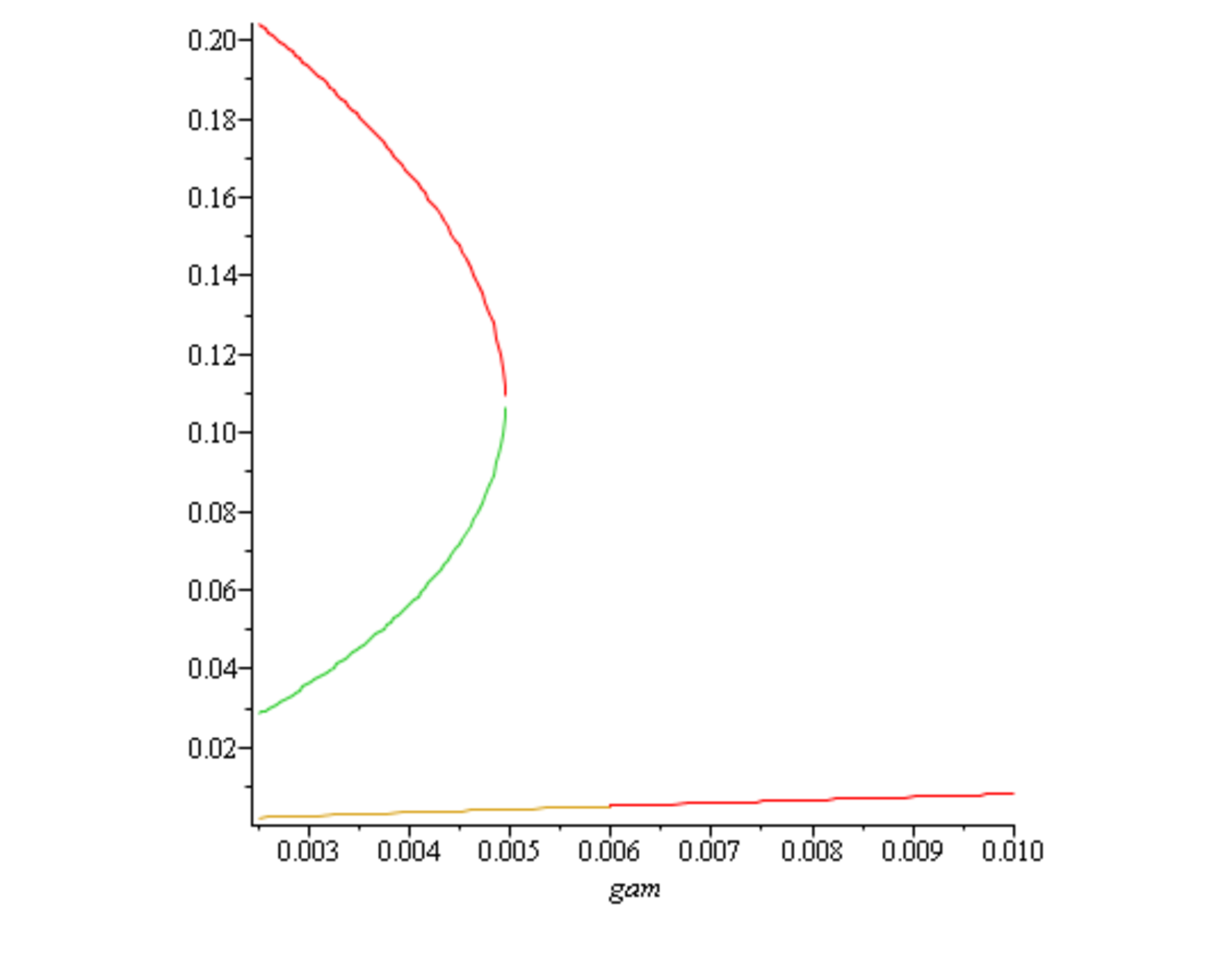}
	\end{center}
	\caption{
		The p-m2bi $\gamma\to \beta_c(\mu,\gamma)$ function for $\alpha=1$. On the $x$ axis, the variable is $\gamma$.
		Left: $\mu=5$; there is extinction above the curve and survival below.
		Right: $\mu=.25$. The three roots are jointly represented. The interpretation
		of $\beta_c$ as the pseudo-inverse of $\gamma_c$ suggests that the $\beta_c$ 
		function should be the upper envelope of these curves (this upper envelope is
		not depicted). To see this, use the
		fact that the function $\gamma_c^-$ is non monotonic in this case.
		The discontinuity is at the point where $\gamma_c^-$ reaches its maximum.
	}
	\label{fig:betaa}
\end{figure}

The second approach consists in devising the local "inverse" of the $\beta\to\gamma_c(\beta)$
function discussed above. As illustrated by the left part of Figure \ref{fig:gammacritmb2}
for $\gamma$ small, the inverse function $\gamma\to \beta_c(\gamma)$ should be decreasing.
It should then reach a minimum at $\gamma=\gamma_0$ and then increase to $\mu\alpha$ for $\gamma$ large.
This is in line with what we see on Fig. \ref{fig:betaa}.
Hence we get the following p-m2bi $(\mu,\gamma)$-phase diagram:

\begin{Res}
	There is survival for all $\beta<\beta_c$ and extinction above with 
	$\beta_c=\beta_c(\alpha,\mu,\gamma)$ solution of (\ref{eq:fant-mb2-mob-betaa}).
	\begin{itemize}
		\item If $\mu<\mu_0$, then the function $\gamma\to \beta_c(\alpha,\mu,\gamma)$ is discontinuous.
		\item If $\mu>\mu_0$, then the function $\gamma\to \beta_c(\alpha,\mu,\gamma)$ is continuous.
	\end{itemize}
\end{Res}

Instances of the function $\gamma\to \beta_c(\gamma)$
defined through (\ref{eq:fant-mb2-mob-betaa}) are plotted in Figure \ref{fig:betaa}
where we see that it is {\em not monotonic} in $\gamma$ in general.

\subsubsection{B1I}
\paragraph{$(\mu,\beta)$-phase diagram}
By an analysis of the polynomial system similar to that for m2bi above, we get that
$p\sim 0$ is only possible if
\begin{equation}
	\label{eq:fant-b1-basic}
	2(\mu\alpha-\beta)\gamma^2
	+(2 \beta (\mu\alpha-\beta) +\beta^2(\rho-1) -\beta \alpha) \gamma
	+\beta^3(\rho-1)=0,
\end{equation}
with $\rho=\left(\frac{\alpha \mu}{\beta}\right)^{\frac 2 3}>1$.
By looking at the discriminant of this quadratic, we get that
a necessary (but not sufficient) condition for a positive real root to exists is that
$$\mu\alpha< \beta+ 1+ \frac{\alpha}{2}.$$
In words, $\mu\alpha$ has to be close enough to $\beta$.

Here is a more precise analysis.
There are two real roots (which are necessarily both positive) iff
\begin{equation*} \Delta:=(2 \beta (\mu\alpha-\beta) +\beta^2(\rho-1) -\beta \alpha)^2 -8
	(\mu\alpha-\beta)\beta^3(\rho-1)>0
\end{equation*}
and in this case, 
\begin{eqnarray}
	\label{eq:fant-b1-gammac}
	\gamma_c^+(\mu,\beta)& =& \frac
	{ \beta(\alpha- 2 (\mu\alpha-\beta)) -\beta^2(\rho-1) +\sqrt{\Delta} }
	{4(\mu\alpha-\beta) }
\end{eqnarray}
and
\begin{eqnarray}
	\label{eq:fant-b1-gammacmin}
	\gamma_c^-(\mu,\beta)& =& \frac
	{ \beta(\alpha- 2 (\mu\alpha-\beta)) -\beta^2(\rho-1) -\sqrt{\Delta} }
	{4(\mu\alpha-\beta) }.
\end{eqnarray}
These two roots are plotted in Fig. \ref{fig:gammac-bli-mu5}.

\begin{figure}[h!]
	\begin{center}
		\includegraphics[width=.45\textwidth]{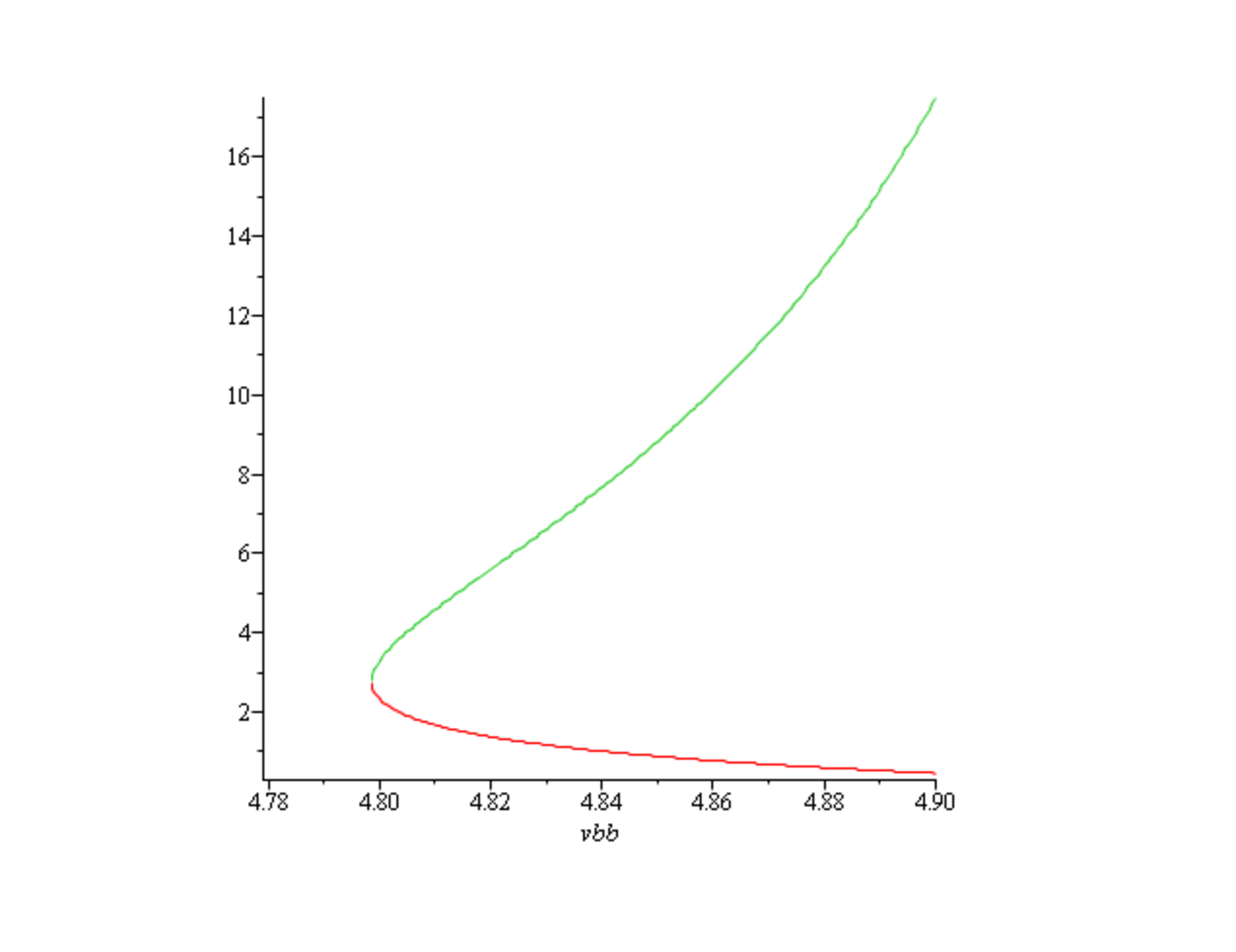}
		\includegraphics[width=.44\textwidth]{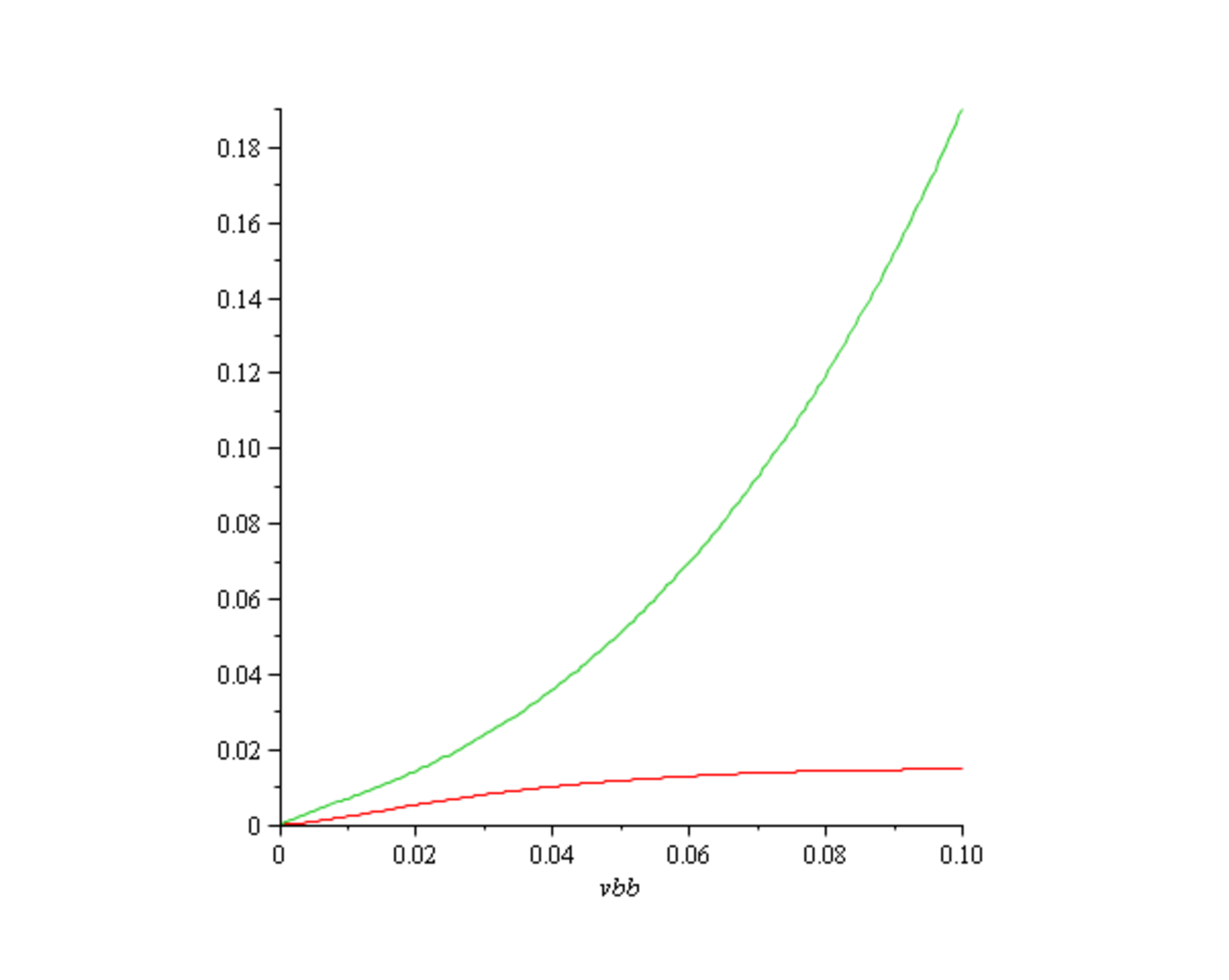}
	\end{center}
	\caption{The p-b1i functions $\beta\to \gamma_c^+(\mu,\beta)$ (in green)
		and $\beta\to \gamma_c^-(\mu,\beta)$ (in red) for $\alpha=1$. Left: $\mu=5$. Right: $\mu=.25$.}
	\label{fig:gammac-bli-mu5}
\end{figure}

For $\gamma>\gamma_c^+$ (in particular for $\gamma$ very large), there is survival.
For $\gamma_c^-<\gamma <\gamma_c^+$, there
is extinction, and when $0<\gamma<\gamma_c^-$, there is survival again.
We will give an interpretation of these phenomena below.

There exists a $\mu_0$ (for $\alpha=1$, $\mu_0\sim 0.263$ - for p-b1g1 $0.306$)
such that (i) for $\mu <\mu_0$, and all $\beta<\alpha\mu$, $\Delta >0$ 
and hence $\gamma_c^+$ and $\gamma_c^-$ exist,
and (ii) for all values of $\mu>\mu_0$, there exists a minimal $\beta$, say $\beta_0=\beta_0(\alpha\mu)$,
for $\Delta $ to be positive and hence for non-degenerate $\gamma_c^+$ and $\gamma_c^-$ to exist.
In the latter case, the function $\alpha\mu\to \beta_0(\alpha\mu)$ is solution of 
\begin{equation} (2 \beta (\mu\alpha-\beta) +\beta^2(\rho-1) -\beta \alpha)^2 =8
	(\mu\alpha-\beta)\beta^3(\rho-1).
	\label{eq:fant-b1-exist-gammac}
\end{equation}
The $\beta_0$ function is plotted in Figure \ref{fig:betagamma0}.
\begin{figure}[h!]
	\begin{center}
		\includegraphics[width=.45\textwidth]{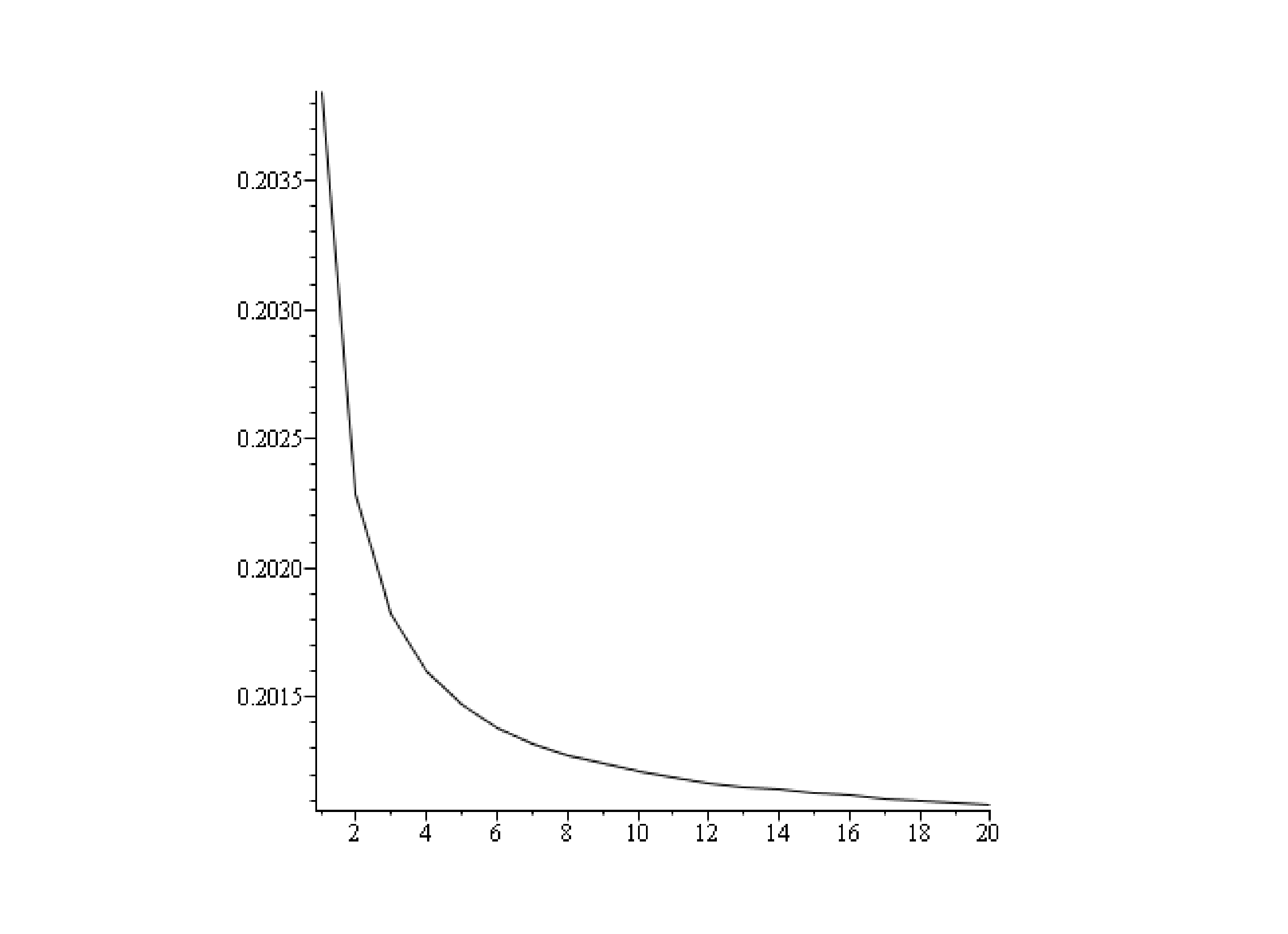}
		\includegraphics[width=.39\textwidth]{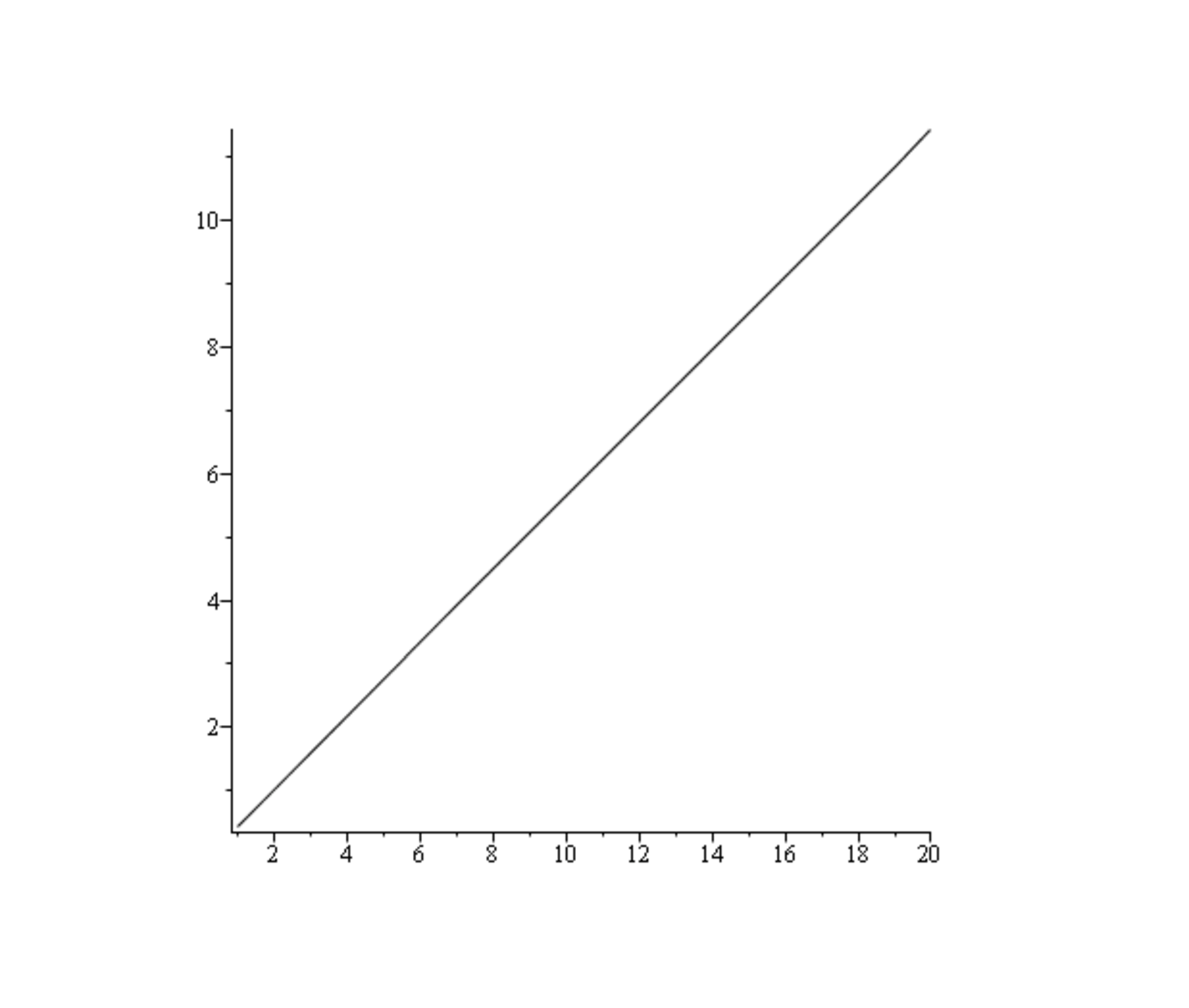}
	\end{center}
	\caption{Left: The p-b1i function $\mu\rightarrow\beta_0(\mu)$ for $\alpha=1$.
		Right: The p-b1i function $\gamma_0(\mu)$ for $\alpha=1$.}
	\label{fig:betagamma0}
\end{figure}
At $\beta_0$, $\gamma_c^+=\gamma_c^-:=\gamma_0$, with $\gamma_0$ strictly positive.
The function $\gamma_0$ is $\gamma_c$ for $\beta=\beta_0$.
It is plotted in Figure \ref{fig:betagamma0}.

We summarize this in the following p-b1i $(\mu,\beta)$-phase diagram:
\begin{Res}
	Assume that $\mu\alpha >\beta$.
	\begin{itemize}
		\item In the motion-subcritical region $\mu<\mu_0$, 
		the discriminant is positive and there is survival
		for values of $\gamma$ larger than
		$\gamma_c^+(\alpha \mu,\beta)$ given by (\ref{eq:fant-b1-gammac}) or smaller than $\gamma_c^-$
		(\ref{eq:fant-b1-gammacmin})
		and extinction for $\gamma$ between there two values.
		\item In the motion-supercritical region $\mu>\mu_0$, 
		\begin{itemize}
			\item If $\beta <\beta_0(\alpha\mu)$, with $\beta_0$ solution of
			(\ref{eq:fant-b1-exist-gammac}), we have motion-insensitivity;
			\item If $\beta> \beta_0(\alpha\mu)$, then there motion sensitivity as above.
		\end{itemize}
	\end{itemize}
\end{Res}

\paragraph{$(\mu,\gamma)$-phase diagram}
Consider now the function $\gamma\to \beta_c(\mu,\gamma)$.
Locally, this is just the inverse (in the sense of increasing functions) of the last function. 
It can also be obtained as a solution of (\ref{eq:fant-b1-basic}).
More precisely, we have the following p-b1i $(\mu,\gamma)$-phase diagram:

\begin{Res}
	There is survival for all $\beta<\beta_c$ and extinction above with 
	$\beta_c=\beta_c(\alpha,\mu,\gamma)$ solution of (\ref{eq:fant-b1-basic})
	which can be seen as a degree 9 polynomial in $\beta$. 
\end{Res}
This function is depicted in Figure \ref{fig:betac-b1-ex}.

\begin{figure}[h!]
	\begin{center}
		\includegraphics[width=.44\textwidth]{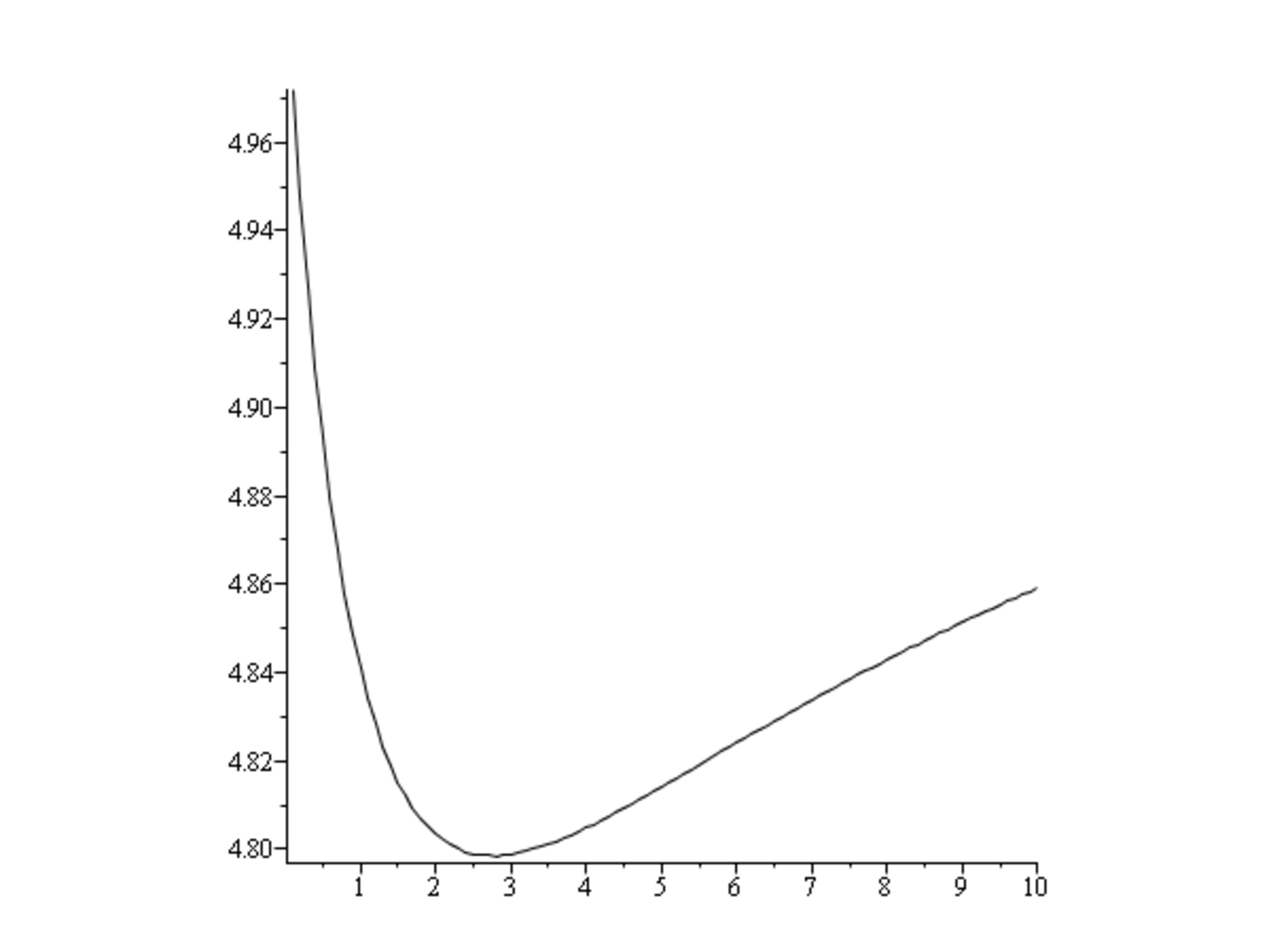}
		\includegraphics[width=.44\textwidth]{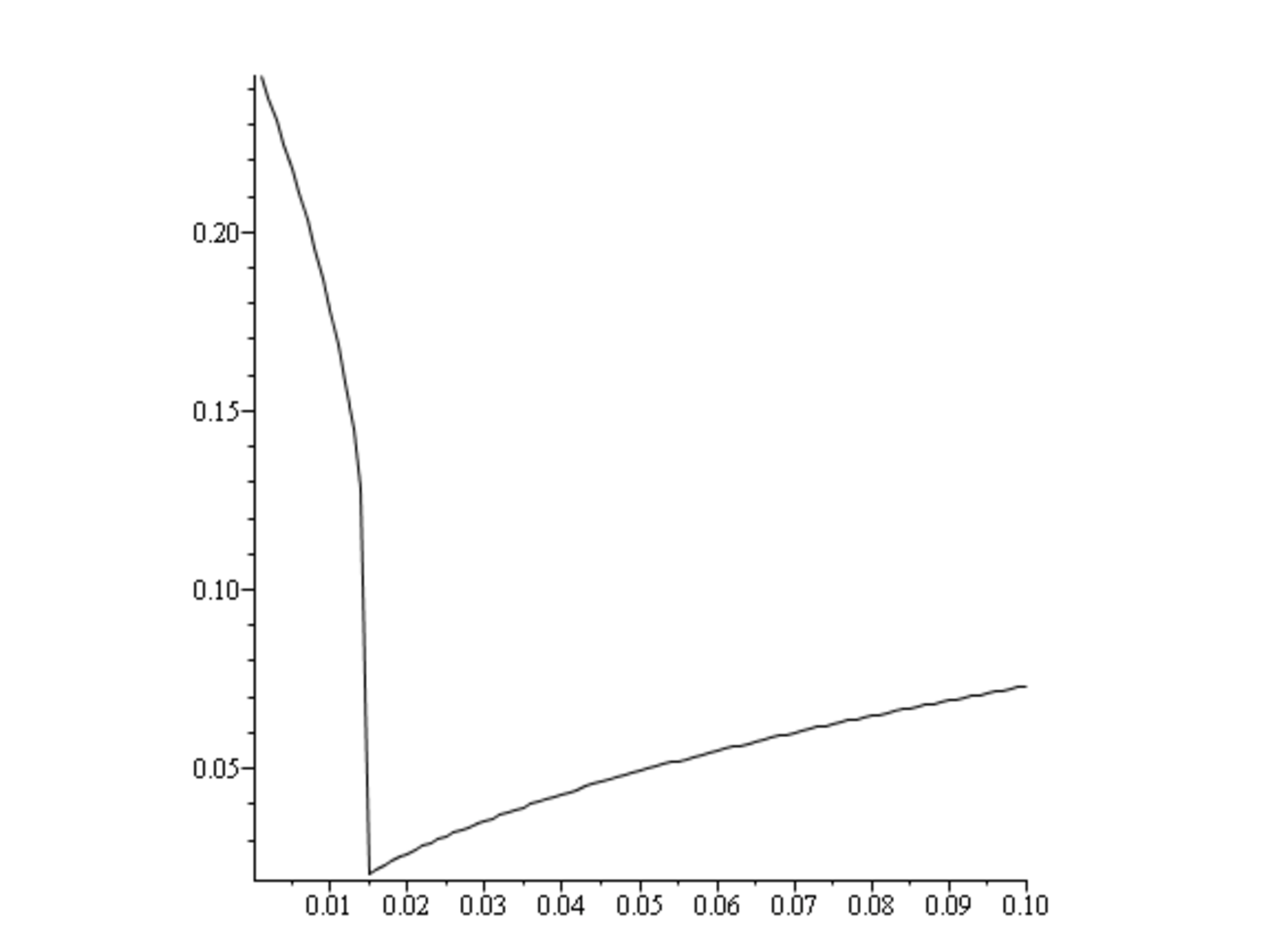}\\
	\end{center}
	\caption{The p-b1i function $\beta_c(\gamma)$.
		Left: $\alpha\mu=5$.
		Right: $\alpha\mu=.25$.}
	\label{fig:betac-b1-ex}
\end{figure}

\subsection{Simulation Validation}

\subsubsection{Stationary densities}
This subsection gives simulation based fractions of infected points in various
regions of the phase diagram and compares them to the polynomial solutions.

\paragraph{Motion-insensitive, Boolean subcritical example}
In this region, {\em any positive motion}
instantly transforms the epidemic from extinct to surviving.
This is illustrated by the results of Tables \ref{table:sick_proportion-vela-below-lowr}
and \ref{table:sick_proportion-vela-wbelow-lowr} (where simulation
and second order heuristics convene).
\begin{table}[h!]
	\centering
	\begin{tabular}{||c|c|c|c|c|c|c|c|c||} 
	\hline
	$\gamma$              & .1  & .5  & 1   & 2   & 5   &  10 & 100  &$\infty$\\
	\hline\hline
	$p_{\mathrm {sim}}$   & 0.04& 0.14& 0.20& 0.23& 0.28&     &      &       \\
	\hline
	$p_{\mathrm{p-b1g1}}$ & 0.18& 0.21& 0.23& 0.26& 0.29& 0.31& 0.33 & 0.33  \\
	\hline
	$p_{\mathrm{p-m2bi}}$ & 0.21& 0.21& 0.22& 0.25& 0.28& 0.30& 0.33 & 0.33  \\
	\hline
	\end{tabular}
\caption{Effect of mobility.
Boolean-subcritical ($\mu=3$), motion-supercritical ($\mu>\mu_0$), 
motion-insensitive ($\beta<\beta_0$) case.  Fraction of infected points ($p$)
obtained by simulation, the functional fixed point equation and the polynomial equation.
This is for $\beta=2$, $a=1$, $\lambda\sim 0.955$ and $\alpha=1$, so that $\mu= 3$. 
According to p-b1i, p-b1g1, and p-m2bi, the threshold $\gamma_c$ is equal
to 0. This means that the epidemic survives for all positive speeds $\gamma$, in
spite of the fact that it dies out for $\gamma=0$.}
\label{table:sick_proportion-vela-below-lowr}
\end{table}

\begin{table}[h!]
	\centering
	\begin{tabular}{||c|c|c|c|c|c|c|c|c||} 
	\hline
	$\gamma$              & .1  & .5  & 1   & 2   & 5   &  10 & 100  &$\infty$\\
	\hline\hline
	$p_{\mathrm {sim}}$   & .09 &     & 0.47&     &     & 0.59&      &       \\
	\hline
	$p_{\mathrm{p-b1g1}}$ & 0.18& 0.36& 0.45& 0.52& 0.57& 0.58& 0.60 & 0.60  \\
	\hline
	$p_{\mathrm{p-m2bi}}$ & 0.21& 0.37& 0.45& 0.52& 0.56& 0.58& 0.60 & 0.60  \\
	\hline
	\end{tabular}
\caption{Effect of mobility, way below Boolean-percolation.
Boolean-subcritical ($\mu=1$), 
motion-supercritical ($\mu>\mu_0$), 
motion-insensitive ($\beta<\beta_0$) case.
Fraction of infected points ($p$) obtained by simulation,
the functional fixed point equation and the polynomial equation.
This is for $\beta=0.4$. 
According to p-b1i, p-b1g1, and p-m2bi, $\gamma_c=0$.
The epidemic survives for all positive speeds $\gamma$, in
spite of the fact that it dies for $\gamma=0$.}
\label{table:sick_proportion-vela-wbelow-lowr}
\end{table}

\paragraph{Motion-sensitive and Boolean-subcritical example}
In this region, for $\gamma$ equal to $0$, 
the epidemic dies out; for positive but small values of $\gamma$
(more precisely, for $0<\gamma<\gamma_c^-$ with $\gamma_c^->0$),
the epidemic survives; for intermediate values of $\gamma$ 
($\gamma_c^-<\gamma<\gamma_c^+$), the epidemic is extinct;
for $\gamma>\gamma_c^+$, the epidemic survives again.
A motion-supercritical instance of this situation is given
in Table \ref{table:sick_proportion-vela-below-highr} 
(where $\mu\sim 3.14$ and $\beta=3>\beta_0$).

\begin{table}[h!]
	\centering
	\begin{tabular}{||c|c|c|c|c|c|c|c|c||} 
		\hline
		$\gamma$              &2.6& 2.7        &  3.3       & 8.2        & 8.3        & 8.6        & 10         &$\infty$\\
		\hline\hline
		$p_{\mathrm {sim}}$   &   &            &            &            &            &            &            &       \\
		\hline
		$p_{\mathrm{p-b1i}}$  & 0 &$3\ 10^{-4}$&            &$2\ 10^{-2}$&$2\ 10^{-2}$&$2\ 10^{-2}$& 0.045      &       \\
		\hline
		$p_{\mathrm{p-b1g1}}$ & 0 &      0     &$7\ 10^{-5}$&$2\ 10^{-2}$&$2\ 10^{-2}$&$2\ 10^{-2}$& 0.045      &       \\
		\hline
		$p_{\mathrm{p-m2bi}}$ & 0 & 0          &     0      &$4\ 10^{-4}$&$2\ 10^{-3}$&$6\ 10^{-4}$& 0.045      &       \\
		\hline
	\end{tabular}
	\caption{Effect of mobility. Below percolation, recovery rate close to $\mu$.
		Fraction of infected points ($p$) obtained by simulation,
		the functional fixed point equation and the polynomial equation.
		This is for $\beta=3$, $a=1$, $\lambda=1$ and $\alpha=1$,
		so that $\mu\sim 3.14$. No simulation results are possible for this case.
		According to p-b1i, the threshold is $\gamma_c^+\sim 2.65$.
		According to p-b1g1, the threshold is $\gamma_c^+\sim 3.12$.
		According to p-m2bi, the threshold is $\gamma_c^+\sim 8.15$. So again m2bi is more resistant to the
		epidemic than m1bi; b1g1 is intermediate.
	}
	\label{table:sick_proportion-vela-below-highr}
\end{table}

\paragraph{Motion-sensitive and Boolean-supercritical example}
In this region, for $\gamma$ equal to $0$, there is survival;
for small but positive values of $\gamma$ (more precisely, for
$0<\gamma<\gamma_c(\beta)$ with $\gamma_c(\beta)>0$), there is extinction;
from $\gamma_c$ on, one starts having survival and, above this value,
the fraction of infected points is strictly increasing in $\gamma$.
In other words, in this case, {\em moderate motion stops the survival
	present in the no-motion case}. One possible explanation is that, in the no-motion
Boolean-supercritical case, the persistence of well connected clusters helps
maintaining the epidemic and motion dissolves these clusters and makes it
more challenging for the epidemic to survive. For high enough values of $\gamma$,
motion again helps for survival. This situation is illustrated in
Figures \ref{fig:betaa} and \ref{fig:betac-b1-ex}.

\subsubsection{Comparison of heuristics}
In this subsection, we numerically compare the various heuristics.

The estimates of $\mu_0$ are 0.263 for b1i, 0.306 for b1g1, and 0.343 for m2bi.  
Table \ref{table:thrbeta0} compares $\beta_0(\mu)$ for $\mu>\mu_0$
for the various heuristics.
Table \ref{table:thrgammac-f} compares $\gamma_c(\mu,\beta)$ for the three heuristics.
Table \ref{table:thrgamma0} compares $\gamma_0$ for the three heuristics. Finally,
Table \ref{table:thrbetac} compares $\beta_c$ for the three heuristics.

\begin{table}[h!]
	\centering
	\begin{tabular}{||c |c|c|c|c|c|c||} 
		\hline
		$\alpha\mu$                           & 0.5 & 1    &  5   & 10   & 20   \\
		\hline\hline
		$\beta_{0, {\mathrm {sim}}}$          &     &      &      &$>8.5$&      \\
		\hline
		$\beta_{0, {\mathrm{p-b1i}}}$         &0.291&0.796 & 4.798& 9.798&19.798\\
		\hline
		$\beta_{0, {\mathrm{p-b1g1}}}$        &0.265&0.772 & 4.777& 9.777&19.777\\
		\hline
		$\beta_{0, {\mathrm{p-m2bi}}}$        &0.157&0.657 & 4.657& 9.657&19.657\\
		\hline
	\end{tabular}
	\caption{Threshold $\beta_0$. Below $\beta_0$ there is survival for all motions.
		Above $\beta_0$ one needs high enough motion for the epidemic to survive.
		This is for $\alpha=1$. For this controllability criterion, m2bi is more resistant
		than b1i and b1g1 is intermediate.}
	\label{table:thrbeta0}
\end{table}

\begin{table}[h!]
	\centering
	\begin{tabular}{||c|c|c|c|c|c||} 
		\hline
		$\beta$                         & 4.75        & 4.80 & 4.85 & 4.90 & 4.95\\
		\hline\hline
		$\gamma_{c, {\mathrm {sim}}}$   &             &      &      &      &      \\
		\hline
		$\gamma_{c, {\mathrm{p-b1i}}}$  &     0       &3.298 & 8.824&17.517&42.711\\
		\hline
		$\gamma_{c, {\mathrm{p-b1g1}}}$ &     0       &6.265 &10.808&19.379&44.557\\
		\hline
		$\gamma_{c,{\mathrm {p-m2bi}}}$ &    5.417    &8.042 &12.290&20.680&45.720\\
		\hline
	\end{tabular}
	\caption{Threshold $\gamma_c$ for $\beta>\beta_0$.
		Below $\gamma_c$ there is extinction, above, there is survival.
		This is for $\alpha=1$ and $\mu=5$. Once more, m2bi is more resistant than
		b1g1 which is more resistant that b1i.}
	\label{table:thrgammac-f}
\end{table}

\begin{table}[h!]
	\centering
	\begin{tabular}{||c|c|c|c|c|c||} 
		\hline
		$\alpha \mu$                          & 0.5 & 1    &  5   & 10   & 20  \\
		\hline\hline
		$\gamma_{0, {\mathrm {sim}}}$          &     &      &      &      &      \\
		\hline
		$\gamma_{0,{\mathrm {p-b1i}}}$         &0.158& 0.450& 2.514& 5.599&11.309\\
		\hline
		$\gamma_{0,{\mathrm {p-b1g1}}}$        &0.121& 0.370& 2.467& 4.857& 9.818\\
		\hline
		$\gamma_{0,{\mathrm {p-m2bi}}}$        &0.055&0.232 & 1.647& 3.417& 6.955\\
		\hline
	\end{tabular}
	\caption{Threshold $\gamma_0$ when $\mu>\mu_0$.
		Below $\gamma_0$, $\beta_c=\beta_0$ and survival/extinction is insensitive to $\gamma$.
		This is for $\alpha=1$. We see than m2bi is more sensitive to gamma than b1g1, which
		is in turn more sensitive than b1i.}
	\label{table:thrgamma0}
\end{table}

\begin{table}[h!]
	\centering
	\begin{tabular}{||c |c|c|c|c|c|c||} 
		\hline
		$\gamma$                              & 0.2 & 1    &  5   & 10   & 100  & $\infty$  \\
		\hline\hline
		$\beta_{c, {\mathrm {sim}}}$          &     &      &      &      &      &      \\
		\hline
		$\beta_{c,{\mathrm {p-b1i}}}$         &4.798& 4.798& 4.814& 4.859& 4.976& 5 \\
		\hline
		$\beta_{c,{\mathrm {p-b1g1}}}$        &4.932& 4.811& 4.802& 4.854& 4.976& 5 \\
		\hline
		$\beta_{c,{\mathrm {p-m2bi}}}$        &4.657& 4.657& 4.740& 4.826& 4.976& 5 \\
		\hline
	\end{tabular}
	\caption{Threshold $\beta_c(\gamma)$ slightly above Boolean-percolation
		($\alpha=1$ and $\mu=5$). For $\beta<\beta_c$ there is survival. For $\beta>\beta_c$,
		there is extinction. The function $\beta_c$ is constant equal to $\beta_0$ for
		$\gamma<\gamma_0$ and increases otherwise. Results are quite close with here b1g1 more resistant
		than b1i and b1i more resistant than m2bi.}
	\label{table:thrbetac}
\end{table}

\subsubsection{Simulation close to criticality}
\label{sec:soc}

Simulating the SIS evolution close to criticality is a challenge.
By definition, the criticality region is that with a vanishing fraction of
infected points. By construction the behavior of the epidemic in this region
is very sensitive to the size of the torus. Consider parameters such
that the infinite system exhibits survival. For all finite size tori with these
parameters, the epidemic ends up dying out in finite time, and random fluctuations
make this time shorter when decreasing the size of the torus.
Since there is no way to simulate arbitrarily large tori, there is hence
no direct way of checking by simulation where the exact value of the
critical parameters are located in general. 

Here are however two natural ways to assess where the threshold lies, which are
both based on the {\em Mean Time Till Absorption} (MTTA) and which
are used to derive the values in the tables of the last sections.

The MTTA is a function of the parameter of interest (say $\gamma$), the torus side, say $L$, 
and the initial condition.
To normalize things, we take as initial condition that with all points infected.
The first method to separate the subcritical and the supercritical regions
consists in fixing a large $L$ and in checking the value of the parameter for which
there is a clear inflection of the MTTA.
The second one leverages the idea that the MTTA should grow slowly with $L$
in the subcritical case (e.g., logarithmically on a grid) and fast in the supercritical
case (e.g. exponentially on a grid). Unfortunately, the exact behavior of the MTTA
on a torus is not known, so that this method cannot be used for a proof at this stage.

\paragraph{Example of estimate of $\beta_c$}
The setting is that of Table \ref{table:sick_proportion-vela-below-vary-beta}, that is $\mu=5$ and $\gamma=1$. Note that $(\mu,\gamma)$ belongs to 
UMI region and to the Boolean-percolation region.

\begin{table}[h!]
	\centering
	\begin{tabular}{||c|c|c|c|c|c|c|c|c|c||} 
		\hline
		$\beta$               &  0 & .2  & 1   & 2    & 2.5 &2.78  & 2.86 & 2.88 & 2.95\\
		\hline\hline
		$p_{\mathrm {sim}}$   &  1 & 0.92&0.59 & 0.22 &  0  &  0   &  0   &  0   &  0  \\
		\hline
		$p_{\mathrm{p-b1i}}$  &    & 0.915&0.598& 0.254&0.103&0.028&0.011 & 0.007  & 0 \\
		\hline
		$p_{\mathrm{p-b1g1}}$ &    & 0.921&0.609& 0.264&0.109&0.029&0.011 & 0.006  & 0 \\
		\hline
		$p_{\mathrm{p-m2bi}}$ & 1  &0.918 &0.605&0.254 &0.092&0.006& 0    &  0   & 0 \\
		\hline
		$p_{\mathrm {p-m}\infty\mathrm{bi}}$  &    & 0.92 &0.61 & 0.27 & 0.12&0.046& 0.028&  0   & 0 \\
		\hline
	\end{tabular}
	\caption{Effect of recovery rate $\beta$. Below Boolean-percolation, medium mobility.
		Fraction of infected points ($p$) obtained by simulation,
		the functional fixed point equation and the polynomial equation.
		This is for $\gamma=1$, $a=1$, $\lambda=1$ and $\alpha=1$, so that $\mu\sim 3.14$.
		According to p-b1i, $\beta_c\sim 2.94$\ ($2.93$ for p-b1g1).
		According to p-m2bi, $\beta_c\sim 2.82$. So, for these values, m2bi is a bit more
		'resistant' to the epidemic than b1i.
		The simulator yields an estimate for this threshold around 2.7 (see 
		Subsection \ref{sec:soc}). 
		Note that the three estimates by b1 and m2bi are quite close and
		consistent with simulation.
	}
	\label{table:sick_proportion-vela-below-vary-beta}
\end{table}

Figure \ref{fig:met1betac} illustrates the two methods.
Both give a $\beta_c$ between 2.5 and 2.6.
We recall that according to p-b1i, $\beta_c\sim 2.94$\ ($2.93$ for p-b1g1) and
according to p-m2bi, $\beta_c\sim 2.82$.\\
Note also that Table \ref{table:sick_proportion-vela-below-vary-beta}
lends experimental support to the continuity assumption on phase transitions
that we made and described in Section \ref{ss:criticality_phase_diag}.

\begin{figure}
	\begin{subfigure}[b]{.5\textwidth}
		\centering
		\includegraphics[trim={2cm 1cm 2cm 0}, clip, width=\textwidth]{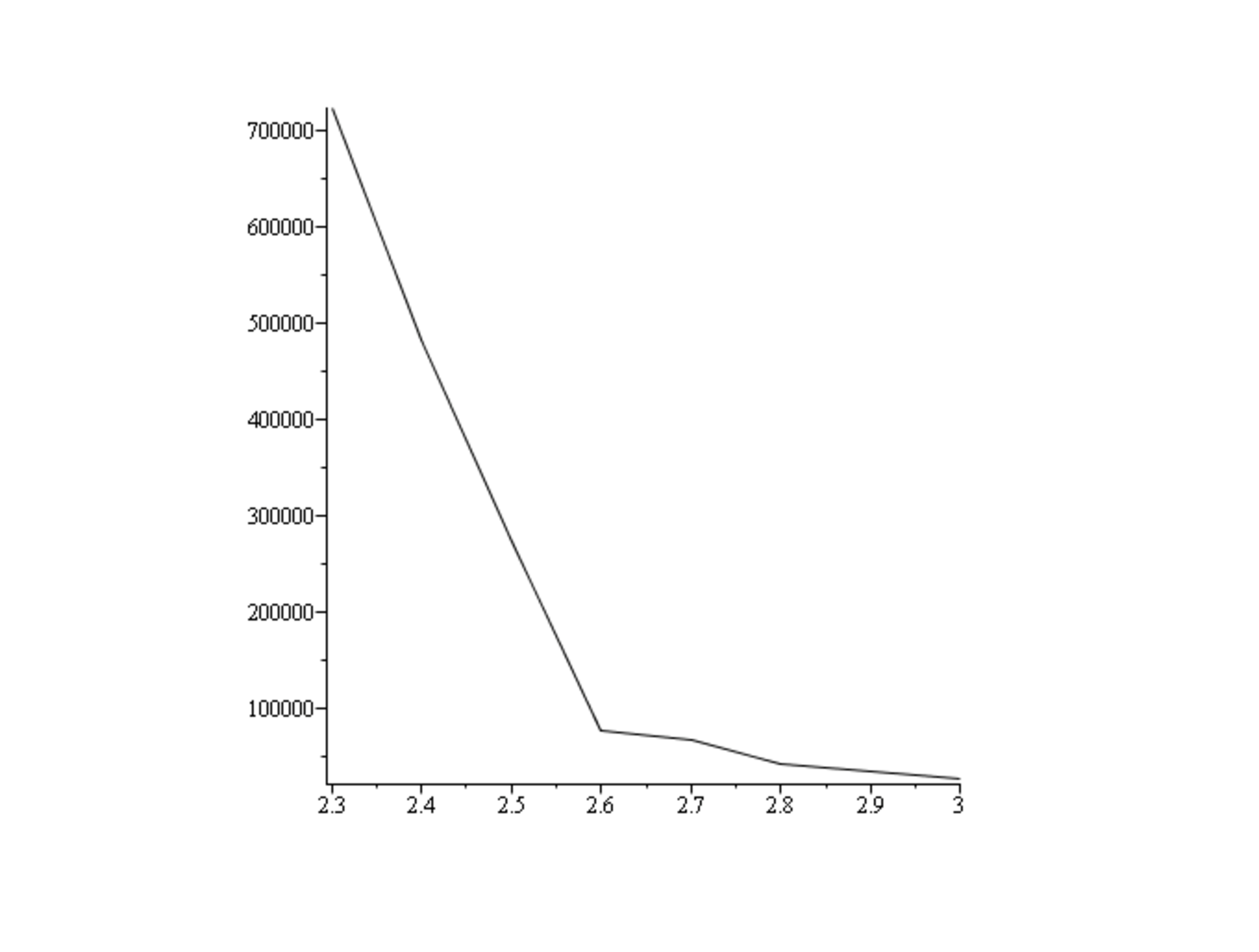}
	\end{subfigure}%
	\begin{subfigure}[b]{.5\textwidth}
		\centering
		\includegraphics[trim={2cm 0 2cm 0}, clip, width=\textwidth]{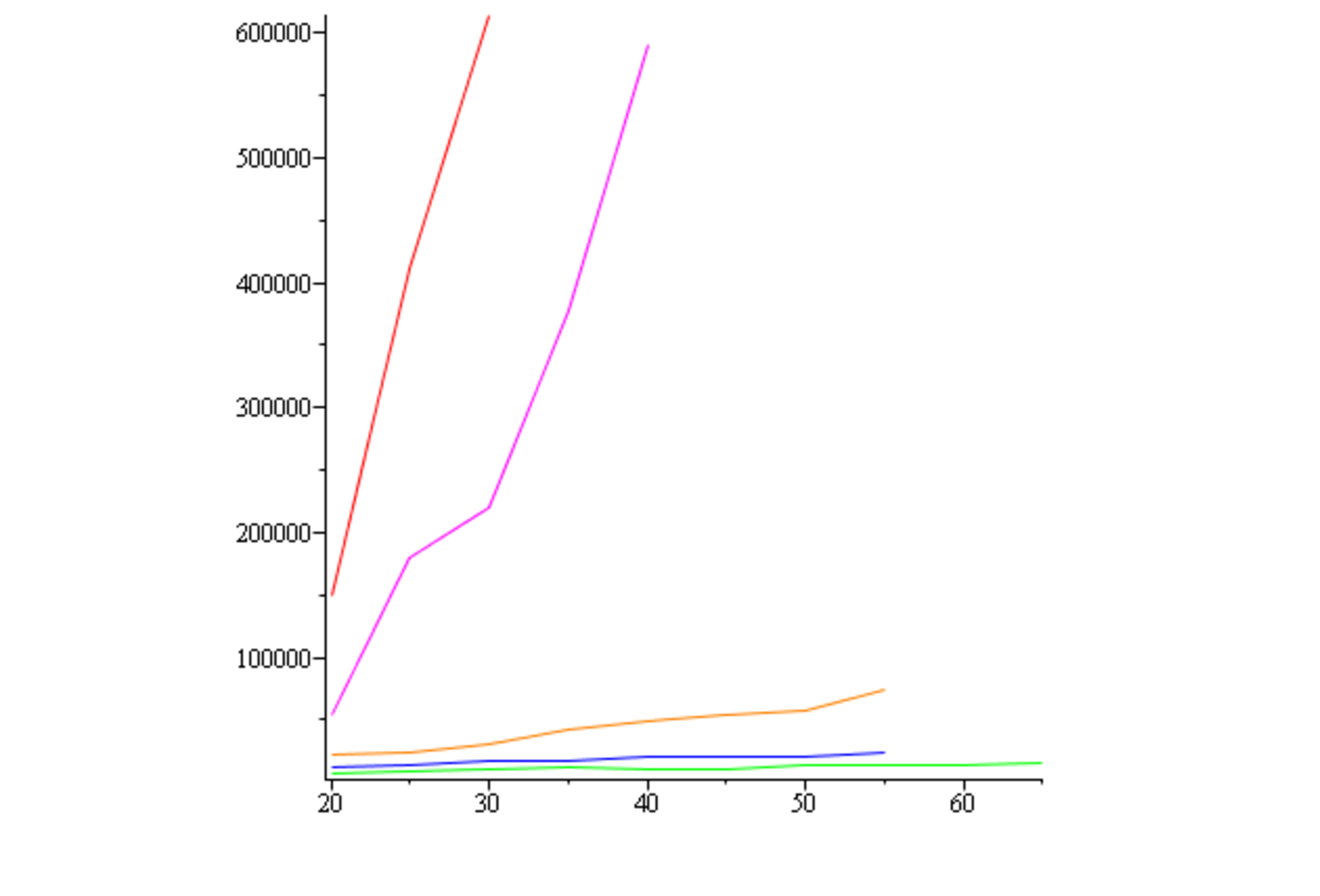}
	\end{subfigure}
	\caption{Left: illustration of Method 1 (dependency on the parameter $\beta$);
		the $x$ axis features $\beta$; the $y$ axis features
		the average of the MTTA over $L=20,25,30$ and over 20 runs of simulation
		for each case.
		Right: illustration of Method 2 (dependency on the torus side $L$):
		on the $x$ axis, $L$; on the $y$ axis,
		the average of the MTTA;
		the green curve is for $\beta=3$;
		the blue curve is for $\beta=2.8$;
		the orange curve is for $\beta=2.6$;
		the pink curve is for $\beta=2.4$;
		the red curve is for $\beta=2.3$;
		All curves are for
		$\gamma=1$, $a=1$, $\lambda=1$ and $\alpha=1$, so that $\mu\sim 3.14$.
	}
	\label{fig:met1betac}
\end{figure}

\paragraph{Examples of estimate of $\gamma_c$}

The first example is in the UMS region.
The setting is that of Table \ref{table:thrgammac-f} with $\alpha=1$, $\beta=4.8$ and $\mu=5$.
For p-b1i, $\gamma_c^+\sim 3.3$\ ($6.3$ for p-b1g1), while for p-m2bi, $\gamma_c^+\sim 8.0$.
When using the methodology described above, simulation
suggests a value of $\gamma_c^+$ around 3 (see Figure \ref{fig:met1gammac548} left).
For $\gamma_c^-$, the value predicted by p-mb2i 
is 0.36 and that by p-bli 2.3, while
simulation suggests a value around 0.4 (Figure \ref{fig:met1gammac548} right).

\begin{figure}[h!]
	\begin{center}
		\includegraphics[width=.7\textwidth]{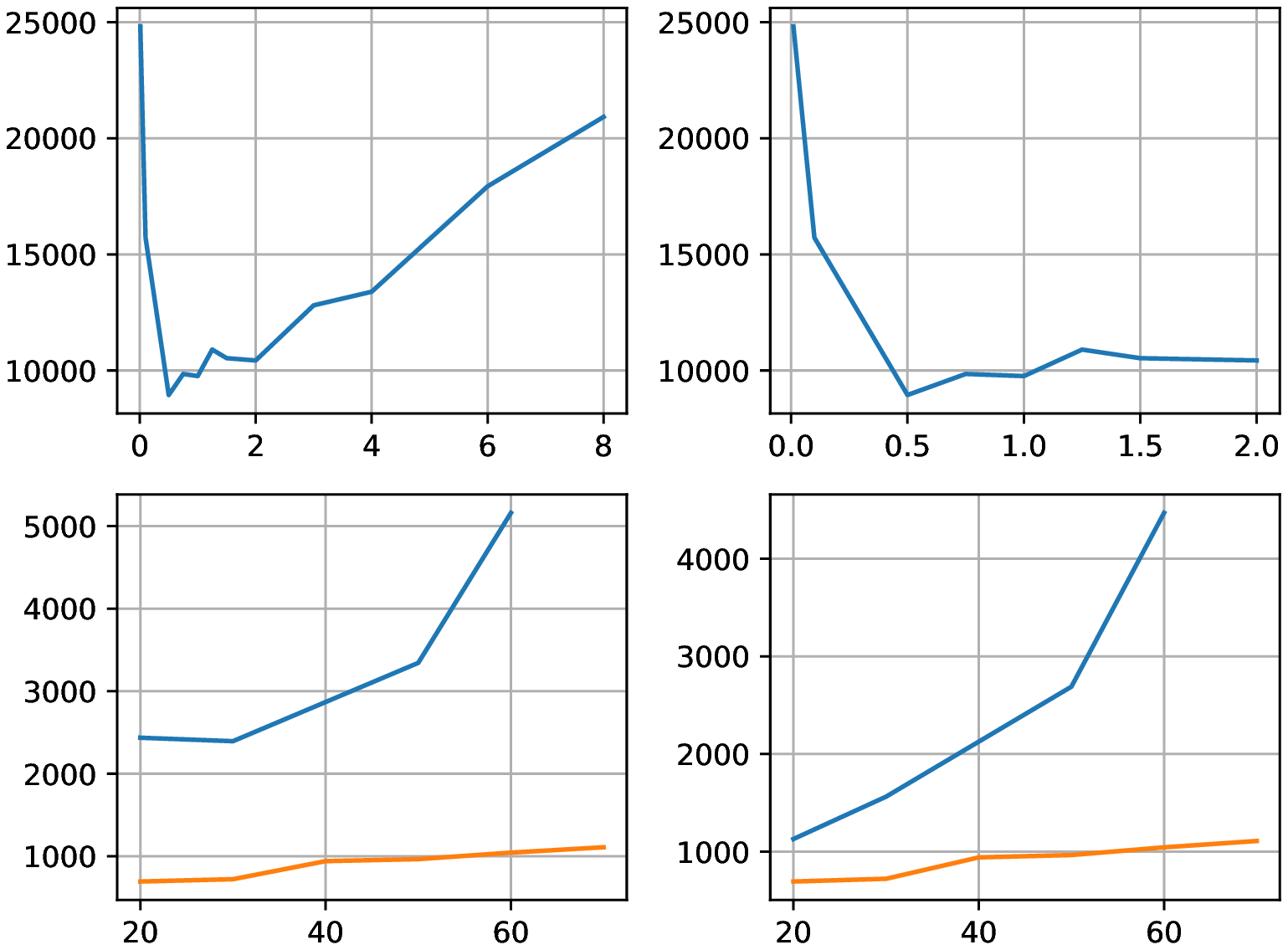}
	\end{center}
	\caption{
		Here, $\beta=4.8$, $\alpha=1$ and $\mu = 5$, with $\lambda=1$, $a=1.261$.
		Top: Illustration of Method 1 (dependency on the parameter $\gamma$).
		On the $x$ axis, $\gamma$. On the $y$ axis, the MTTA averaged out over the cases $L=40$ and $L=50$. 
		Left: whole curve.
		Right: region for the evaluation of $\gamma_c^-$.
		Bottom: Illustration of Method 2: 
		On the $x$ axis, $L$. On the $y$ axis, the MTTA averaged out over 10 samples.
		Bottom left: the top curve is for $\gamma=0.001$, the bottom one for $\gamma=0.5$.
		Bottom right: the top curve is for $\gamma=15$, the bottom one for $\gamma=0.5$.
	}
	\label{fig:met1gammac548}
\end{figure}

The second example is in the motion-subcritical UMS region.
The setting is $\mu=1/4$ ($\lambda =1/\pi$, $a=1/2$), $\alpha=1$, and 
$\beta =1/5$.
The value of $\gamma_c^+$ predicted by p-b1i is 1.73.
That predicted by p-mb2i is 1.84.
When using the methodology described above, Method 1
suggests a value of $\gamma_c^+$ around 2.5 (see Figure \ref{fig:met1gammac}, left).
For $\gamma_c^-$, the value predicted by p-mb2i 
is 0.003 and that by p-bli 0.007.
Simulation suggests a value of $\gamma_c^-$ around 0.1 (see Figure \ref{fig:met1gammac}, right).

\begin{figure}[h!]
	\begin{subfigure}{.5\textwidth}
		\centering
		\includegraphics[width=\textwidth]{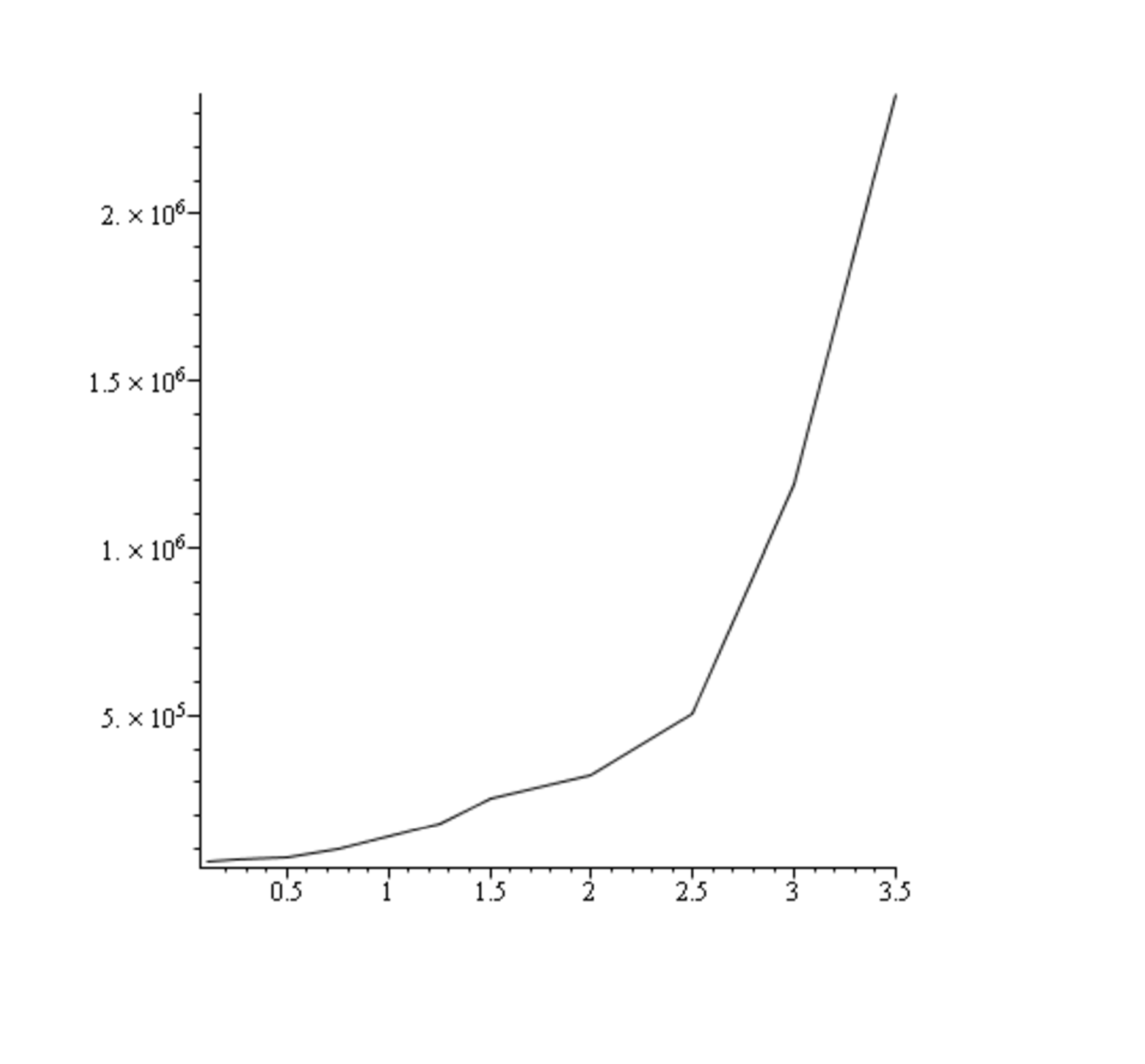}
	\end{subfigure}%
	\begin{subfigure}{.5\textwidth}
		\centering
		\includegraphics[width=1.2\textwidth]{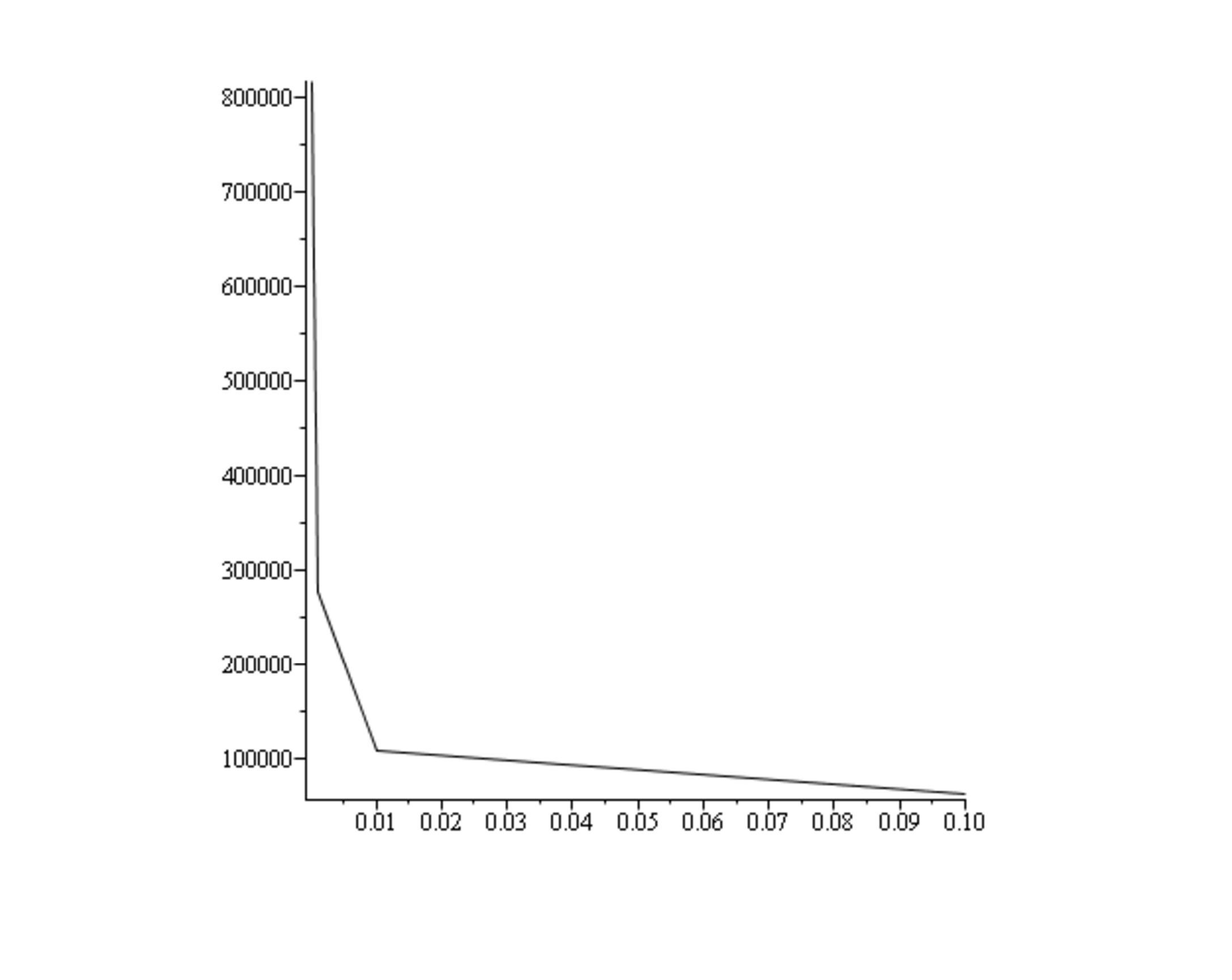}
	\end{subfigure}

	\caption{
		Here, $\beta=0.2$, $\alpha=1$ and $\mu = 0.25$, with $\lambda=1/\pi$, $a=1/2$.
		Illustration of Method 1 (dependency on the parameter $\gamma$).
		On the $x$ axis, $\gamma$. On the $y$ axis, the MTTA averaged out over the
		cases $L=40$ and $L=50$. 
		Left: region for the evaluation of $\gamma_c^+$.
		Right: region for the evaluation of $\gamma_c^-$.
	}
	\label{fig:met1gammac}
\end{figure}
Method 2 is illustrated for $\gamma_c^-$ and the same case
on Fig. \ref{fig:met2gammac}. 
\begin{figure}[h!]
	\begin{center}
		\includegraphics[trim={0 4cm 0 3.3cm}, clip, width=.7\textwidth]{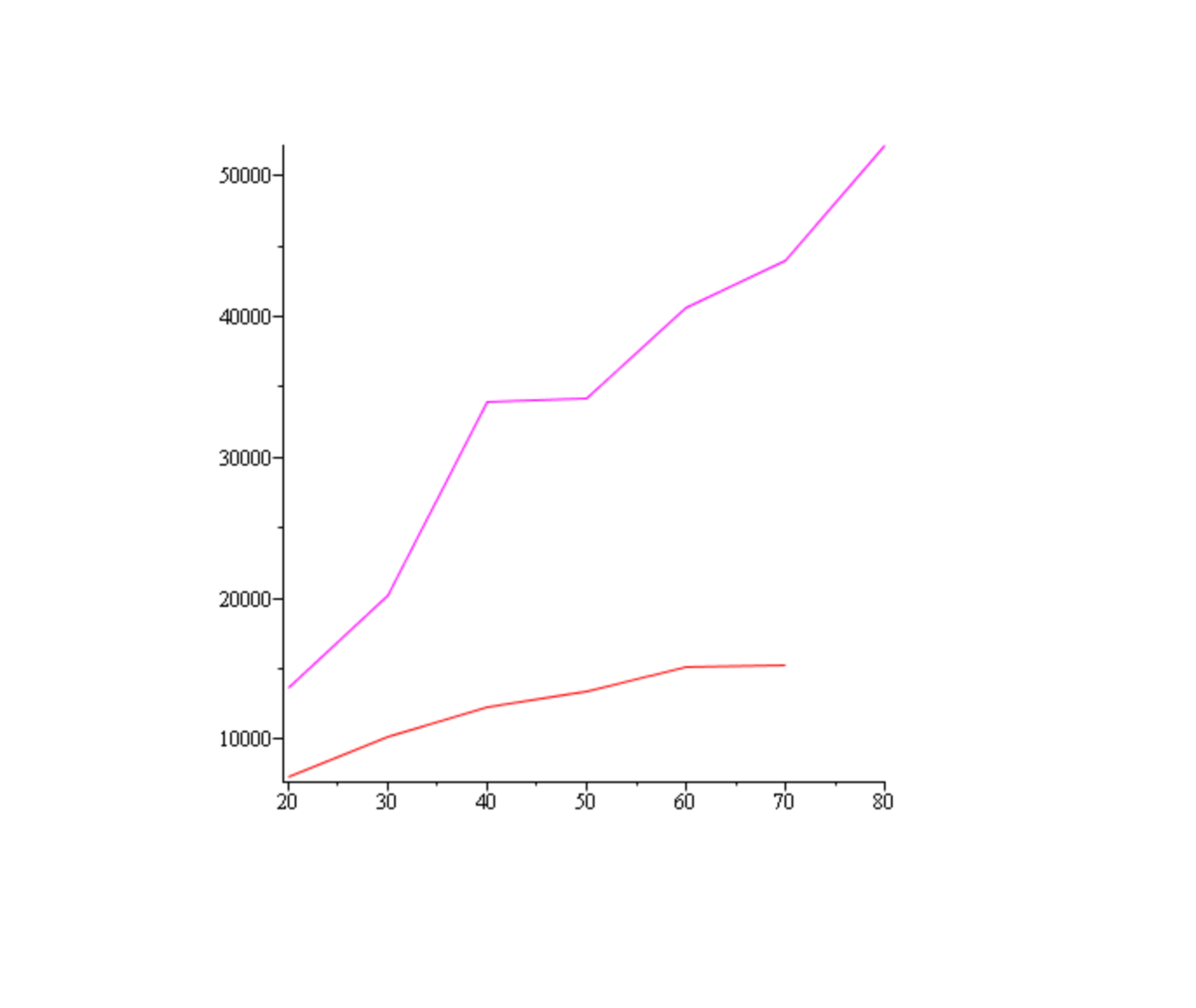}
	\end{center}
	\caption{
		Here, $\beta=0.2$, $\alpha=1$ and $\mu = 0.25$, with $\lambda=1/\pi$, $a=1/2$.
		Illustration of Method 2 (dependency on the parameter $\gamma$) for $\gamma_c^-$.
		On the $x$ axis, $L$. On the $y$ axis, the MTTA averaged out over 10 runs.
		The upper curve is for $\gamma=0.001$; the lower one is for $\gamma=1$.
	}
	\label{fig:met2gammac}
\end{figure}

Let us stress once more that none of these methods provides a proof of the $(\mu,\gamma)$-phase diagram.
Nevertheless, we can deduce from Method 1 (together with confidence intervals) that in the UMS region,
the MTTA in a large torus is a decreasing function of $\gamma$ around $\gamma_c^-$ and an increasing function
of $\gamma$ around $\gamma^+_c$. The former property in itself is a surprising fact.
We illustrate this using 95\% confidence intervals in Figure \ref{fig:log}.
\begin{figure}[h!]
	\begin{center}
		\includegraphics[trim={0 1cm 0 0.3cm}, clip, width=.49\textwidth]{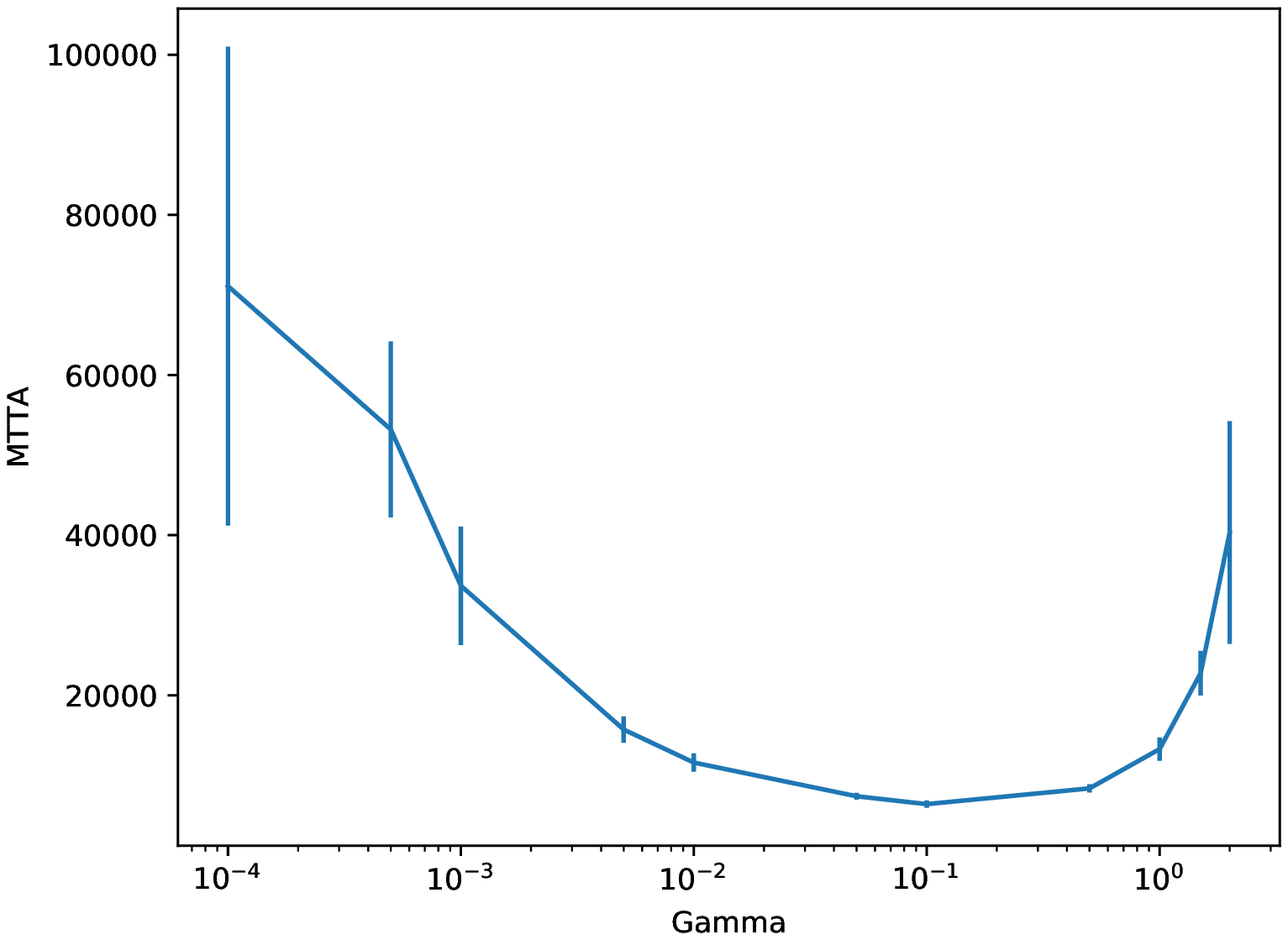}
		\includegraphics[trim={0 1cm 0 0.3cm}, clip, width=.49\textwidth]{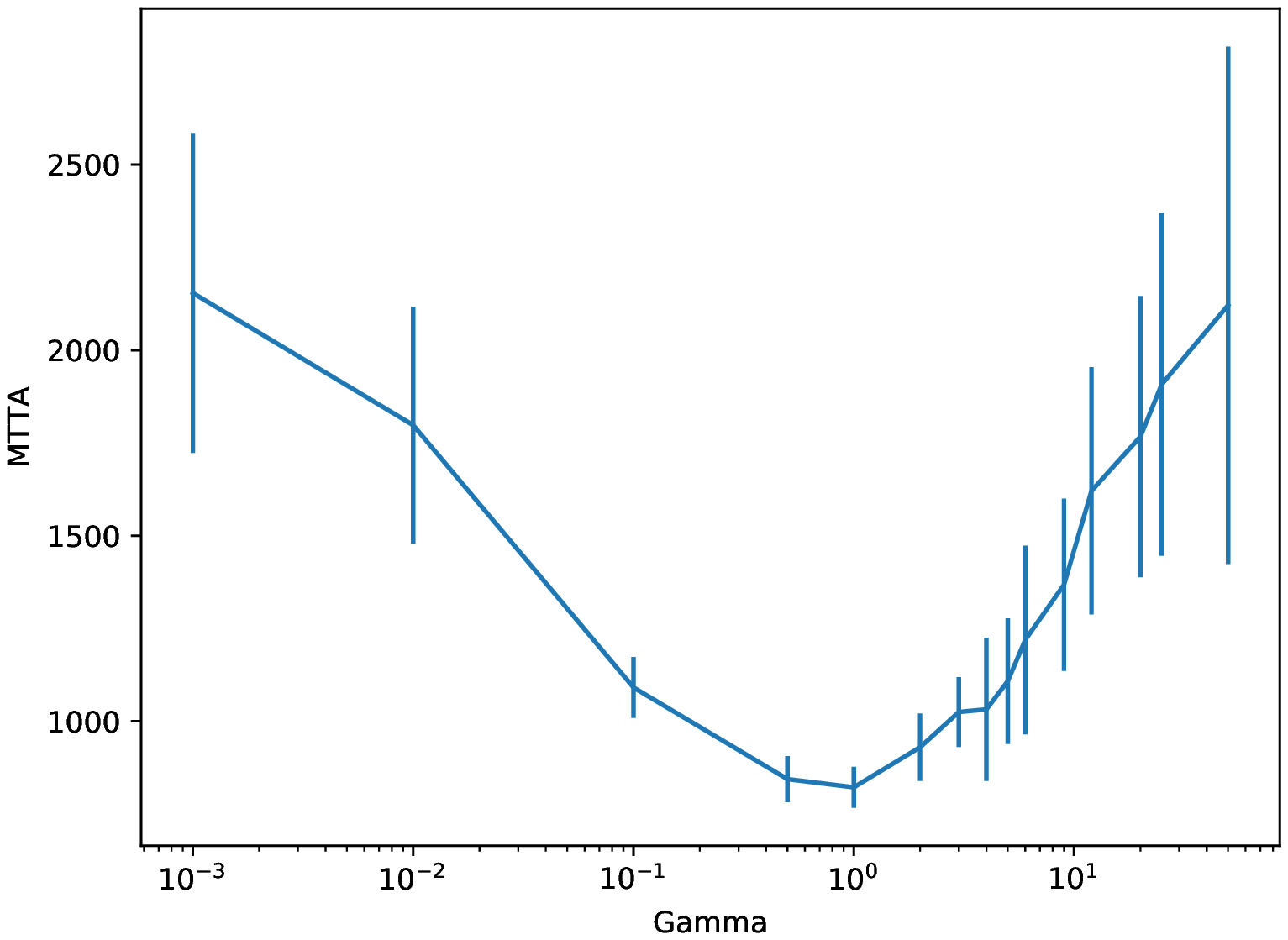}
	\end{center}
	\caption{Left: the parameters are those of Figure \ref{fig:met1gammac548}:
		$\beta=4.8$, $\alpha=1$ and $\mu = 5$, with here $\lambda=0.8$, $a=1.41$.
		The torus side is $L=40$.
		Right: the parameters are those of Figure \ref{fig:met1gammac548}:
		$\beta=.2$, $\alpha=1$ and $\mu = .25$, with here $\lambda=1/\pi$, $a=1/2$.
		The torus side is $L=40$.
		In both cases, the $x$-axis gives the log of $\gamma$.
		The $y$ axis gives the MTTA with 95\% confidence intervals. The MTTA is a unimodal
		function which decreases around $\gamma_c^-$ and increases around $\gamma_c^+$.
	}
	\label{fig:log}
\end{figure}
The physical explanation is that already mentioned:
due to the randomness of the configurations, there are clusters with high
connectivity; for lower motion rates, these clusters persist for a longer time,
which in turn favors the survival of the epidemic.
\section{No-Motion Case}
\label{sec:rcp2}

One option to study the case without motion is to fix $\gamma$ to 0 in
the equations of the last section. This approach suffers of a weakness;
it does not allow one to incorporate the fact that the steady state 
is necessarily the empty measure on the finite components of the Boolean model.

We will hence follow another way below which consists in looking
at the epidemic on the infinite cluster of the Boolean model.
The reason is that this is a more precise formulation of the problem, given that
the underlying graph is not connected and that the epidemic
dies out in finite time on any finite connected component.

In this section, the Poisson point process is $\Xi=\Xi_0$ at all
times, with $\Xi$ of intensity $\lambda$.
The infinite cluster of the associated Boolean model is denoted by $\widetilde \Xi$. 
The infected and susceptible sub-point processes of $\widetilde \Xi$
are denoted by $\widetilde \Phi$ and $\widetilde \Psi$, respectively.
That is, the Poisson point process is decomposed into three point processes rather than two:
those in finite clusters, say $\Pi$ (whose state is hence 0 in the stationary
regime), those infected and in the infinite cluster, $\widetilde \Phi$, and
those susceptible and in the infinite cluster, $\widetilde \Psi$. We have
$$\Xi=\Pi+\widetilde \Xi,\quad \widetilde \Xi= \widetilde \Phi+\widetilde \Psi.$$
The intensity of $\Pi$ is $\lambda(1-q)$ (the probability for the origin of the Poisson
point process to be in a finite cluster under Palm),
whereas that of $\widetilde \Xi$ is $\widetilde \lambda=\lambda q$. When
denoting by $\widetilde p$ the steady state fraction of infected points in $\widetilde \Xi$,
we have that $\widetilde \Phi$ and $\widetilde \Phi$ have intensities $\widetilde \lambda \widetilde p$
and $\widetilde \lambda  (1-\widetilde p)$ respectively.

Let us now give the special forms of the pair correlation functions of these processes.
We have
$$ 1=(1-q)^2\xi_{\Pi,\Pi}(r)+ q^2\xi_{\widetilde \Xi,\widetilde \Xi}(r) +2q(1-q)\xi_{\Pi,\widetilde \Xi}(r)
,\quad \forall r,$$
with also
$$\xi_{\Pi,\widetilde \Xi}(r) =0,\quad \mbox{for}\ r<a.$$
Hence
\begin{eqnarray*}
\label{eq:labonne}
q^2(\widetilde p^2 \xi_{\widetilde \Phi,\widetilde \Phi}(r) +(1-\widetilde p)^2\xi_{\widetilde \Psi,\widetilde \Psi}(r)
+2\widetilde p(1-\widetilde p)\xi_{\widetilde \Psi,\widetilde \Phi}(r)) & & \nonumber\\
&  &\hspace{-4cm} =
q^2 \xi_{\widetilde \Xi,\widetilde \Xi}(r)\nonumber\\
&  &\hspace{-4cm} =
1-(1-q)^2\xi_{\Pi,\Pi}(r)
\quad \mbox{for}\ r<a.
\end{eqnarray*}
The $\xi_{\widetilde \Xi,\widetilde \Xi}(\cdot)$ function that shows up in the last equation has nothing to do
with the SIS dynamics. It only depends on the random graph on which the epidemic develops, namely
on the parameters $\lambda$ and $f$. 
We are only interested in this function above percolation. See Appendix \ref{append:pcfic}
for branching type estimates of the parameter $q$.

Note that when we are way above the percolation threshold, the infinite cluster
is well approximated by the Poisson point process so that taking
$\xi_{\widetilde \Xi,\widetilde \Xi}(r)\sim 1$ should be a good approximation.

The main questions of interest on the case where the Boolean model percolates
are again the estimation of the critical functions $\alpha_c$ and $\beta_c$ and 
the estimation of the density $\widetilde p$ of infected points in any stationary regime.

\subsection{Rate Conservation Principle for Intensities}

Consider the case where the Boolean model is supercritical.
We can then give another version of Lemma \ref{lem2} on the infinite cluster $\widetilde \Xi$. 

Below we will use the notation
$\widetilde \mu = \widetilde \lambda \pi a^2$.
Note the following relations with the earlier notation:
\begin{equation}
\label{eq:roth}
\widetilde \lambda = q \lambda,\quad
\widetilde \mu=q  \mu,\quad
p =q\widetilde p.
\end{equation}
With this notation, we have:
\begin{Lem}\label{lem1-bis}
In the no-motion, Boolean-supercritical case
\begin{equation}\label{eq:rcp1-bis}
\widetilde p\beta = (1-\widetilde p) \mathbb E^0_{\widetilde \Psi} [I_{\widetilde \Phi}(0)].
\end{equation}
\end{Lem}

In the no-motion, Boolean-supercritical case,
if the $(\widetilde \Phi,\widetilde \Psi)$ repulsion property holds (this is
not conjectured here), we have
\begin{equation}
\label{eq:critclustab}
\alpha_c \ge \frac \beta {\widetilde \mu},  \quad
\beta_c \le \alpha  \widetilde \mu.
\end{equation}
If it is not 0, the fraction $\widetilde p$ of 
infected points of $\widetilde \Xi$ satisfies the bound
\begin{equation}\label{eq:bound-bis}
0< \widetilde p \le 1- \frac{\beta}{\alpha \widetilde \mu}.
\end{equation}

Note that the $(\widetilde \Phi,\widetilde \Psi)$ repulsion property (on the infinite cluster) is not
a corollary of that on  $(\Phi,\Psi)$, namely on the Poisson point process.
This second repulsion property possibly only holds in special cases (and is not conjectured).
The intuitive reason is that the substrate point process (the infinite cluster)
exhibits attraction, which may dominate the repulsion between infected and susceptible.
Note that if this cluster repulsion property does not hold, we still have the bounds
(\ref{eq:critab}) under the first repulsion conjecture.

Here is the pair correlation function reformulation of RCP1 on the infinite cluster:
\begin{equation}
\label{eq:rcp1-integ-pol-bis}
\beta =\widetilde \lambda (1-\widetilde p)
2 \pi \int_{\mathbb R^+} \xi_{\widetilde \Psi,\widetilde \Phi} (r) f(r) r dr.
\end{equation}

\subsection{Second Moment RCP}
We write the conservation equations for the second order moment measures.
The justification is the same as in \cite{bmn}.
\paragraph{$(\widetilde \Phi,\widetilde \Phi)$}
The "mass birth rate" in $\rho^{(2)}_{\widetilde \Phi,\widetilde \Phi} (r)$ is
$$ 
2 \rho^{(2)}_{\widetilde \Psi,\widetilde \Phi} (r)
\left( f(r)+ \int_{\mathbb R^2} \mu(\widetilde \Phi)^{0,r}_{\widetilde \Psi,\widetilde \Phi}(x) f(||x||) {\rm d}x\right) ,$$
with $\mu(\widetilde \Phi)_{\widetilde \Psi,\widetilde \Phi}^{0,r}(x)$ 
the conditional density of $\widetilde \Phi$ at $x$ given that
$\widetilde \Psi$ has a points at $(0,0)$ and $\widetilde \Phi$ a point at $(r,0)$.
The 2 comes from the fact that the infection of the point of $\widetilde \Psi$ creates two
infected points at a distance $r$ of each other.
The "mass death rate" in $\rho^{(2)}_{\widetilde \Phi,\widetilde \Phi} (r)$ is
$$ 2 \rho^{(2)}_{\widetilde \Phi,\widetilde \Phi} (r) \beta.$$
Hence
\begin{eqnarray}
\label{eq:rcpmm2}
& &\hspace{-2cm} \widetilde \lambda^2\widetilde p^2\xi_{\widetilde \Phi,\widetilde \Phi} (r) \beta\\
& =  & 
\widetilde \lambda^2 \widetilde p(1-\widetilde p) 
\xi_{\widetilde \Psi,\widetilde \Phi} (r) \left( f(r)+\int_{\mathbb R^2} \mu(\widetilde \Phi)_{\widetilde \Psi,\widetilde \Phi}^{0,r}(x) f(||x||) {\rm d}x\right).
\nonumber
\end{eqnarray}

\paragraph{$(\widetilde \Psi,\widetilde \Psi)$}

The "mass birth rate" in $\rho^{(2)}_{\widetilde \Psi,\widetilde \Psi} (r)$ is
$$ 2\rho^{(2)}_{\widetilde \Psi,\widetilde \Phi} (r) \beta.$$
The "mass death rate" in $\rho^{(2)}_{\widetilde \Psi,\widetilde \Psi} (r)$ is
$$ 
2 \rho^{(2)}_{\widetilde \Psi,\widetilde \Psi} (r)
\int_{\mathbb R^2} \mu(\widetilde \Phi)^{0,r}_{\widetilde \Psi,\widetilde \Psi}(x) f(||x||) {\rm d}x,$$
with a similar notation.  Hence
\begin{eqnarray}
\label{eq:rcpmm3}
& &\hspace{-2cm}  \widetilde \lambda^2 \widetilde p (1-\widetilde p)\xi_{\widetilde \Psi,\widetilde \Phi} (r) \beta\\
& =  & 
 \widetilde \lambda^2 (1-\widetilde p)^2 
\xi_{\widetilde \Psi,\widetilde \Psi} (r) \int_{\mathbb R^2} \mu(\widetilde \Phi)_{\widetilde \Psi,\widetilde \Psi}^{0,r}(x) f(||x||) {\rm d}x.
\nonumber
\end{eqnarray}

\paragraph{$(\widetilde \Psi,\widetilde \Phi)$}

The "mass birth rate" in $\rho^{(2)}_{\widetilde \Psi,\widetilde \Phi} (r)$ is
$$ 2 \rho^{(2)}_{\widetilde \Phi,\widetilde \Phi} (r) \beta
+
2 \rho^{(2)}_{\widetilde \Psi,\widetilde \Psi} (r) \int_{\mathbb R^2} \mu(\widetilde \Phi)^{0,r}_{\widetilde \Psi,\widetilde \Psi}(x) f(||x||) {\rm d}x.$$
The "mass death rate" in $\rho^{(2)}_{\widetilde \Psi,\widetilde \Phi} (r)$ is
$$ 2 \rho^{(2)}_{\widetilde \Psi,\widetilde \Phi} (r) \beta
+
2 \rho^{(2)}_{\widetilde \Psi,\widetilde \Phi} (r) \left( f(r)+\int_{\mathbb R^2} \mu(\widetilde \Phi)_{\widetilde \Psi,\widetilde \Phi}^{0,r}(x) f(||x||) {\rm d}x\right),$$
with the same type of notation.
So
\begin{eqnarray}
\label{eq:rcpmm1}
& & \hspace{-.3cm} \widetilde \lambda^2\widetilde p^2\xi_{\widetilde \Phi,\widetilde \Phi} (r) \beta
+
 \widetilde \lambda^2 (1-\widetilde p)^2 \xi_{\widetilde \Psi,\widetilde \Psi} (r)
\int_{\mathbb R^2} \mu(\widetilde \Phi)^{0,r}_{\widetilde \Psi,\widetilde \Psi}(x) f(||x||) {\rm d}x = \\
& &\hspace{-0.7cm} \widetilde \lambda^2 \widetilde p(1-\widetilde p) \xi_{\widetilde \Psi,\widetilde \Phi} (r) \beta
+ \widetilde \lambda^2 \widetilde p(1-\widetilde p) 
\xi_{\widetilde \Psi,\widetilde \Phi} (r) \left( f(r)+\int_{\mathbb R^2} \mu(\widetilde \Phi)_{\widetilde \Psi,\widetilde \Phi}^{0,r}(x) f(||x||) {\rm d}x\right).
\nonumber
\end{eqnarray}

We see that (\ref{eq:rcpmm1}) is obtained by linear combination of
(\ref{eq:rcpmm2}) and (\ref{eq:rcpmm3}). 

\begin{Thm}
\label{thm9}
If there exists a stationary regime, then the stationary pair correlation functions 
satisfy the following system of integral equations:
\begin{eqnarray}
\label{eq:rcpmm-glob}
\widetilde p\beta \xi_{\widetilde \Phi,\widetilde \Phi} (r)
& =  & 
(1-\widetilde p) 
\xi_{\widetilde \Psi,\widetilde \Phi} (r) \left( f(r)+\int_{\mathbb R^2} \mu(\widetilde \Phi)_{\widetilde \Psi,\widetilde \Phi}^{0,r}(x) f(||x||) {\rm d}x\right)
\nonumber\\
\widetilde p\beta \xi_{\widetilde \Psi,\widetilde \Phi} (r)
& =  & (1-\widetilde p) 
\xi_{\widetilde \Psi,\widetilde \Psi} (r) \int_{\mathbb R^2} \mu(\widetilde \Phi)_{\widetilde \Psi,\widetilde \Psi}^{0,r}(x) f(||x||) {\rm d}x,
\end{eqnarray}
with $ \mu(\widetilde \Phi)_{\widetilde \Psi,\widetilde \Psi}^{0,r}(x)$ and
$ \mu(\widetilde \Phi)_{\widetilde \Psi,\widetilde \Phi}^{0,r}(x)$ the conditional densities defined above.
\end{Thm}

The last system can be complemented by the following relations,
which were established above:
\begin{equation}\label{eq:rcp1-integ-rep}
\beta = (1-\widetilde p) \widetilde \lambda \int_{\mathbb R^2} \xi_{\widetilde \Psi,\widetilde \Phi} (x) f(||x||) {\rm d}x
\end{equation}
and
\begin{eqnarray}
\label{eq:labonne-rep}
\widetilde p^2 \xi_{\widetilde \Phi,\widetilde \Phi}(r) +(1-\widetilde p)^2\xi_{\widetilde \Psi,\widetilde \Psi}(r)
+\widetilde 2p(1-\widetilde p)\xi_{\widetilde \Psi,\widetilde \Phi}(r))=c(r),
\end{eqnarray}
where $c(\cdot)$ denotes the pair correlation function of the infinite component of the Boolean cluster:
\begin{equation}
\label{eq:cder}
c(r) := \xi_{\widetilde \Xi,\widetilde \Xi}(r),
\end{equation}
under the assumption of Boolean-percolation.

Note that once the pair correlation functions solution of this system are determined, we
get the fraction of infected points $p$ in the Poisson
point process of intensity $\lambda$
from (\ref{eq:roth}) and (\ref{eq:rcp1-integ-rep}) through the relation
\begin{equation}
\label{eq:roth2}
p= q \widetilde p =
q- \frac{\beta}
{\lambda \int_{\mathbb R^2} \xi_{\widetilde \Psi,\widetilde \Phi} (x) f(||x||) {\rm d}x}
.
\end{equation}

\subsection{Heuristics}
\label{ss:H5.3}
We use the methodology and classification of Sections \ref{ss:heuri} and \ref{ss:prince} to obtain and name heuristics for the no-motion case. We present one example of a heuristic set of equations here and list the rest in Appendix \ref{append:no_motion_heuristics}.
\subsubsection{Heuristic B1I}
\paragraph{Functional equation f-b1i}
The functional equation reads
\begin{eqnarray}
\label{eq:inteqb1i}
\widetilde p \xi_{\widetilde \Phi,\widetilde \Phi} (r) \beta
& =  &  (1-\widetilde p) \xi_{\widetilde \Psi,\widetilde \Phi} (r) f(r)\nonumber\\
& &  \hspace{-1cm}
+  \widetilde \lambda (1-\widetilde p) \widetilde p
\xi_{\widetilde \Psi,\widetilde \Phi} (r)^{\frac 2 3} 
\int_{\mathbb R^2} 
\xi_{\widetilde \Psi,\widetilde \Phi}(||\mathbf{x}||)^{\frac 2 3} 
\xi_{\widetilde \Phi,\widetilde \Phi}(||\mathbf{x-(r,0)}||)^{\frac 2 3} f(||\mathbf{x}||) d\mathbf{x}\nonumber\\
\xi_{\widetilde \Psi,\widetilde \Phi} (r) \beta & = &
\widetilde \lambda   (1-\widetilde p) \xi_{\widetilde \Psi,\widetilde \Psi} (r)^{\frac 2 3}
\int_{\mathbb R^2} 
\xi_{\widetilde \Psi,\widetilde \Phi}(||\mathbf{x}||)^{\frac 2 3}
\xi_{\widetilde \Psi,\widetilde \Phi}(||\mathbf{x-(r,0)}||)^{\frac 2 3} f(||\mathbf{x}||) d\mathbf{x}.  \nonumber
\\
\end{eqnarray}

\paragraph{Polynomial equation p-b1i}
Under the conditions for the polynomial setting, we get 
the p-b1 polynomial system
\begin{eqnarray*}
\beta \widetilde p v  & = &  \alpha (1-\widetilde p) w+ \beta \widetilde p  v^{\frac 2 3} w^{\frac 1 3}  \\ 
w^{\frac 2 3} & = &  z^{\frac 2 3}.
\end{eqnarray*}
Since $w$ and $z$ are positive, we get $z=w$. Hence,
\begin{eqnarray}
\label{eq:polsysb1i}
\beta \widetilde p v  & = &  \alpha (1-\widetilde p) w+ \beta \widetilde p  v^{\frac 2 3} w^{\frac 1 3} \nonumber\\ 
w (1-\widetilde p)^2 &= & c- \widetilde p^2 v -2\widetilde p(1-\widetilde p) w
\end{eqnarray}
and we get the following polynomial system (at the cost of introducing spurious solutions)
\begin{eqnarray}
v^{2} w (\alpha \widetilde \mu w -\beta)^3 & = & \left(v(\alpha \widetilde \mu w -\beta)-w\right)^3 \nonumber \\
v (\alpha \widetilde \mu w -\beta)^2 &  = & w(\alpha \widetilde \mu w (c \alpha \widetilde \mu -2\beta) +\beta^2).
\end{eqnarray}

\paragraph{Critical values}
Let us look at what happens when $\widetilde p$ tends to 0.  
It follows from the first moment equation  
that $w$ tends to $\frac\beta{\alpha\widetilde \mu}$, and from the second equation in
(\ref{eq:polsysb1i}) that the associated critical value of $\beta$ is 
$\beta_c=c\alpha\widetilde \mu =\alpha c q \mu$. 

We conclude that according to p-b1i, in the no-motion case, for all fixed $\beta$, $\lambda$ and $a$
such that the Boolean model with parameters $(\lambda,a)$ is supercritical,
$\beta_c=\alpha c q \mu$.

\subsection{Numerical Results}

\subsubsection{Densities}
Table \ref{table:sick_proportion-no-vel}
studies a case where the Boolean model is way above percolation
(we have $\widetilde \mu\sim 12.56> \widetilde \mu_c\sim 4.5$).
Table \ref{table:sick_proportion-no-vel-pt} studies situations
where $\widetilde \mu$ is close to (but above) the percolation threshold $\mu_c\sim 4.5$.

\begin{table}[h!]
	\centering
	\begin{tabular}{||c|c|c|c|c||} 
		\hline
                $\beta$              &  2   &   4  &  8   & 12  \\
		\hline\hline
		$p_{\mathrm {sim}}$  & 0.82 & 0.64 & 0.26 &  0  \\
		\hline
		\hline
		$p_{\mathrm {f-b1i}}$ & 0.82 & 0.64 & 0.33 &     \\
		\hline
		$p_{\mathrm {p-b1i}}$ & 0.80 & 0.62 & 0.30 & 0.020 \\
		\hline
		$p_{\mathrm {p-b1g1}}$& 0.81 & 0.63 & 0.30 & 0.022 \\
		\hline
		$p_{\mathrm {p-m2bi}}$ & 0.81 & 0.63 & 0.31 & 0.020   \\
		\hline
	$p_{{\mathrm {p-m}}\infty\mathrm{bi}}$& 0.82  & 0.64 & 0.32 & 0.032   \\
		\hline
		\hline
		$p_{\mathrm {hs}}$  & 0.84 & 0.67 & 0.35 & 0.045 \\
		\hline
	\end{tabular}
\caption{No-mobility case. Way above percolation.
Fraction of infected points ($\widetilde p$) obtained by simulation,
solution of the integral equation, the solution of the polynomial equation, and the 
high speed (hs) formula (\ref{eq:first-order}) Various heuristics are considered. This is for
$a=2$, $\lambda=1$ and $\alpha=1$.
Prediction is good. It is better when $p$ is not too close to 0. 
Note that b1 and m2bi provide similar results.}
\label{table:sick_proportion-no-vel}
	\centering
	\begin{tabular}{||c|c|c|c|c||} 
		\hline
                $\widetilde \mu$ & 4.54 & 4.80 & 5.31 & 6.28 \\
		\hline \hline
		$p_{\mathrm {sim}}$ & 0.07 &  &  0.23   & 0.39   \\
		\hline
	      $p_{{\mathrm {p-b1i}}}$ & 0.02 &   & 0.26 & 0.41  \\
		\hline
              $p_{{\mathrm {p-b1g1}}}$& 0.03 &   & 0.28 & 0.43  \\
		\hline
	       $p_{\mathrm {p-m2bi}}$ & 0.02 &      & 0.27 & 0.43 \\
		\hline\hline
		$q_{\mathrm {sim}}$ & 0.86 &  &  0.95   & 0.99 \\
		\hline
	\end{tabular}
	\caption{No-mobility case. Above and close to the percolation threshold.
                Fraction of infected points ($\widetilde p$) obtained by simulation,
		the functional equation and the polynomial equation.
		This is for $\alpha=1$ and $\beta = 3$. 
                Prediction is good again, except for very small values of $p$,
                with b1i and m2bi providing quite similar results, which slightly overestimate $p$.}
	\label{table:sick_proportion-no-vel-pt}
\end{table}

\subsubsection{Pair correlation functions}
The pair correlation functions of Heuristic B1 are depicted in Figure \ref{fig:1-3}.
\begin{figure}
\begin{center}
\includegraphics[width=.32\textwidth]{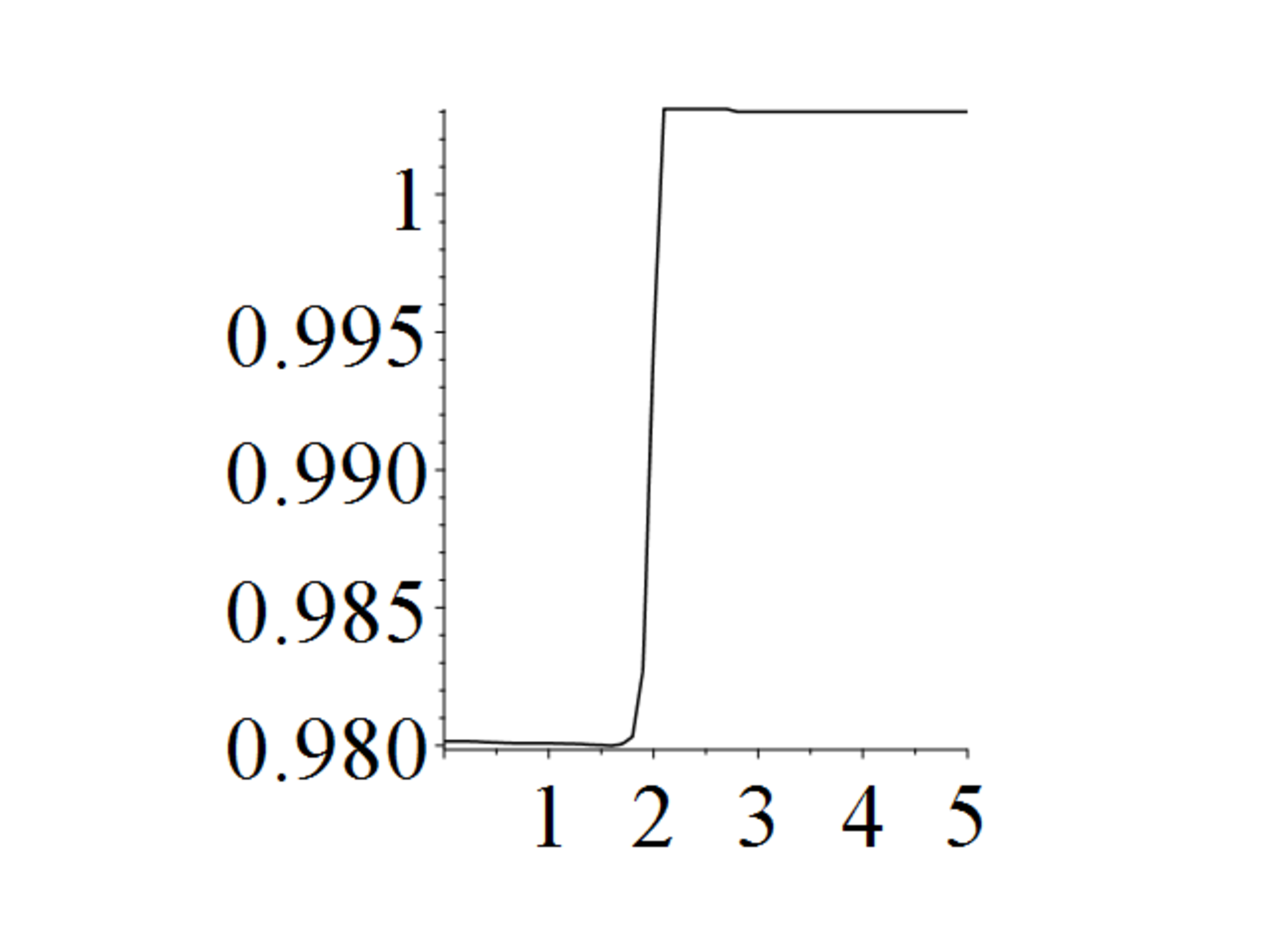}
\includegraphics[width=.32\textwidth]{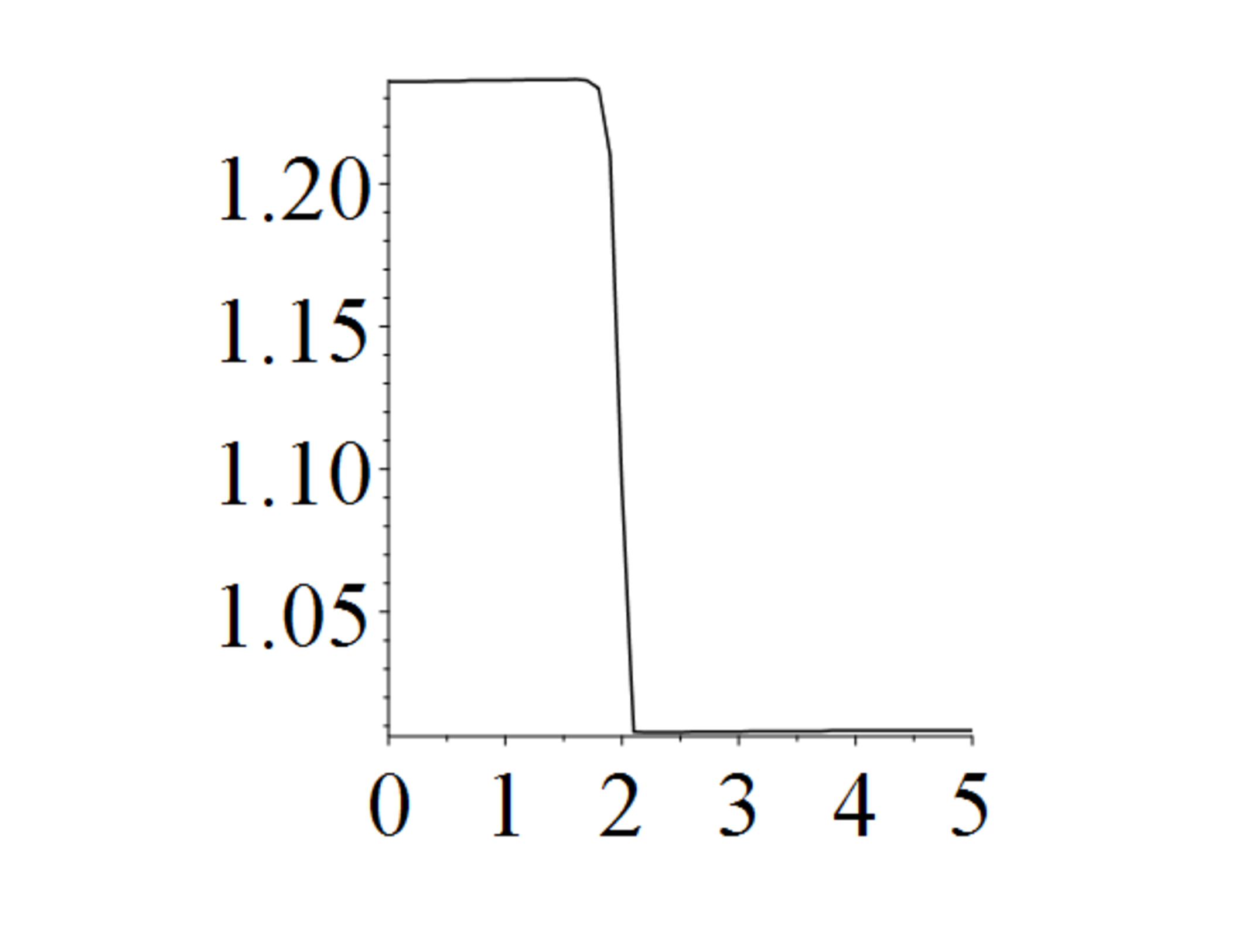}
\includegraphics[width=.32\textwidth]{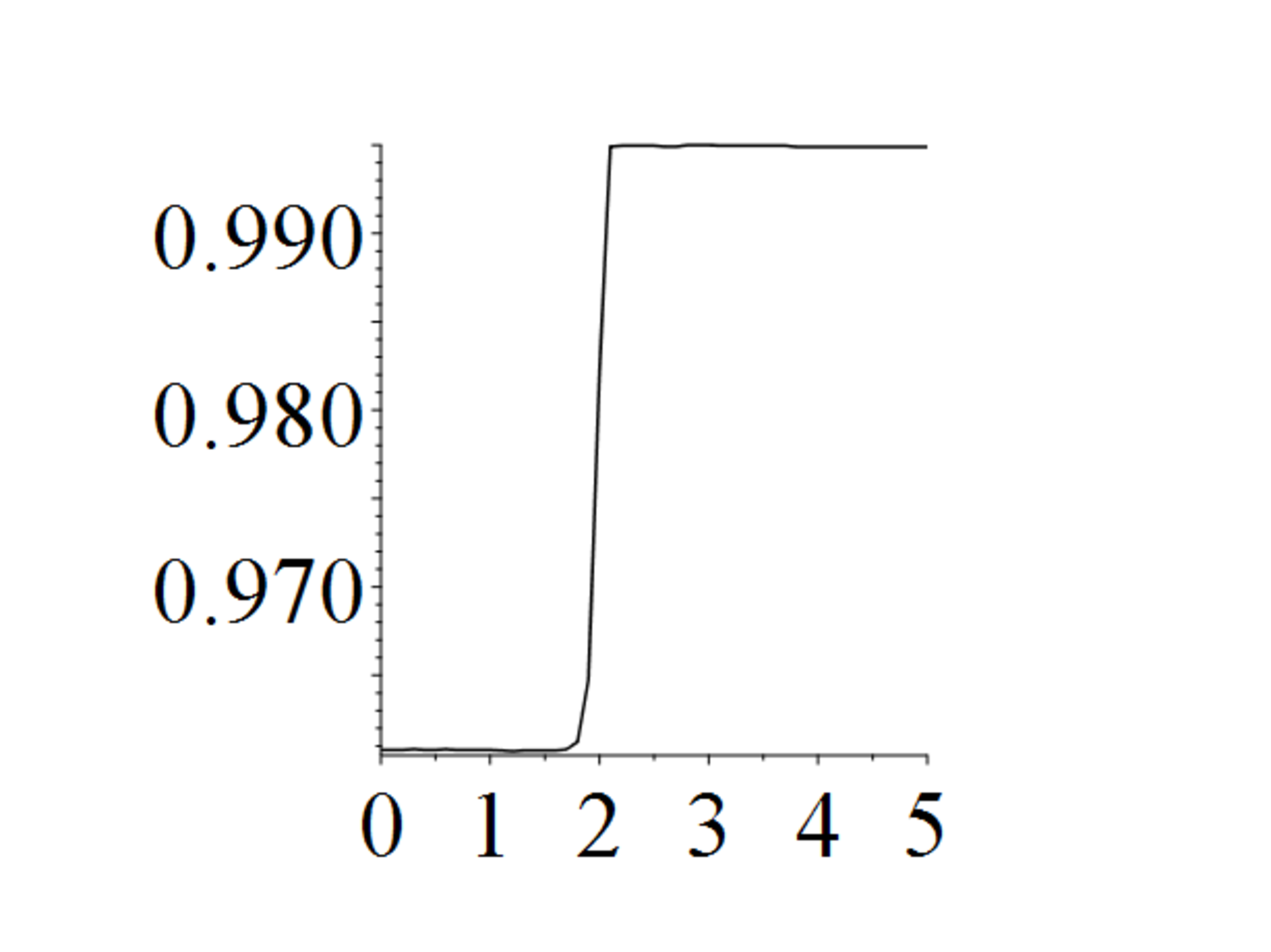}
\end{center}
\caption{Left: the $\xi_{\widetilde \Psi,\widetilde \Phi}(r)$ function;
Center: the $\xi_{\widetilde \Phi,\widetilde \Phi}(r)$ function;
Right: the $\xi_{\widetilde \Psi,\widetilde \Psi}(r)$ function. This is Heuristic f-b1.
}
\label{fig:1-3}
\end{figure}

\section{Model Variants}
\label{sec:var}

\subsection{Epidemic Model Variants}
In relation with certain epidemics,
the basic model described above is unsatisfactory in several ways.

First the SIS dynamics is not sufficient. SIR (or further variants like SEIR) would be more satisfactory.

In addition, points should be compartmented in at least two classes, say at risk $A$ and not at risk $N$.
points not at risk would go along the SIS cycle (or their variants).
In a first model, points at risk die when exposed to a high viral charge. 
For modelling this in the Markov SIS framework, an point $X$
of type $A$ in state $S$ has a death rate equal to  
$$ {\mathcal{D}}(X,\widetilde \Phi_t)= \sum_{Y \in \widetilde \Phi_t} g(||X-Y||).$$
Here, we could take $g=f$ or $g=\delta f$, with $\delta$ a positive constant. 
Another (and more favorable) variant is that where points at risk become sick at rate
$$ {\mathcal{D}}(X,\widetilde \Phi_t)= \sum_{Y \in \Phi_t} g(||X-Y||)$$
and where the result of sickness is death with probability $\nu$
and recovery with probability $1-\nu$. This compartmenting can be combined with SIR or extensions.

Since we will analyze stationary regimes, the compartmental model is more interesting
when there is a birth rate of points, with births representing
either births in the biological sense or arrivals from far away.
In the absence of births, all points
at risk eventually die and the steady state boils down to that of the SIS model
for points of type $N$, namely the basic model. A natural model
for births is that of a Poisson rain with intensity $\lambda$.
Newborns have a given probability to be in any state of $\{N,A\}\times\{I,S\}$.

Some of these variants will be discussed in the note. However the basic
model will be the basic one.

\subsection{Far Random Waypoint and Death}

One can combine the setting of Section \ref{sec:CSP}
with that of deaths as described in the last subsection.
The model features external births with a rate $\eta$, with say only susceptible
points, migration with a rate $\gamma$, and 
death with probability $\nu$ upon infection.
If $\nu=\eta=0$, we obtain the last model.
In this new case, the point process $\Xi$ (we should say $\widetilde \Xi$ but decided to drop the tilde)
is not Poisson any longer and both its intensity $\lambda$ and its pair correlation functions are unknown.
The equations are
\begin{eqnarray}
\label{eq:rcpmmmixed-glob-death}
p\xi_{ \Phi,\Phi} (r) (\beta+\gamma)
& =  &  p \gamma\nonumber\\
& & \hspace{-3.5cm}+ (1-p)(1-\nu)
\xi_{\Psi,\Phi} (r) \left(
f(r)+\int_{\mathbb R^2} \mu(\Phi)_{\Psi,\Phi}^{0,r}(x) f(||x||) {\rm d}x\right) \nonumber\\
\lambda p \xi_{\Psi,\Phi} (r) \beta +\lambda (1-p) \gamma + \eta
& =  & \nonumber\\
& & \hspace{-3.5cm} \lambda (1-p) 
\xi_{\Psi,\Psi} (r)
\left( \gamma+ \int_{\mathbb R^2} \mu(\Phi)_{\Psi,\Psi}^{0,r}(x) f(||x||) {\rm d}x \right),
\end{eqnarray}
We also have
\begin{eqnarray}
\beta & = & \lambda (1-p) (1-\nu)  2\pi
\int_{\mathbb R^+} \xi_{\Psi,\Phi} (r) f(r) r dr\nonumber\\
\eta & = & \lambda^2 (1-p) p \nu  2\pi
\int_{\mathbb R^+} \xi_{\Psi,\Phi} (r) f(r) r dr,
\end{eqnarray}
which is the first moment RCP in this case (the first equation says that
the rate of entrance in the susceptible state is the rate of infections that
do not lead to death, and the second one that the external birth rate is
the total death rate), and
\begin{eqnarray}
\xi_{\Xi,\Xi} (r)
& = &
(1-p)^2\xi_{\Psi,\Psi} (r)+
p^2 \xi_{\Phi,\Phi} (r)
+2p(1-p) \xi_{\Psi,\Phi} (r),
\end{eqnarray}
which is the conservation equation discussed above.
We also have
\begin{eqnarray}
\lambda^2 \gamma +\lambda \eta & = &  \gamma \lambda^2\xi_{\Xi,\Xi}(r)\nonumber\\
& & \hspace{-1cm} 
+\nu \lambda^2 (1-p)^2 \xi_{\Psi,\Psi} (r)
\int_{\mathbb R^2} \mu(\Phi)_{\Psi,\Psi}^{0,r}(x) f(||x||) {\rm d}x\nonumber\\
& & \hspace{-1cm} 
+\nu \lambda^2 (1-p) p \xi_{\Psi,\Phi} \left(f(r)+
\int_{\mathbb R^2} \mu(\Phi)_{\Psi,\Phi}^{0,r}(x) f(||x||) {\rm d}x\right).
\end{eqnarray}
This is obtained by balancing the mass birth and death rates in $\rho^{(2)}_{\Xi,\Xi}$.

So we have 4 unknown pair correlation functions, 2 unknown parameters ($\lambda$ and $p$),
and 6 equations relating them.

Associated with this model, one can define a functional equation based on any of 
the Bayes' heuristics discussed above as well as a polynomial equation.

\section{List of Conjectures}
\label{sec:lico}

We list here the main conjectures stated in the present paper:
\begin{itemize}
\item The repulsion conjecture in Section \ref{ss:ncfs};
\item The high velocity conjecture in Section \ref{ss:rtmm};
\item The $(\mu,\beta)$-phase diagram in Section \ref{ss:atpd}.
\end{itemize}
The first two conjectures are strongly backed by simulation.
The last one is only partly backed by simulation.
The three conjectures are mutually compatible: (i) the repulsion vanishes
in the high velocity regime; (ii) everywhere in the unsafe region,
high velocity leads to survival.

\section{Acknowledgements}
The authors would like to thank Charles Radin and Fabien Mathieu for their
valuable suggestions on this preprint. The authors acknowledge the support of the Texas Advanced Computing Center for providing access to computing resources that were used to carry out simulations.
F. Baccelli was supported by the ERC NEMO grant,
under the European Union's Horizon 2020 research and innovation programme,
grant agreement number 788851 to INRIA.

\appendix
\appendixpage
\section{Alternate Representations of the Basic Model}
\label{append:basic_model_alternate}
\subsection{Basic Model as a Collection of State Processes}
\label{append:basic_model_state_process}
Number the points of $\Xi_0$ with the integers.
Associate a piecewise constant, left continuous
stochastic process $S_i(t)$, with state space $\{0,1\}$, to point $i$,
with $0$ for susceptible and 1 for infected. 
For each $i$, the transition rates from one state to the other are time-point processes
with certain stochastic intensities.
The stochastic intensity of transitions of the state of point $i$ from 1 to 0 is
\begin{equation} b_i(t,\omega)= \beta 1_{S_i(t)=1}.\end{equation}
Let $V_i(t)$ denote the displacement of point $i$ of $\Xi$ at time $t$.
This random variable depends on the motion model. The position of point $i$ at time $t$ is hence
$$X_i(t)=X_i(0)+V_i(t).$$
Let 
$$I_{i,\Phi(t)}= I_{\Phi(t)}(X_i(t)) =\sum_{j\ne i} f(||X_i(t)-X_j(t)||) 1_{S_j(t)=1} $$
denote the infection rate seen by point $i$ at time $t$, whatever its current state.
The stochastic intensity of transitions of the state of point $i$ from 0 to 1 is
\begin{equation} a_i(t,\omega)= I_{i,\Phi(t)} 1_{S_i(t)=0} .\end{equation}
If there is no-motion, we have a countable collection of coupled Markov chains 
with transition rates 
\begin{equation} b_i(t,\omega)= \beta 1_{S_i(t)=1}\end{equation}
and
\begin{equation} a_i(t,\omega)= 
\left(\sum_{j\ne i} f(||X_i(0)-X_j(0)||) 1_{S_j(t)=1}\right) 1_{S_i(t)=0}.
\end{equation}

\subsection{Basic Model as an SIS epidemic on Random Geometric Graphs}
Here is an alternate interpretation of our model as an epidemic evolving on a random geometric graph: at time $t$, the vertex set of the graph consists of the points of $\Xi_t$, and an edge exists between two points $X, Y \in \Xi_t$ if $||X-Y|| < a$. Neighbours in the graph have the potential to infect each other. Points in the graph are either infected or susceptible. Susceptible points are infected at a rate that is proportional to the number of infected neighbours they have, and infected points recover at constant rate. A classical result related to random geometric graphs allows us to quantify the average degree of points, $\mu$ as $\mu = \lambda \pi a^2$. This random geometric graph is also known in literature as the Gilbert or Boolean model.\\~\\



\section{Proofs to Generalize the Graphical Representation}
\label{append:graphical_proofs}
\subsection{Random Waypoint Motion}
\label{appendss:rwp}
Under the random waypoint motion model, at times given by an exponential clock with rate $\gamma$, a point
jumps to another location keeping its SIS state. Then, whenever the displacements
are according to a smooth symmetrical distribution $D$ (e.g. Gaussian centered
independent per dimension), the three basic
properties (monotonicity, additivity, and self-duality) based on the graphical representation 
can be extended to this case.
The proof leverages a space-time random graph whose state at time $t$ consists of the positions 
of the particles at that time. Each particle stays put for an exponential time and then
jumps from its position to a new one (obtained by adding an independent random
variable with distribution $D$. This graph being given, one proceeds as in the basic
graphical representation by adding potential recovery epochs and pairwise infection epochs
on this space-time graph. Given an initial SIS state of the points, one then extends
the notion of causal infection path of the classical case to this case of point motion.
The proof of the desired properties are then obtained from the following facts:
(i) the space-time random graph is infinite and has a single connected component (this follows from fact that
for all pairs of points, the first time at which they are at distance less than $a$ is a.s.
finite due to the recurrence of the associated random walk in $\mathbb{R}^2$) 
(ii) the directed paths in the
associated space-time graphical representations are still ``reversible'', which is
instrumental in extending self-duality.
\subsection{Far Random Waypoint Motion}
\label{appendss:far_rwp}
We will state two lemmas whose proofs we defer to Appendix \ref{app:stein}.
\begin{Lem}
	\label{lem:stein_chen}
	Let $\Phi=\sum_n \delta_{X_n}$ be a stationary point process of intensity $\lambda$ in $\mathbb R^2$.
	If $\{D_n\}_n$ is a sequence of i.i.d. Gaussian vectors with each vector having 
	independent $\mathcal{N}(0,\sigma^2)$ entries, then,
	conditionally on $\Phi$, 
	$\Psi=\sum_n \delta_{X_n+D_n}$ is approximately Poisson homogeneous of intensity
	$\lambda$ in $\mathbb R^2$ when $\sigma$ is large.
\end{Lem}
The setting is that of the previous lemma. We now state:
\begin{Lem}
	\label{lem:mixing}
	Assume in addition that $\Phi$ is mixing
	and that its second factorial moment measure admits a density.
	Then, when $\sigma$ tends to infinity,
	(i) $\Psi$ and $\Phi$ are asymptotically independent;
	(ii) $S(\Phi)$ tends to $\lambda |A|$.
\end{Lem}

The properties mentioned in Section \ref{sec:graprep} (monotonicity, self-duality and additivity)
can now be established for the far random waypoint model using the below stated consequences of Lemmas \ref{lem:stein_chen} and \ref{lem:mixing}.
\begin{itemize}
	\item Let $\{\Phi_t\}_t$ and $\{\Psi_t\}$ be the steady state processes of the infected 
	and susceptible points. Let $\lambda p$ and $\lambda(1-p)$ denote their intensities.
	\item Let $t<u$ be fixed. Let $\Xi_k$, $k\ge 1$, be sub-point process of $\Phi_t$  
	which have $k$ jumps in $[t,u]$;
	\item Let $\widetilde \Phi_k$ denote the displaced points of $\Phi_k$; by the 
	arguments of the last two lemmas, $\widetilde \Phi_t$ is approximately a Poisson point
	process of intensity $\lambda p q(\gamma(u-t),k)$ with 
	$q(\gamma(u-t),k)$ the probability that a Poisson random variable of parameter $\gamma(u-t)$
	is equal to $k$. The point processes $\widetilde \Phi_k$ are independent, so that
	the "infected arrival" point process $\cal{I}=\sum_k \widetilde \Phi$ in this interval
	is Poisson of intensity $\lambda p \gamma(u-t)$;
	\item By the same argument, 
	the "susceptible arrival" point process $\cal S$ in this interval is Poisson of intensity
	$\lambda (1-p) \gamma(u-t)$;
	\item The point processes $\cal I$, $\cal S$, and $(\Phi_u,\Psi_u)$ are independent. 
\end{itemize}

\section{Proofs of Lemmas \ref{lem:stein_chen} and \ref{lem:mixing}}
\label{app:stein}

\begin{proof} (Lemma \ref{lem:stein_chen})
	We use the Stein-Chen method.
	Let $A$ be a bounded measurable set of $\mathbb R^2$.
	Let us first prove that, when $\sigma$ is large, $\Psi(A)$
	is approximately Poisson when $\sigma$ is large. We have
	$$\Psi(A)= \sum_n 1_{X_n+D_n\in A}$$
	and
	$$\mathbb P[ X_n+D_n\in A \mid \Phi] = \pi_n(\Phi)= \int_A f_\sigma(y-X_n) {\rm d}y$$
	with $f_\sigma(.)$ the Gaussian density alluded to above.
	The first observation is that the series
	$\sum_n \pi_n(\Phi)$ is a.s. convergent. Indeed,
	\begin{eqnarray*}
		\mathbb{E} [\sum_n \pi_n(\Phi)]
		& = & \int_{\mathbb R^2} \lambda {\rm d}x \int_A f_\sigma(y-x) {\rm d}y \\
		& = &  \lambda \int_A {\rm d}y \int_{\mathbb R^2} f_\sigma(y-x) {\rm d}x 
		= \lambda |A|, 
	\end{eqnarray*}
	where we successively used the Campbell-Mecke theorem, Fubini's theorem, and the fact that
	the integral of the Gaussian density is 1.
	Let 
	$$S(\Phi,A):= \sum_n \pi_n(\Phi).$$
	From the Stein-Chen theorem, conditionally on $\Phi$, the variation distance
	between $\Psi(A)$ and the Poisson law of parameter $S(\Phi,A)$ is bounded above by 
	$$ 2 \sum_n \pi_n^2(\Phi).$$
	We now prove that when $\sigma$ tends to infinity, the last sum tends to 0. 
	We have, with $x=(x_1,x_2)$ and $y=(y_1,y_2)$,
	\begin{eqnarray*}
		\mathbb{E} [\sum_n \pi_n^2(\Phi)]
		& = & \int_{\mathbb R^2} \lambda {\rm d}x \left(\int_A f_\sigma(y-x) {\rm d}y\right)^2 \\
		& = & \int_{\mathbb R^2} \lambda {\rm d}x \left(\int_A \frac 1 {2\pi \sigma^2} 
		e^{-\frac{(y_1-x_1)^2}{2\sigma^2}}
		e^{-\frac{(y_2-x_2)^2}{2\sigma^2}} {\rm d}y_1 {\rm d}y_2\right)^2 
	\end{eqnarray*}
	Using now the fact that
	\begin{eqnarray*}
		\int_A \frac 1 {2\pi \sigma^2} 
		e^{-\frac{(y_1-x_1)^2}{2\sigma^2}}
		e^{-\frac{(y_2-x_2)^2}{2\sigma^2}} {\rm d}y_1 {\rm d}y_2
		\le  |A|  \frac 1 {2\pi \sigma^2}, 
	\end{eqnarray*}
	we get
	\begin{eqnarray}
	\label{eq:nibo}
	\mathbb{E} [\sum_n \pi_n^2(\Phi)]
	& \le & \int_{\mathbb R^2} \lambda {\rm d}x |A|  \frac 1 {2\pi \sigma^2} 
	\int_A \frac 1 {2\pi \sigma^2} 
	e^{-\frac{(y_1-x_1)^2}{2\sigma^2}}
	e^{-\frac{(y_2-x_2)^2}{2\sigma^2}} {\rm d}y_1 {\rm d}y_2\nonumber \\
	& = & \lambda |A|^2 \frac 1 {2\pi \sigma^2}. 
	\end{eqnarray}
	So 
	\begin{eqnarray*}
		\lim_{_\sigma\to \infty} \mathbb{E} [\sum_n \pi_n^2(\Phi)]= 0,
	\end{eqnarray*}
	which shows that $\sum_n \pi_n^2(\Phi)$ tends to 0 in $L1$ when $\sigma\to \infty$.
	What we want here is almost sure convergence. 
	This follows from the following observation:
	for all families of non-negative random variables $Z_k$, if
	$ \sum_{k} \mathbb{E}[Z_k] <\infty,$ then $Z_k\to 0$ a.s. as $k\to \infty$. 
	This immediately follows from the fact that $\sum_k Z_k $ is then a.s. finite. 
	But we get from (\ref{eq:nibo}) that
	$$ 
	\sum_k \mathbb{E} [\sum_n \pi_n^2(\Phi,k)] <\infty,$$
	where $\pi(\Phi,\sigma)$ is what was denoted by $\pi(\Phi)$ above and where the
	dependency in $\sigma$ was made explicit. 
	Hence
	$$\lim_{k\to \infty} \sum_n \pi_n^2(\Phi,k) <\infty, \quad a.s.$$
	
	The fact that $\Psi$ is a conditional Poisson point process of intensity
	measure $S(\Phi,\cdot)$ follows from the easily proved fact that, conditionally on $\Phi$,
	for all $A$ and $B$ bounded and disjoint, $\Psi(A)$ and $\Psi(B)$ are independent.
\end{proof}

\begin{proof}(Lemma \ref{lem:mixing})
	The first property implies the second one since
	$$ S(\Phi)=\mathbb{E}[\Psi(A) \mid \Phi]=
	\mathbb{E}[\Psi(A)]= \lambda A.$$
	
	Let us now prove the first one.
	For all $A$ and $B$ bounded, we have
	\begin{eqnarray*}
		\mathbb{E}[\Phi(A) \Psi(B)] & = & 
		\mathbb{E}\left[\left(\sum_n 1_{X_n\in A}\right)\left(\sum_m 1_{X_m+D_m\in B} \right) \right]\\
		& = & 
		\mathbb{E}\left[\sum_n 1_{X_n\in A}1_{X_n+D_n\in B} \right]\\
		& + & 
		\mathbb{E}\left[\left(\sum_n 1_{X_n\in A}\right)\left(\sum_{m\ne n} 1_{X_m+D_m\in B} \right) \right].
	\end{eqnarray*}
	From the Campbell-Mecke theorem,
	\begin{eqnarray*}
		\mathbb{E}\left[\sum_n 1_{X_n\in A}1_{X_n+D_n\in B} \right]
		& = & \lambda \int_{x\in A} \int_z f(z) 1_{x+z\in B} {\rm d}z {\rm d}x\\
		& \le & \lambda |A| |B| \frac{1}{2\pi \sigma^2} \to_{\sigma\to \infty} 0.
	\end{eqnarray*}
	Let $c(u)$ denote the pair correlation function of $\Phi$. We have
	\begin{eqnarray*}
		\mathbb{E}\left[\Phi(A)\Psi(B)\right] & = &
		\mathbb{E}\left[\left(\sum_n 1_{X_n\in A}\right)\left(\sum_{m\ne n} 1_{X_m+D_m\in B} \right) \right]\\
		&=& \int_{x\in A}\int_y\int_{z} f(z) 1_{y+z\in B} \lambda^2 c(x-y) {\rm d}x {\rm d}y {\rm d}z\\
		&=& \int_{x\in A}\lambda {\rm d}x \int_y c(x-y) {\rm d}y \int_{t\in B} f(t-y) \lambda {\rm d}t\\
		&=& \int_{x\in A}\lambda {\rm d}x \int_{t\in B} \lambda {\rm d}t \int_y c(x-y) f(t-y) {\rm d}y.
	\end{eqnarray*}
	But for all fixed $x$ and $t$,
	\begin{eqnarray*}
		\int_y c(x-y) f(t-y) {\rm d}y
		=\int_y c(u-t+x) f(u) du
		= \mathbb{E}[ c(D_\sigma-t+x)],
	\end{eqnarray*}
	with $D_\sigma$ a Gaussian random variable as above.
	To analyze the limiting behavior of the last expression, consider the coupling
	$D_\sigma=\sigma D$. Using the dominated convergence theorem,
	$$\lim_{\sigma\to\infty}
	\mathbb{E}[ c(D_\sigma-t+x)] =
	\lim_{\sigma\to\infty}
	\mathbb{E}[\lim_{\sigma\to \infty} c(D_\sigma-t+x)] = 1.$$
	The fact that 
	$\lim_{\sigma\to \infty} c(D_\sigma-t+x) = 1$  a.s. follows from the mixing assumption.
	Using this and once more the dominated convergence theorem
	\begin{eqnarray*}
		\lim_{\sigma\to \infty} \mathbb{E}\left[\Phi(A)\Psi(B)\right] & = &
		\lim_{\sigma\to \infty} 
		\int_{x\in A}\lambda {\rm d}x \int_{t\in B} \lambda {\rm d}t 
		\mathbb{E}[ c(D_\sigma-t+x)] \\
		&=& \int_{x\in A}\lambda {\rm d}x \int_{t\in B} \lambda {\rm d}t 
		\lim_{\sigma\to \infty} 
		\mathbb{E}[ c(D_\sigma-t+x)] \\
		&=& \lambda |A| \lambda |B|=\mathbb{E} [\Phi(A)] 
		\mathbb{E} [\Psi(B)]. 
	\end{eqnarray*}
	
\end{proof}

\section{Heuristics in the Motion Case}
\subsection{Combining and Classifying Heuristics}
\label{append:motion_heuristics}
\subsubsection{Combinations}
In the Bayes' approach, one can replace the conditional independence step 
by the mean heuristic. For the geometric mean case, this leads to
\begin{eqnarray*}
	& &\hspace{-1cm} \left(\mu(\Phi)^{\mathbf{0,r}}_{\Psi,\Psi}(\mathbf{x}) 
	\xi_{\Psi,\Psi}(r) \lambda^2 (1-p)^2\right)^{k+2l}\\
	& = &
	\left(\mu(\Psi,\Psi)^{\mathbf{x}}_{\Phi}(\mathbf{0,r}) \lambda p\right)^{k} 
	\left(\mu(\Psi,\Phi)^{\mathbf{r}}_{\Psi}(\mathbf{0,x}) \lambda (1-p)\right)^{l} 
	\left(\mu(\Psi,\Phi)^{\mathbf{0}}_{\Psi}(\mathbf{r,x}) \lambda (1-p)\right)^{l} \\
	& = &
	\left(
	\xi_{\Psi,\Phi}(||\mathbf{x}||) \lambda(1-p) 
	\sqrt{\xi_{\Psi,\Phi}(||\mathbf{x-r}||) \xi_{\Psi,\Psi}(r)} \lambda(1-p)
	\lambda p \right)^{k} \\
	& & \left(
	\sqrt{\xi_{\Psi,\Psi}(r) \xi_{\Psi,\Phi}(||\mathbf{x}||) } \lambda(1-p)
	\xi_{\Psi,\Phi}(||\mathbf{x-r}||)\lambda p
	\lambda(1-p) \right)^{l} \\
	& & \left(
	\sqrt{\xi_{\Psi,\Psi}(r) \xi_{\Psi,\Phi}(||\mathbf{x-r}||)} \lambda(1-p)
	\xi_{\Psi,\Phi}(||\mathbf{x}||) \lambda p
	\lambda(1-p) \right)^{l}
\end{eqnarray*}
and
\begin{eqnarray*}
	& &\hspace{-1cm} \left(\mu(\Phi)^{\mathbf{0,r}}_{\Psi,\Phi}(\mathbf{x}) 
	\xi_{\Psi,\Phi}(r) \lambda^2 (1-p)p\right)^{k+2l}\\
	& = &
	\left(\mu(\Psi,\Phi)^{\mathbf{x}}_{\Phi}(\mathbf{0,r}) \lambda p\right)^{l} 
	\left(\mu(\Psi,\Phi)^{\mathbf{r}}_{\Phi}(\mathbf{0,x}) \lambda p\right)^{l} 
	\left(\mu(\Phi,\Phi)^{\mathbf 0}_{\Psi}(\mathbf{r,x}) \lambda (1-p)\right)^{k} \\
	& = &
	\left(
	\xi_{\Psi,\Phi}(||\mathbf{x}||) \lambda(1-p) 
	\sqrt{\xi_{\Phi,\Phi}(||\mathbf{x-r}||) \xi_{\Psi,\Phi}(r)} \lambda p
	\lambda p \right)^{l} \\
	& & \left(
	\xi_{\Psi,\Phi}(r) \lambda(1-p)
	\sqrt{\xi_{\Phi,\Phi}(||\mathbf{x-r}||) \xi_{\Psi,\Phi}(||\mathbf{x}||)} \lambda p
	\lambda p \right)^{l}\\
	& & \left(
	\xi_{\Psi,\Phi}(r) \lambda p
	\sqrt{\xi_{\Psi,\Phi}(||\mathbf{x}||) \xi_{\Phi,\Phi}(||\mathbf{x-r}||)}\lambda p
	\lambda(1-p) \right)^{k} .
\end{eqnarray*}
The rationale is as above.
This now leads to
\begin{eqnarray}
\label{eq:genb1d}
\mu(\Phi)^{\mathbf{0,r}}_{\Psi,\Psi}(\mathbf{x})
=\lambda p
\xi_{\Psi,\Phi}(||\mathbf{x}||)^{\frac{k+3l/2}{k+2l}}
\xi_{\Psi,\Phi}(||\mathbf{x-r}||)^{\frac{k/2+3l/2}{k+2l}}
\xi_{\Psi,\Psi}(r)^{-\frac{l+k/2}{k+2l}}
\end{eqnarray}
and
\begin{eqnarray}
\label{eq:genb1d2}
\mu(\Phi)^{\mathbf{0,r}}_{\Psi,\Phi}(\mathbf{x})
=\lambda p
\xi_{\Psi,\Phi}(||\mathbf{x}||)^{\frac{k/2+3l/2}{k+2l}}
\xi_{\Phi,\Phi}(||\mathbf{x-r}||)^{\frac{l+k/2}{k+2l}}
\xi_{\Psi,\Phi}(r)^{-\frac{l/2}{k+2l}}.
\end{eqnarray}

\subsubsection{Classification/Nomenclature}

\paragraph{Bayes Independent}
Here are a few special cases to be used below and named using the value of the ratio $l/k$.
We recall that a smaller ratio emphasizes the positive correlation (that between $\Psi$ and $\Psi$
or that between $\Phi$ and $\Phi$, whereas a bigger ratio emphasizes
the negative correlation between $\Phi$ and $\Psi$.
\begin{itemize}
	\item
	Heuristic {B0I} (Bayes 0 Independent) is for $k=\infty$ (maximal emphasis on negative correlation)
	and conditional independence:
	\begin{eqnarray} 
	\label{eq:heurb0i1}
	\mu(\Phi)^{\mathbf{0,r}}_{\Psi,\Psi}(\mathbf{x}) & = & \lambda p \frac{\xi_{\Psi,\Phi}(||\mathbf{x}||)
		\xi_{\Psi,\Phi}(||\mathbf{x-r}||)}
	{\xi_{\Psi,\Psi}(r)}
	\end{eqnarray}
	and
	\begin{eqnarray} 
	\label{eq:heurb0i2}
	\mu(\Phi)^{\mathbf{0,r}}_{\Psi,\Phi}(\mathbf{x}) & = & \lambda p  \xi_{\Psi,\Phi}(||\mathbf{x}||),
	\end{eqnarray}
	where we see that the influence of the susceptible at $\mathbf r$ is ignored.
	
	\item Heuristic {B1I} corresponds to $l=k=1$ (equal emphasis on positive and negative correlations)
	and conditional independence:
	\begin{eqnarray} 
	\label{eq:heurb1i1}
	\mu(\Phi)^{\mathbf{0,r}}_{\Psi,\Psi}(\mathbf{x}) & = & \lambda p \frac{\xi_{\Psi,\Phi}(||\mathbf{x}||)^{\frac 2 3}
		\xi_{\Psi,\Phi}(||\mathbf{x-r}||)^{\frac 2 3}}
	{\xi_{\Psi,\Psi}(r)^{\frac 1 3}} 
	\end{eqnarray}
	and
	\begin{eqnarray} 
	\label{eq:heurb1i2}
	\mu(\Phi)^{\mathbf{0,r}}_{\Psi,\Phi}(\mathbf{x}) & = & \lambda p  \frac{\xi_{\Psi,\Phi}(||\mathbf{x}||)^{\frac 2 3}
		\xi_{\Phi,\Phi}(||\mathbf{x-r}||)^{\frac 2 3}}
	{\xi_{\Psi,\Phi}(r)^{\frac 1 3}}.
	\end{eqnarray}
	
	\item Heuristic B$\frac 1 2$I corresponds to $2l=k=1$ (variant of the latter with a bit more emphasis on 
	negative correlation) and conditional independence:
	\begin{eqnarray} 
	\label{eq:heurb.5i1}
	\mu(\Phi)^{\mathbf{0,r}}_{\Psi,\Psi}(\mathbf{x}) & = & \lambda p \frac{\xi_{\Psi,\Phi}(||\mathbf{x}||)^{\frac 3 4}
		\xi_{\Psi,\Phi}(||\mathbf{x-r}||)^{\frac 3 4}}
	{\xi_{\Psi,\Psi}(r)^{\frac 1 2}} 
	\end{eqnarray}
	and
	\begin{eqnarray} 
	\label{eq:heurb.5i2}
	\mu(\Phi)^{\mathbf{0,r}}_{\Psi,\Phi}(\mathbf{x}) & = & \lambda p  \frac{\xi_{\Psi,\Phi}(||\mathbf{x}||)^{\frac 3 4}
		\xi_{\Phi,\Phi}(||\mathbf{x-r}||)^{\frac 1 2}}
	{\xi_{\Psi,\Phi}(r)^{\frac 1 4}}.
	\end{eqnarray}
	
	\item Heuristic B$\infty$I is for $l=\infty$ (all emphasis on positive correlations)
	and conditional independence:
	\begin{eqnarray} 
	\label{eq:heurbinftyi1}
	\mu(\Phi)^{\mathbf{0,r}}_{\Psi,\Psi}(\mathbf{x}) & = & \lambda p 
	\xi_{\Psi,\Phi}(||\mathbf{x}||)^{\frac 1 2} 
	\xi_{\Psi,\Phi}(||\mathbf{x-r}||)^{\frac 1 2}
	\end{eqnarray}
	and
	\begin{eqnarray} 
	\label{eq:heurbinftyi2}
	\mu(\Phi)^{\mathbf{0,r}}_{\Psi,\Phi}(\mathbf{x}) & = & \lambda p  
	\frac{\xi_{\Psi,\Phi}(||\mathbf{x}||)^{\frac 1 2}
		\xi_{\Phi,\Phi}(||\mathbf{x-r}||)}{
		\xi_{\Psi,\Phi}(r)^{\frac 1 2}}
	\end{eqnarray}
\end{itemize}

\paragraph{Geometric and Arithmetic Dependent}

\begin{itemize}
	\item Heuristic {G1} (Geometric Dependent 1) is for $\frac{\eta}{1-\eta}=1$
	\begin{eqnarray}
	\label{eq:heurgd11}
	\mu(\Phi)^{\mathbf{0,r}}_{\Psi,\Psi}(\mathbf{x})
	=\lambda p
	\xi_{\Psi,\Phi}(||\mathbf{x}||)^{\frac 1 2}
	\xi_{\Psi,\Phi}(||\mathbf{x-r}||)^{\frac 1 2}
	\end{eqnarray}
	and
	\begin{eqnarray}
	\label{eq:heurgd12}
	\mu(\Phi)^{\mathbf{0,r}}_{\Psi,\Phi}(\mathbf{x})
	=\lambda p
	\xi_{\Psi,\Phi}(||\mathbf{x}||)^{\frac 1 2}
	\xi_{\Phi,\Phi}(||\mathbf{x-r}||)^{\frac 1 2}.
	\end{eqnarray}
	\item Heuristic {G$\rho$} (Geometric Dependent $\rho$) is for $\frac{\eta}{1-\eta}=\rho$
	with equations given in (\ref{eq:gd1})-(\ref{eq:gd2}).
	Note that a bigger $\rho$ puts more emphasis on the positive correlation.
	\item Heuristic {A1} (Arithmetic Dependent 1) is for $\frac{\eta}{1-\eta}=1$
	\begin{eqnarray}
	\label{eq:heurad11}
	\mu(\Phi)^{\mathbf{0,r}}_{\Psi,\Psi}(\mathbf{x})
	=\lambda p \frac{
		\xi_{\Psi,\Phi}(||\mathbf{x}||)+
		\xi_{\Psi,\Phi}(||\mathbf{x-r}||)}
	{2}
	\end{eqnarray}
	and
	\begin{eqnarray}
	\label{eq:heurad12}
	\mu(\Phi)^{\mathbf{0,r}}_{\Psi,\Phi}(\mathbf{x})
	=\lambda p
	\frac{\xi_{\Psi,\Phi}(||\mathbf{x}||)
		+\xi_{\Phi,\Phi}(||\mathbf{x-r}||)}
	{2}.
	\end{eqnarray}
	\item Heuristic {A$\rho$} (Geometric Dependent $\rho$) is for $\frac{\eta}{1-\eta}=\rho$
	with equations given in (\ref{eq:ad1})-(\ref{eq:ad2}).
	Note that a bigger $\rho$ puts less emphasis on the positive correlation.
\end{itemize}

\paragraph{Combinations}
There are dependent version of the latter. For instance
\begin{itemize}
	\item Heuristic {B1G1} corresponds to B1 with independence replaced by  
	GD1, namely:
	\begin{eqnarray} 
	\label{eq:heurb1g11}
	\mu(\Phi)^{\mathbf{0,r}}_{\Psi,\Psi}(\mathbf{x}) & = &
	\lambda p \frac{\xi_{\Psi,\Phi}(||\mathbf{x}||)^{\frac 5 6}
		\xi_{\Psi,\Phi}(||\mathbf{x-r}||)^{\frac 2 3}}
	{\xi_{\Psi,\Psi}(r)^{\frac 1 2}} 
	\end{eqnarray}
	and
	\begin{eqnarray} 
	\label{eq:heurb1g12}
	\mu(\Phi)^{\mathbf{0,r}}_{\Psi,\Phi}(\mathbf{x}) & = &
	\lambda p  \frac{\xi_{\Psi,\Phi}(||\mathbf{x}||)^{\frac 2 3}
		\xi_{\Phi,\Phi}(||\mathbf{x-r}||)^{\frac 1 2}}
	{\xi_{\Psi,\Phi}(r)^{\frac 1 6}}.
	\end{eqnarray}
\end{itemize}

\paragraph{Mixtures}
One can also take mixtures of the above cases.
The numbering here is w.r.t. the number of terms in the mixture.

\begin{itemize}
	\item Heuristic {M2BI} (Mixture of two types of Bayes Independent)
	is the following linear combination of B0I and B$\infty$I:
	\begin{eqnarray} 
	\label{eq:heurm2bi1}
	\mu(\Phi)^{\mathbf{0,r}}_{\Psi,\Psi}(\mathbf{x}) & = & \frac{\lambda p} 2 \frac{\xi_{\Psi,\Phi}(||\mathbf{x}||)
		\xi_{\Psi,\Phi}(||\mathbf{x-r}||)}
	{\xi_{\Psi,\Psi}(r)} \nonumber\\
	&& +\frac {\lambda p} 2 \xi_{\Psi,\Phi}(||\mathbf{x}||)^{\frac 1 2} \xi_{\Psi,\Phi}(||\mathbf{x-r}||)^{\frac 1 2}
	\end{eqnarray}
	and
	\begin{eqnarray} 
	\label{eq:heurm2bi2}
	\mu(\Phi)^{\mathbf{0,r}}_{\Psi,\Phi}(\mathbf{x}) & = & \frac{\lambda p} 2 \xi_{\Psi,\Phi}(||\mathbf{x}||)
	\nonumber\\
	&& +\frac {\lambda p} 2
	\frac{ \xi_{\Psi,\Phi}(||\mathbf{x}||)^{\frac 1 2} \xi_{\Phi,\Phi}(||\mathbf{x-r}||)}
	{\xi_{\Psi,\Phi}(r)^{\frac 1 2}}.
	\end{eqnarray}
	
	\item Heuristic {M3BI} (Mixture of three types of Bayes Independent)
	is the following linear combination of B0I, B1I and B$\infty$I:
	\begin{eqnarray} 
	\label{eq:heurm3bi1}
	\mu(\Phi)^{\mathbf{0,r}}_{\Psi,\Psi}(\mathbf{x}) & = &
	\frac{\lambda p} 3 \frac{\xi_{\Psi,\Phi}(||\mathbf{x}||)
		\xi_{\Psi,\Phi}(||\mathbf{x-r}||)}
	{\xi_{\Psi,\Psi}(r)}
	\nonumber\\ & & + \frac{\lambda p} 3
	\frac{\xi_{\Psi,\Phi}(||\mathbf{x}||)^{\frac 2 3}
		\xi_{\Psi,\Phi}(||\mathbf{x-r}||)^{\frac 2 3}}
	{\xi_{\Psi,\Psi}(r)^{\frac 1 3}} 
	\nonumber\\
	&& +\frac {\lambda p} 3 \xi_{\Psi,\Phi}(||\mathbf{x}||)^{\frac 1 2} \xi_{\Psi,\Phi}(||\mathbf{x-r}||)^{\frac 1 2}
	\end{eqnarray}
	and
	\begin{eqnarray} 
	\label{eq:heurm3bi2}
	\mu(\Phi)^{\mathbf{0,r}}_{\Psi,\Phi}(\mathbf{x}) & = & \frac{\lambda p} 3  \xi_{\Psi,\Phi}(||\mathbf{x}||)
	\nonumber\\ & & +\frac{\lambda p} 3
	\frac{\xi_{\Psi,\Phi}(||\mathbf{x}||)^{\frac 2 3}
		\xi_{\Phi,\Phi}(||\mathbf{x-r}||)^{\frac 2 3}}
	{\xi_{\Psi,\Phi}(r)^{\frac 1 3}}
	\nonumber\\
	&& +\frac {\lambda p} 3
	\frac{ \xi_{\Psi,\Phi}(||\mathbf{x}||)^{\frac 1 2} \xi_{\Phi,\Phi}(||\mathbf{x-r}||)}
	{\xi_{\Psi,\Phi}(r)^{\frac 1 2}}.
	\end{eqnarray}
	
	\item Heuristic {M$\infty$BI} mixes all Bayes' Independent formulas (there is one for all $\eta$) 'equally':
	Let $\eta=(k+l)/(k+2l)$.
	By passing to the continuum, one can use (\ref{eq:genb1}) and (\ref{eq:genb2}) 
	to get the following integral form:
	\begin{eqnarray}
	\label{eq:genintminftybi1}
	\hspace{-.6cm}\mu(\Phi)^{\mathbf{0,r}}_{\Psi,\Psi}(\mathbf{x})
	=\lambda p \int_0^1
	\xi_{\Psi,\Phi}(||\mathbf{x}||)^{\eta}
	\xi_{\Psi,\Phi}(||\mathbf{x-r}||)^{\eta}
	\xi_{\Psi,\Psi}(r)^{1-2\eta} d\eta
	\end{eqnarray}
	and
	\begin{eqnarray}
	\label{eq:genminftybi2}
	\hspace{-.6cm}\mu(\Phi)^{\mathbf{0,r}}_{\Psi,\Phi}(\mathbf{x})
	=\lambda p \int_0^1
	\xi_{\Psi,\Phi}(||\mathbf{x}||)^{\eta}
	\xi_{\Phi,\Phi}(||\mathbf{x-r}||)^{2(1-\eta)}
	\xi_{\Psi,\Phi}(r)^{\eta-1} d\eta.\nonumber\\
	\end{eqnarray}
	
	\item Heuristic {M$\infty$BG1} mixes all Bayes' formulas with geometric mean of parameter 1 'equally',
	i.e., if $\eta=(k+l)/(k+2l)$, then
	\begin{eqnarray}
	\label{eq:genminftybg11}
	\hspace{-.6cm}\mu(\Phi)^{\mathbf{0,r}}_{\Psi,\Psi}(\mathbf{x})
	=\lambda p \int_0^1
	\xi_{\Psi,\Phi}(||\mathbf{x}||)^{\frac 1 2+ \frac \eta 2}
	\xi_{\Psi,\Phi}(||\mathbf{x-r}||)^{ 1 -\frac \eta 2}
	\xi_{\Psi,\Psi}(r)^{-\frac 1 2} d\eta
	\end{eqnarray}
	and
	\begin{eqnarray}
	\label{eq:genminftybg12}
	\hspace{-.6cm}\mu(\Phi)^{\mathbf{0,r}}_{\Psi,\Phi}(\mathbf{x})
	=\lambda p \int_0^1
	\xi_{\Psi,\Phi}(||\mathbf{x}||)^{ 1 -\frac \eta 2}
	\xi_{\Phi,\Phi}(||\mathbf{x-r}||)^{\frac 1 2}
	\xi_{\Psi,\Phi}(r)^{\frac \eta 2 -\frac 1 2} d\eta.\nonumber\\
	\end{eqnarray}
	
\end{itemize}
\subsection{Polynomial Heuristics}
\label{append:motion_heuristics_polynomial}

\subsubsection{Heuristic B1G1}

\paragraph{Functional equation}
Under Heuristic B1G1, the version of (\ref{eq:rcpmmmixed-glob}) is
\begin{eqnarray}
\label{eq:recursive-vb1d-2}
(\beta+\gamma) p
\xi_{\Phi,\Phi} (r) 
& =  & p \gamma
+ (1-p) \xi_{\Psi,\Phi} (r) f(r)\nonumber\\
& &\hspace{-2cm}  + \lambda \left(1-p\right) p 
\xi_{\Psi,\Phi}(r)^{\frac  5 6}
\int_{\mathbb R^2} 
\xi_{\Psi,\Phi}(||\mathbf{x}||)^{\frac  2 3 }
\xi_{\Phi,\Phi}(||\mathbf{x-r}||)^{\frac 1 2} f(||\mathbf{x}||) d\mathbf{x}\nonumber\\
\beta p \xi_{\Psi,\Phi} (r)  & = &
(1-p) \gamma \left( \xi_{\Psi,\Psi} (r) -1\right)
\nonumber \\
& &\hspace{-2cm} +  \lambda \left(1-p\right) p
\xi_{\Psi,\Psi} (r)^{\frac 1 2 }\int_{\mathbb R^2}
\xi_{\Psi,\Phi}(||\mathbf{x}||)^{\frac 5 6}\xi_{\Psi,\Phi}(||\mathbf{x-r}||)^{\frac 2 3}
f(||\mathbf{x}||) d\mathbf{x}.\nonumber\\
\end{eqnarray}

\paragraph{Polynomial equation}
The associated polynomial equations read
\begin{eqnarray}
(\gamma+\beta) pv & = &  \gamma p +  \alpha (1-p)  w+
\beta p  v^{\frac 1 2} w^{\frac 1 2}  \nonumber \\
\beta p w   & = &  (1-p) \gamma (z-1)
+ \beta p z^{\frac 1 2 } w^{\frac 1 2}\nonumber\\
\beta & = & (1-p) \alpha \mu w\nonumber
\\
1  & = & (1-p)^2 z + 2p(1-p)w + p^2 v. 
\label{eq:fantastic-b1d}
\end{eqnarray}

\subsubsection{Heuristic M$\infty$BI}
\paragraph{Functional equation}

Under Heuristic M$\infty$BI, the version of (\ref{eq:rcpmmmixed-glob}) is
\begin{eqnarray}
\label{eq:recursive-v-2mbinfty}
(\beta+\gamma) p
\xi_{\Phi,\Phi} (r) 
& =  & p \gamma
+ (1-p) \xi_{\Psi,\Phi} (r) f(r)\nonumber\\
& &\hspace{-4cm}
+ \lambda \left(1-p\right) p 
\int_{\mathbb R^2} \int_0^1
\xi_{\Psi,\Phi}(||x||)^{\eta} 
\xi_{\Phi,\Phi}(||x-(r,0)||)^{2 (1-\eta)} 
\xi_{\Psi,\Phi}(r)^{\eta}
f(||x||) d\eta {\rm d}x\nonumber\\
\beta p \xi_{\Psi,\Phi} (r)  & = &
(1-p) \gamma \left( \xi_{\Psi,\Psi} (r) -1\right)
\nonumber \\
& &\hspace{-4cm} 
+ \lambda \left(1-p\right) p 
\xi_{\Psi,\Psi}(r)
\int_{\mathbb R^2} \int_0^1
\xi_{\Psi,\Phi}(||x||)^{\eta}
\xi_{\Psi,\Phi}(||x-(r,0)||)^{\eta} f(||x||)
\xi_{\Psi,\Psi}(r)^{1-2\eta} d\eta {\rm d}x.\nonumber\\
\end{eqnarray}
This should again be complemented by (\ref{eq:recursivec1-mot}) and (\ref{eq:recursivec2-mot}).

\paragraph{Polynomial equation}

The associated polynomial\footnote{This is an abuse of terminology here} equations read
\begin{eqnarray}
(\gamma+\beta) pv & = &\gamma p +\alpha (1-p)  w
+ \beta p \frac 1 w \frac{v^2-w^2}{\log(v^2)-\log(w^2)}  
\nonumber
\\
\beta p w   & = &  (1-p) \gamma (z-1)
+ 
\beta p \frac 1 w  \frac{w^2-z^2}{\log(w^2)-\log(z^2)} 
\nonumber\\
\beta & = & (1-p) \alpha \mu w\nonumber
\\
1  & = & (1-p)^2 z + 2p(1-p)w + p^2 v. 
\label{eq:fantasticmbinftylog}
\end{eqnarray}
This should again be complemented by (\ref{eq:recursivec1-mot}) and (\ref{eq:recursivec2-mot}).

\subsubsection{Heuristic M$\infty$BG1}

\paragraph{Functional equation}

Under Heuristic M$\infty$BG1, the version of (\ref{eq:rcpmmmixed-glob}) is
\begin{eqnarray}
\label{eq:recursive-v-2mbdinfty}
(\beta+\gamma) p
\xi_{\Phi,\Phi} (r) 
& =  & p \gamma
+ (1-p) \xi_{\Psi,\Phi} (r) f(r)\nonumber\\
& &\hspace{-4cm}
+ \lambda \left(1-p\right) p 
\xi_{\Psi,\Phi}(r)
\int_{\mathbb R^2} \int_0^1
\xi_{\Psi,\Phi}(||x||)^{1-\frac \eta 2} 
\xi_{\Phi,\Phi}(||x-(r,0)||)^{\frac 1 2 } 
\xi_{\Psi,\Phi}(r)^{\frac \eta 2 -\frac 1 2}
f(||x||) d\eta dx\nonumber\\
\beta p \xi_{\Psi,\Phi} (r)  & = &
(1-p) \gamma \left( \xi_{\Psi,\Psi} (r) -1\right)
\nonumber \\
& &\hspace{-4cm} 
+ \lambda \left(1-p\right) p 
\xi_{\Psi,\Psi}(r)
\int_{\mathbb R^2} \int_0^1
\xi_{\Psi,\Phi}(||x||)^{\frac 1 2 +\frac \eta 2}
\xi_{\Psi,\Phi}(||x-(r,0)||)^{1-\frac \eta 2} f(||x||)
\xi_{\Psi,\Psi}(r)^{-\frac 1 2} d\eta dx.\nonumber\\
\end{eqnarray}
This should again be complemented by (\ref{eq:recursivec1-mot}) and (\ref{eq:recursivec2-mot}).

\paragraph{Polynomial equation}

The associated polynomial equations read
\begin{eqnarray}
(\gamma+\beta) pv & = &\gamma p +\alpha (1-p)  w
+ \beta p  v^{\frac 1 2} w^{\frac 1 2} 
\nonumber
\\
\beta p w   & = &  (1-p) \gamma (z-1)
+ 
\beta p z^{\frac 1 2} w^{\frac 1 2}
\nonumber\\
\beta & = & (1-p) \alpha \mu w\nonumber
\\
1  & = & (1-p)^2 z + 2p(1-p)w + p^2 v. 
\label{eq:fantasticmbinfty}
\end{eqnarray}
So p-b1g1 and p-m$\infty$bg1 are the same.

\section{On the Boolean Cluster above Percolation}
\label{append:pcfic}
First, we recall some classical results from \cite{Stoyan}.
For a Poisson point process of intensity $\lambda$ in $\mathbb R^2$,
the Boolean model of radius $a$ percolates for $\mu=\lambda \pi a^2> \mu_c$,
with $\mu_c\sim 4.5$. That is, if $\lambda=1$, the model percolates for $a>a_c\sim 1.2$.
If $a=1$ the model percolates for $\lambda>\lambda_c=1.4$.

Consider the Poisson point process $\widetilde \Xi$ of intensity $\widetilde \lambda$
under its double Palm distribution at 0 and $R=(0,r)$.
Consider the Boolean model with radius $a$ on $\Xi$.
Let $\pi(r)$ denote the probability that $0$ and $R$ are connected.
When denoting by $x\sim y$ the event that $x$ and $y$ are connected, we have
$$1-\pi(r)= \mathbb{P} ( \cap_{i: X_i\in B(0,a)} \{X_i\sim R)\}^c).$$
Under the approximation stating that the events $\{X_i\sim R)\}$ are conditionally independent
given $\Xi\cap B(0,a)$, with respective probabilities $\pi(||X_i-R||)$, we get
$$\pi(r)= 1- \exp\left(-\lambda \int_{v=0}^a \int_{\theta=0}^{2\pi} 
\pi\left(\sqrt{r^2+v^2+2rv\cos(\theta)}\right) vdv d\theta\right), \quad r>a,$$
with $\pi(r)=1$ for $r\le a$.
For $r$ large, we get that $q=\pi(\infty)$ solves the Lambert equation
\begin{equation}
\label{eq:lambert}
q=1-\exp(-\widetilde \mu q),
\end{equation}
which is also the equation for the probability of survival of a branching process with
Poisson progeny of parameter $\widetilde \mu=\widetilde \lambda\pi a^2$.
This equation has no positive solution when $\widetilde \mu<1$ and a single
solution $q_{\widetilde \mu}$ in $(0,1)$ for $\widetilde \mu>1$, which can be interpreted as
the probability of survival. In contrast the probability of extinction (of
the branching process) is $1-q_{\widetilde \mu}$. It follows that, in this approximation, the density of points
that lie in finite clusters is $1-q_{\widetilde \mu}$ so that the intensity of the infinite cluster
(seen as a stationary point process $\Xi$) is $q_{\widetilde \mu}$.
This leads to the following approximation for the pair correlation
function of $\Xi$: 
\begin{equation}
\xi_{\Xi,\Xi}(r)=\frac{\pi(r)}{q_{\widetilde \mu}}.
\end{equation}
In particular, in this approximation
\begin{equation}
\xi_{\Xi,\Xi}(r)=c_{\widetilde \mu}:=\frac{1}{q_{\widetilde \mu}}, \quad \forall r<a.
\end{equation}

\section{Heuristics For the No-Motion Case}
\label{append:no_motion_heuristics}
\subsection{Heuristic G1}
\paragraph{Functional equation f-g1}
\begin{eqnarray}
\label{eq:inteq-g1}
\widetilde p \xi_{\widetilde \Phi,\widetilde \Phi} (r) \beta
& =  &  (1-\widetilde p) \xi_{\widetilde \Psi,\widetilde \Phi} (r) f(r)\nonumber\\
& &  \hspace{-1cm}
+  \widetilde \lambda (1-\widetilde p) \widetilde p
\xi_{\widetilde \Psi,\widetilde \Phi} (r) 
\int_{\mathbb R^2} 
\xi_{\widetilde \Psi,\widetilde \Phi}(||\mathbf{x}||)^{\frac 1 2} 
\xi_{\widetilde \Phi,\widetilde \Phi}(||\mathbf{x-(r,0)}||)^{\frac 1 2} f(||\mathbf{x}||) d\mathbf{x}\nonumber\\
\xi_{\widetilde \Psi,\widetilde \Phi} (r) \beta & = &
\widetilde \lambda   (1-\widetilde p) \xi_{\widetilde \Psi,\widetilde \Psi} (r)
\int_{\mathbb R^2} 
\xi_{\widetilde \Psi,\widetilde \Phi}(||\mathbf{x}||)^{\frac 1 2}
\xi_{\widetilde \Psi,\widetilde \Phi}(||\mathbf{x-(r,0)}||)^{\frac 1 2} f(||\mathbf{x}||) d\mathbf{x}.  \nonumber
\\
\end{eqnarray}
\paragraph{Polynomial equation p-g1}
Under the conditions for the polynomial setting, we get 
the p-g1 polynomial system
\begin{eqnarray*}
	\beta \widetilde p v  & = &  \alpha (1-\widetilde p) w+ \beta \widetilde p  v^{\frac 1 2} w^{\frac 1 2}  \\ 
	\beta w & = &  \beta z.
\end{eqnarray*}
This equation coincides with that of 
p-mb$\infty$g1 and p-b1g1 and will be studied below
(see Equation \ref{eq:fantastic-mbdinfty}).

\subsection{Heuristic B1G1}
\paragraph{Functional equation f-b1g1}
The functional equation reads
\begin{eqnarray}
\label{eq:inteq-b1g1}
\widetilde p \xi_{\widetilde \Phi,\widetilde \Phi} (r) \beta
& =  &  (1-\widetilde p) \xi_{\widetilde \Psi,\widetilde \Phi} (r) f(r)\nonumber\\
& &  \hspace{-1cm}
+  \lambda (1-\widetilde p) \widetilde p
\xi_{\widetilde \Psi,\widetilde \Phi} (r)^{\frac 5 6} 
\int_{\mathbb R^2} 
\xi_{\widetilde \Psi,\widetilde \Phi}(||\mathbf{x}||)^{\frac 2 3} 
\xi_{\widetilde \Phi,\widetilde \Phi}(||\mathbf{x-(r,0)}||)^{\frac 1 2} f(||\mathbf{x}||) d\mathbf{x}\nonumber\\
\xi_{\widetilde \Psi,\widetilde \Phi} (r) \beta & = &
\widetilde \lambda   (1-\widetilde p) \xi_{\widetilde \Psi,\widetilde \Psi} (r)^{\frac 1 2}
\int_{\mathbb R^2} 
\xi_{\widetilde \Psi,\widetilde \Phi}(||\mathbf{x}||)^{\frac 5 6}
\xi_{\widetilde \Psi,\widetilde \Phi}(||\mathbf{x-(r,0)}||)^{\frac 2 3} f(||\mathbf{x}||) d\mathbf{x}.  \nonumber
\\
\end{eqnarray}

\paragraph{Polynomial equation p-b1g1}
Under the conditions for the polynomial setting, we get 
the p-b1g1 polynomial system
\begin{eqnarray*}
	\beta \widetilde p v  & = &  \alpha (1-\widetilde p) w+ \beta \widetilde p  v^{\frac 1 2} w^{\frac 1 2}  \\ 
	\beta w & = &  \beta z^{\frac 1 2 }w^{\frac 1 2}.
\end{eqnarray*}
This polynomial equation coincides with that of p-g1 above and 
that of p-mb$\infty$g1 below.
It will be studied in Equation \ref{eq:fantastic-mbdinfty})
in the p-mb$\infty$g1 paragraph.

\subsection{Heuristic M2BI}
\label{hmb2nomot}

\paragraph{Functional equation f-m2bi}
Equations (\ref{eq:rcpmm2}) and (\ref{eq:rcpmm3}) 
lead to the following system of integral equations (f-m2bi):
\begin{eqnarray}
\label{eq:inteqm2bi}
\widetilde p \xi_{\widetilde \Phi,\widetilde \Phi} (r) \beta
& =  &  (1-\widetilde p) \xi_{\widetilde \Psi,\widetilde \Phi} (r) f(r)\nonumber\\
& & 
+ \frac 1 2 \lambda (1-\widetilde p) \widetilde p
\xi_{\widetilde \Psi,\widetilde \Phi} (r) 
\int_{\mathbb R^2} 
\xi_{\widetilde \Psi,\widetilde \Phi}(||\mathbf{x}||)
f(||\mathbf{x}||) d\mathbf{x} \nonumber\\
& & 
+ \frac 1 2 \lambda (1-\widetilde p) \widetilde p
\xi_{\widetilde \Psi,\widetilde \Phi} (r)^{\frac 1 2}
\int_{\mathbb R^2} 
\xi_{\widetilde \Psi,\widetilde \Phi}(||\mathbf{x}||^{\frac 1 2} 
\xi_{\widetilde \Phi,\widetilde \Phi}(||\mathbf{x-r}||) f(||\mathbf{x}||) d\mathbf{x}\nonumber\\
\xi_{\widetilde \Psi,\widetilde \Phi} (r) \beta & = &\frac 1 2 \lambda   (1-\widetilde p)
\int_{\mathbb R^2} 
\xi_{\widetilde \Psi,\widetilde \Phi}(||\mathbf{x}||)\xi_{\widetilde \Psi,\widetilde \Phi}(||\mathbf{x-r}||) f(||\mathbf{x}||) d\mathbf{x}  \nonumber\\
& & + \frac 1 2 \lambda   (1-\widetilde p)
\xi_{\widetilde \Psi,\widetilde \Psi} (r)
\int_{\mathbb R^2} 
\xi_{\widetilde \Psi,\widetilde \Phi}(||\mathbf{x}||^{\frac 1 2}\xi_{\widetilde \Psi,\widetilde \Phi}(||\mathbf{x-r}||^{\frac 1 2} f(||\mathbf{x}||) d\mathbf{x}.  \nonumber
\\
\end{eqnarray}

\paragraph{Iterative scheme s-m2bi}
The system (\ref{eq:inteqm2bi})
in turn leads to the following iterative integral equation scheme with again
the two "state" functions $\xi^{(n)}_{\widetilde \Phi,\widetilde \Phi} (.)$ and $\xi^{(n)}_{\widetilde \Psi,\widetilde \Phi} (.)$:
\begin{eqnarray}
\label{eq:recursivem2bi}
\xi^{(n+1)}_{\widetilde \Phi,\widetilde \Phi} (r) 
& =  & \frac {1}{\beta} \frac{1-\widetilde p^{(n)}}{\widetilde p^{(n)}} \xi^{(n)}_{\widetilde \Psi,\widetilde \Phi} (r) f(r)
\nonumber\\
&   &\hspace{-2cm} + \frac 1 2 \frac{\widetilde \lambda}{\beta}
\left(1-\widetilde p^{(n)}\right) \int_{\mathbb R^2} 
\xi^{(n)}_{\widetilde \Psi,\widetilde \Phi}(||\mathbf{x}||)
\xi^{(n)}_{\widetilde \Phi,\widetilde \Phi}(||\mathbf{x-r}||)
f(||x||) dx\nonumber\\
& &  \hspace{-2cm}
+ \frac 1 2 \lambda \left(1-\widetilde p^{(n)}\right) \widetilde p^{(n)}
\xi^{(n)}_{\widetilde \Psi,\widetilde \Phi} (r) 
\int_{\mathbb R^2} 
\xi^{(n)}_{\widetilde \Psi,\widetilde \Phi}(||\mathbf{x}||) 
\xi^{(n)}_{\widetilde \Phi,\widetilde \Phi}(||\mathbf{x-r}||) f(||\mathbf{x}||) d\mathbf{x}\nonumber\\
\xi^{(n+1)}_{\widetilde \Psi,\widetilde \Phi} (r)  & = & \frac 1 2  \frac {\widetilde \lambda} {\beta}  \left(1-\widetilde p^{(n)}\right)
\int_{\mathbb R^2}
\xi^{(n)}_{\widetilde \Psi,\widetilde \Phi}(||\mathbf{x}||)\xi^{(n)}_{\widetilde \Psi,\widetilde \Phi}(||\mathbf{x-r}||) 
f(||\mathbf{x}||) d\mathbf{x}\nonumber\\
& &\hspace{-2cm} + \frac 1 2 \lambda   \left(1-\widetilde p^{(n)}\right)
\xi^{(n)}_{\widetilde \Psi,\widetilde \Psi} (r)
\int_{\mathbb R^2} 
\xi^{(n)}_{\widetilde \Psi,\widetilde \Phi}(||\mathbf{x}||)^{\frac 1 2}\xi^{(n)}_{\widetilde \Psi,\widetilde \Phi}(||\mathbf{x-r}||)^{\frac 1 2} f(||\mathbf{x}||) d\mathbf{x}. \nonumber
\\
\end{eqnarray}
In these equations, we have
\begin{eqnarray}
\label{eq:recursivec1}
\widetilde p^{(n)} & = & 1- \frac{\beta}{\widetilde  \lambda 2 \pi \int_{\mathbb R^+} \xi^{(n)}_{\widetilde \Psi,\widetilde \Phi} (r) f(r) r dr},
\end{eqnarray}
and
\begin{equation} 
\label{eq:recursivec2}
\xi^{(n)}_{\widetilde \Psi,\widetilde \Psi} (r) = \frac 1 {\left(1-\widetilde p^{(n)}\right)^2}
\left( c(r)- 
\left(\widetilde p^{(n)}\right)^2 \xi^{(n)}_{\widetilde \Phi,\widetilde \Phi} (r)
-2\widetilde p^{(n)}\left(1-\widetilde p^{(n)}\right) \xi^{(n)}_{\widetilde \Psi,\widetilde \Phi} (r) \right)
\end{equation}
with $c(r)$ defined in (\ref{eq:cder}).

\paragraph{Polynomial equation p-m2bi}

Under the conditions for the polynomial setting, we get 
\begin{eqnarray}
\label{eq:fantapm2bi}
\beta \widetilde p v & = & 2\alpha (1-\widetilde p)w  +\beta \widetilde p w\nonumber \\
\beta  w & = & \frac 1 2 \beta w +\frac 1 2 \beta z.
\end{eqnarray}
Since $w=z$, (\ref{eq:labonne-rep}) reads
\begin{eqnarray}
\label{eq:cons-b1}
c & = & w (1-\widetilde p)^2 + \widetilde p^2v + 2\widetilde p(1-\widetilde p)w .
\end{eqnarray}
The above system can be rewritten as
\begin{eqnarray}
(\alpha\widetilde \mu w-\beta) v & = & 2 \alpha w + (\alpha\widetilde \mu w-\beta) w\nonumber\\
c (\alpha\widetilde \mu w)^2 & = & w\beta^2  + (\alpha\widetilde \mu w-\beta)^2  v +
2 (\alpha\widetilde \mu w-\beta)\beta w.
\end{eqnarray}
By eliminating $v$, we get
\begin{eqnarray}
(\alpha\widetilde \mu w-\beta) (\alpha\widetilde \mu w-\beta +2\alpha ) =
c (\alpha\widetilde \mu )^2 w - 2 (\alpha\widetilde \mu w-\beta)\beta -\beta^2.  
\end{eqnarray}

\paragraph{Critical values}
When $p$ tends to 0, 
$w$ tends to $\frac\beta{\alpha \mu}$ and it follows from
(\ref{eq:cons-b1}) that the associated value of $\beta$ is 
$\beta_c=c\alpha\widetilde \mu$. 
The conclusions are then the same as above.

\subsection{Heuristic M$\infty$BI}
\label{hmbinftynomot}

\paragraph{Functional equation f-m$\infty$bi}
The associated functional equation reads
\begin{eqnarray}
\label{eq:inteqminftybi}
\widetilde p \xi_{\widetilde \Phi,\widetilde \Phi} (r) \beta
& =  &  (1-\widetilde p) \xi_{\widetilde \Psi,\widetilde \Phi} (r) f(r)
+  \lambda (1-\widetilde p) \widetilde p
\xi_{\widetilde \Psi,\widetilde \Phi} (r)
\nonumber\\
& & \hspace{-3cm}
\int_{\mathbb R^2} \int_0^1
\xi_{\widetilde \Psi,\widetilde \Phi}(||\mathbf{x}||)^{\eta} 
\xi_{\widetilde \Phi,\widetilde \Phi}(||\mathbf{x-r}||)^{2(1-\eta)} 
\xi_{\widetilde \Psi,\widetilde \Phi}(r)^{\eta-1}
f(||\mathbf{x}||) d\eta d\mathbf{x}\nonumber\\
\xi_{\widetilde \Psi,\widetilde \Phi} (r) \beta & = &
\widetilde \lambda   (1-\widetilde p) \xi_{\widetilde \Psi,\widetilde \Psi} (r)
\nonumber\\
& & \hspace{-3cm}
\int_{\mathbb R^2} \int_0^1
\xi_{\widetilde \Psi,\widetilde \Phi}(||\mathbf{x}||)^{\eta}
\xi_{\widetilde \Psi,\widetilde \Phi}(||\mathbf{x-r}||)^{\eta} f(||\mathbf{x}||)
\xi_{\widetilde \Psi,\widetilde \Psi}(r)^{1-2\eta} d\eta d\mathbf{x}.  \nonumber
\\
\end{eqnarray}
\paragraph{Polynomial equation p-m$\infty$bi}
The polynomial\footnote{Polynomial is an abuse of terminology here} equation reads
\begin{eqnarray}
\label{eq:fantastic-mbinfty}
\beta \widetilde p v & = & \alpha (1-\widetilde p) w+ \beta \widetilde p \frac {v^2} w  \int_0^1 w^{2\eta} v^{-2\eta} d\eta\nonumber\\
& = & \alpha  (1-\widetilde p)w +\beta \widetilde p  \frac 1 w \frac{w^2-v^2}{\log(w^2)-\log(v^2)} \nonumber \\ 
\beta \widetilde p w  &= & \widetilde p \beta \frac{z^2}{w} 
\int_0^1 w^{2\eta} z^{-2\eta} d\eta
= \beta \widetilde p \frac 1 w  \frac{w^2-z^2}{\log(w^2)-\log(z^2)} \nonumber\\
1-\widetilde p & = & \frac{\beta}{\alpha \widetilde \mu w}\nonumber\\
z(1-\widetilde p)^2 & = & c-\widetilde p^2v-2\widetilde p(1-\widetilde p)w.
\end{eqnarray}
We hence get the system
\begin{eqnarray}
\label{eq:poly-mbinfty}
v(\alpha \widetilde \mu w -\beta) & = & \alpha  w+ 
(\alpha \widetilde \mu w -\beta) \frac 1 {w} \frac{w^2-v^2}{\log(w^2)-\log(v^2)}\nonumber \\ 
w^2  &= &
\frac{w^2-z^2}{\log(w^2)-\log(z^2)}\nonumber \\
z & = &\frac 1{\beta^2}\left( w^2(c\alpha^2\widetilde \mu^2 -2\alpha \beta\widetilde \mu)-(\alpha \widetilde \mu w -\beta)^2v +2\beta^2w\right).
\end{eqnarray}
It is easy to show that the second equation in 
(\ref{eq:poly-mbinfty}) implies that $z=w$ as in the earlier cases.
So the system p-mb$\infty$ actually reads
\begin{eqnarray}
\label{eq:poly-mbinfty-short}
vw(\alpha \widetilde \mu w -\beta) & = & \alpha  w^2+ 
(\alpha \widetilde \mu w -\beta) \frac{w^2-v^2}{\log(w^2)-\log(v^2)}\nonumber \\ 
0 & = &w^2(c\alpha^2\widetilde \mu^2 -2\alpha \beta\widetilde \mu)-(\alpha \widetilde \mu w -\beta)^2v +\beta^2w.
\end{eqnarray}

\paragraph{Critical values}
Let us look at what happens when
$p$ tends to 0 in (\ref{eq:fantastic-mbinfty}). It follows from the third  
equation that $w$ tends to $\frac\beta{\alpha \widetilde \mu}$,
from the last one that $z$ tends to $c$,
and from the first one that
$\widetilde p^2v(\widetilde p)$ tends to $c-\frac{\beta}{\alpha \widetilde \mu}$.
Multiplying the first equation by $p$ and letting $p$ to 0, we get that necessarily
$c= \frac{\beta}{\alpha \widetilde \mu}$. So we have again
$\beta_c=\alpha\widetilde \mu$.

\subsection{Heuristic M$\infty$BG1}
\label{hmbdinftynomot}

\paragraph{Functional equation f-m$\infty$g1}
The associated functional equation reads
\begin{eqnarray}
\label{eq:inteq-mfbdinfty-5}
\widetilde p \xi_{\widetilde \Phi,\widetilde \Phi} (r) \beta
& =  &  (1-\widetilde p) \xi_{\widetilde \Psi,\widetilde \Phi} (r) f(r)
+  \lambda (1-\widetilde p) \widetilde p
\xi_{\widetilde \Psi,\widetilde \Phi} (r)
\nonumber\\
& & \hspace{-3cm}
\int_{\mathbb R^2} \int_0^1
\xi_{\widetilde \Psi,\widetilde \Phi}(||\mathbf{x}||)^{1 -\frac \eta 2} 
\xi_{\widetilde \Phi,\widetilde \Phi}(||\mathbf{x-r}||)^{\frac 1 2} 
\xi_{\widetilde \Psi,\widetilde \Phi}(r)^{\frac \eta 2-\frac 1 2}
f(||\mathbf{x}||) d\eta d\mathbf{x}\nonumber\\
\xi_{\widetilde \Psi,\widetilde \Phi} (r) \beta & = &
\widetilde \lambda   (1-\widetilde p) \xi_{\widetilde \Psi,\widetilde \Psi} (r)
\nonumber\\
& & \hspace{-3cm}
\int_{\mathbb R^2} \int_0^1
\xi_{\widetilde \Psi,\widetilde \Phi}(||\mathbf{x}||)^{\frac 1 2 +\frac \eta 2}
\xi_{\widetilde \Psi,\widetilde \Phi}(||\mathbf{x-r}||)^{1-\frac \eta 2} f(||\mathbf{x}||)
\xi_{\widetilde \Psi,\widetilde \Psi}(r)^{-\frac 1 2} d\eta d\mathbf{x}.  \nonumber
\\
\end{eqnarray}
\paragraph{Polynomial equation p-m$\infty$bg1}
The polynomial\footnote{Polynomial is an abuse of terminology here} equation reads
\begin{eqnarray}
\label{eq:fantastic-mbdinfty}
\beta \widetilde p v & = & \alpha (1-\widetilde p) w+ \beta \widetilde p v^{\frac 1 2} w^{\frac 1 2}\nonumber\\
\beta \widetilde p w  &= & 
\beta \widetilde p z^{\frac 1 2} w^{\frac 1 2}  \nonumber\\
1-\widetilde p & = & \frac{\beta}{\alpha \widetilde \mu w}\nonumber\\
z(1-\widetilde p)^2 & = & c-\widetilde p^2v-2\widetilde p(1-\widetilde p)w.
\end{eqnarray}
We hence get that $z=w$. So
\begin{eqnarray}
\label{eq:polsysb1}
\beta \widetilde p v  & = &  \alpha (1-\widetilde p) w+ \beta \widetilde p  v^{\frac 1 2} w^{\frac 1 2} \nonumber\\ 
w (1-\widetilde p)^2 &= & c- \widetilde p^2 v -2\widetilde p(1-\widetilde p) w.
\end{eqnarray}

\paragraph{Critical values}
Let us look at what happens when
$p$ tends to 0 in (\ref{eq:fantastic-mbdinfty}). It follows from the third  
equation that $w$ tends to $\frac\beta{\alpha \widetilde \mu}$,
from the last one that $z$ tends to $c$. So
$c= \frac{\beta}{\alpha \widetilde \mu}$. So we have again
$\beta_c=\alpha\widetilde \mu$.

\section{Mean Time between Two Infections}
\label{append:littles_law}
Let $\nu$ be the mean time between infections of an point. Now consider the queue of healthy points. Then Little's Law yields:
\begin{align}
\lambda (1-p) &= \lambda p \beta \nu \nonumber \\
\implies \nu &= \frac{1-p}{p\beta}
\end{align}

\end{document}